\documentclass[11pt]{amsart}
\usepackage{amsmath} 
\usepackage{amssymb} 
\usepackage{amsthm}
\usepackage{amsfonts}            
\usepackage{mathrsfs}
\usepackage[all]{xy}
\usepackage{mathtools}
\usepackage{mathabx}
\usepackage{colonequals}
\usepackage[colorlinks,
            linkcolor=black,
            anchorcolor=black,
            citecolor=black
            ]{hyperref}

\usepackage{geometry}
\usepackage{setspace}

\title{Arithmetic Siegel-Weil Formula on $\mathcal{X}_{0}(N)$}
\author{Baiqing Zhu}
\date{April 2023}
\begin{document}
\theoremstyle{definition}        
\newtheorem{definition}{Definition}[subsection] 
\newtheorem{example}[definition]{Example}
\newtheorem{remark}[definition]{Remark}

\theoremstyle{plain}
\newtheorem{theorem}[definition]{Theorem}
\newtheorem{lemma}[definition]{Lemma}         
\newtheorem{proposition}[definition]{Proposition}
\newtheorem{corollary}[definition]{Corollary}

\maketitle
\begin{abstract}
    We establish the arithmetic Siegel-Weil formula on the modular curve $\mathcal{X}_{0}(N)$ for arbitrary level $N$, i.e., we relate the arithmetic degrees of special cycles on $\mathcal{X}_{0}(N)$ to the derivatives of Fourier coefficients of a genus 2 Eisenstein series. We prove this formula by a precise identity between the local arithmetic intersection numbers on the Rapoport-Zink space associated to $\mathcal{X}_{0}(N)$ and the derivatives of local representation densities of quadratic forms. When $N$ is odd and square-free, this gives a different proof of the main results in the work of Sankaran, Shi and Yang \cite{SSY22}. This local identity is proved by relating it to an identity in one dimension higher, but at hyperspecial level.
\end{abstract}
\pagenumbering{roman}
\tableofcontents
\newpage
\pagenumbering{arabic}

\section{Introduction}
\subsection{Background}
The classical Siegel-Weil formula relates certain Siegel Eisenstein series to the arithmetic of quadratic forms, namely it expresses special values of these series as theta functions -- generating series of representation numbers of quadratic forms. Kudla initiated an influential program to establish the arithmetic Siegel-Weil formula relating certain Siegel Eisenstein series to objects in arithmetic geometry. 
\par 
In this article, we study the case of modular curves. Let $N$ be a positive integer, the classical modular curve $\mathcal{Y}_{0}(N)_{\mathbb{C}}$ over $\mathbb{C}$ is defined as the following smooth $1$-dimensional complex curve,
\begin{equation*}
    \mathcal{Y}_{0}(N)_{\mathbb{C}}\coloneqq\textup{GL}_{2}(\mathbb{Q})\backslash \mathbb{H}_{1}^{\pm}\times\textup{GL}_{2}(\mathbb{A}_{f})/\Gamma_{0}(N)(\hat{\mathbb{Z}})\simeq \Gamma_{0}(N)\backslash\mathbb{H}_{1}^{\pm},
\end{equation*}
where $\mathbb{H}_{1}^{\pm}=\mathbb{C}\backslash\mathbb{R}$ and $\mathbb{H}_{1}^{+}=\{z=x+iy\in\mathbb{C}:x\in\mathbb{R}, y\in\mathbb{R}_{>0}\}$ is the upper half plane. The group $\Gamma_{0}(N)(\hat{\mathbb{Z}})$ is the following open compact subgroup of $\textup{GL}_{2}(\mathbb{A}_{f})$,
\begin{equation*}
   \Gamma_{0}(N)(\hat{\mathbb{Z}}) = \left\{x=\begin{pmatrix}
    a & b\\
    Nc & d
    \end{pmatrix}\in \textup{GL}_{2}(\hat{\mathbb{Z}})\,\,:\,\,a,b,c,d\in\hat{\mathbb{Z}}\right\},
\end{equation*}
and $\Gamma_{0}(N)=\Gamma_{0}(N)(\hat{\mathbb{Z}})\bigcap \textup{GL}_{2}(\mathbb{Z})$.
\par
The smooth curve $\mathcal{Y}_{0}(N)_{\mathbb{C}}$ is not proper, its compactification $\mathcal{X}_{0}(N)_{\mathbb{C}} \coloneqq\mathcal{Y}_{0}(N)_{\mathbb{C}}\cup\{\textup{cusps}\}$ is a smooth projective curve over $\mathbb{C}$. It is a classical fact that the curve $\mathcal{Y}_{0}(N)_{\mathbb{C}}$ parameterizes cyclic isogenies between elliptic curves over $\mathbb{C}$, here an isogeny $\pi:E\rightarrow E^{\prime}$ between two elliptic curves over $\mathbb{C}$ is called cyclic if $\textup{ker}(\pi)$ is a cyclic group.
\par
Katz and Mazur \cite{KM85} extends the concept of cyclic isogeny to arbitrary base scheme, an isogeny $\pi:E\rightarrow E^{\prime}$ between two elliptic curves is called cyclic if $\textup{ker}(\pi)$ is a cyclic group scheme (cf. Definition \ref{cycdef}). They also defined the $\Gamma_{0}(N)$-level structures on elliptic curves. In this article, we will mainly work on a 2-dimensional regular flat Deligne-Mumford stack $\mathcal{X}_{0}(N)$, defined by Cesnavicius in \cite{Ces17}, which is the moduli stack of generalized elliptic curves with $\Gamma_{0}(N)$-level structures and whose fiber over $\mathbb{C}$ is $\mathcal{X}_{0}(N)_{\mathbb{C}}$. We define the (arithmetic) special cycles on $\mathcal{X}_{0}(N)$ and study their intersection numbers. Finally, we prove that these intersection numbers are identified with the derivatives of Fourier coefficients of certain Siegel Eisenstein series of genus 2.
\par
When $N$ is an odd, square-free positive integer, the relation has already been obtained in the work of Sankaran, Shi and Yang \cite[Theorem 2.14]{SSY22} by computing both sides explicitly based on the previous works of Yang \cite{Yn98} and Kudla, Rapoport and Yang\cite{KRY06}. In this article, we use a different method and work with arbitrary level $N$. We introduce a formal scheme $\mathcal{N}_{0}(N)$ which is the Rapoport-Zink space associated to $\mathcal{X}_{0}(N)$. Via formal uniformization of the supersingular locus of the stack $\mathcal{X}_{0}(N)$ and its special cycles, we reduce the identity, which relates intersection numbers on $\mathcal{X}_{0}(N)$ and derivatives of Fourier coefficients of Eisenstein series, to a local identity between local arithmetic intersection numbers on $\mathcal{N}_{0}(N)$ and derivatives of local densities of quadratic forms. Now the key observation is that, both sides of the local identity, regardless of the level $N$, can be related to another intersection problem on Rapoport-Zink space of 1 dimension higher, but in a hyperspecial level, while the computation of the latter has been done in the works of Gross and Keating \cite[Proposition 5.4]{GK93}, Wedhorn \cite[$\S$2.16]{Wed07} and Rapoport \cite[Theorem 1.1]{Rap07} (see also the work of Li and Zhang \cite[Theorem 1.2.1]{LZ22}).

\subsection{Summary of main results}
\subsubsection{Arithmetic Siegel-Weil formula on $\mathcal{X}_{0}(N)$}
Let $\Delta(N)$ be the following rank 3 quadratic lattice over $\mathbb{Z}$,
\begin{equation}
    \Delta(N) = \left\{x=\begin{pmatrix}
    -Na & b\\
    c & a
    \end{pmatrix}:\,\, a,b,c\in\mathbb{Z}\right\}
\end{equation}
equipped with the quadratic form $x\mapsto \textup{det}(x)$. 
\par
We use $v$ to denote a place of $\mathbb{Q}$. For every finite place $v$, let $\delta_{v}(N)=\Delta(N)\otimes_{\mathbb{Z}}\mathbb{Z}_{v}$ be a rank 3 quadratic lattice over $\mathbb{Z}_{v}$. Let $\mathbb{A}$ be the ring of ad$\grave{\textup{e}}$les over $\mathbb{Q}$. Let $\mathbb{V}=\{\mathbb{V}_{v}\}$ be the incoherent collection of quadratic spaces of $\mathbb{A}$ of rank 3 nearby $\Delta(N)$ at $\infty$, i.e.,
\begin{equation}
    \mathbb{V}_{v}=\delta_{v}(N)\otimes\mathbb{Q}_{v}\,\,\textup{if $v<\infty$, and $\mathbb{V}_{\infty}$ is positive definite.}
    \label{incoherent}
\end{equation}
\par
For a finite place $v$. Let $\mathbb{V}_{f}\coloneqq\mathbb{V}\otimes\mathbb{A}_{f}$ (resp. $\mathbb{V}_{f}^{v}\coloneqq\mathbb{V}\otimes\mathbb{A}_{f}^{v}$) be the quadratic space of rank 3 over $\mathbb{A}_{f}$ (resp. $\mathbb{A}_{f}^{v}$). Let $\mathscr{S}(\mathbb{V}^{2})$ (resp. $\mathscr{S}(\mathbb{V}_{f}^{2})$, $\mathscr{S}((\mathbb{V}_{f}^{v})^{2})$) be the space of Schwartz functions on $\mathbb{V}^{2}$ (resp. $\mathbb{V}_{f}^{2}$, $(\mathbb{V}_{f}^{v})^{2}$). Associated to $\tilde{\boldsymbol\varphi}=\boldsymbol\varphi\,\otimes\,\boldsymbol\varphi_{\infty}\in\mathscr{S}(\mathbb{V}^{2})$, where $\boldsymbol\varphi_{\infty}$ is the Gaussian function on $\mathbb{V}_{\infty}^{2}$ and $\boldsymbol\varphi\in\mathscr{S}(\mathbb{V}_{f}^{2})$, there is a classical incoherent Eisenstein series $E(\mathsf{z},s,\boldsymbol\varphi)$ (cf. \S\ref{inco}) on the Siegel upper half plane,
\begin{equation*}
    \mathbb{H}_{2}=\{\mathsf{z}=\mathsf{x}+i\mathsf{y}\,\,\vert\,\,\mathsf{x}\in\textup{Sym}_{2}(\mathbb{R}),\mathsf{y}\in\textup{Sym}_{2}(\mathbb{R})_{>0}\}.
\end{equation*}
This is essentially the Siegel Eisenstein series associated to a standard Siegel-Weil section of the degenerate principal series. The Eisenstein series here has a meromorphic continuation and a functional equation relating $s \leftrightarrow -s$.  The central value $E(\mathsf{z},s,\boldsymbol\varphi)=0$ by the incoherence. We thus consider its central derivative
\begin{equation*}
    \partial\textup{Eis}(\mathsf{z},\boldsymbol\varphi)\coloneqq\frac{\textup{d}}{\textup{d}s}\bigg\vert_{s=0}E(\mathsf{z},s,\boldsymbol\varphi).
\end{equation*}
Associated to the standard additive character $\psi:\mathbb{A}/\mathbb{Q}\rightarrow\mathbb{C}^{\times}$, it has a decomposition into the central derivatives of the Fourier coefficients
\begin{equation*}
    \partial\textup{Eis}(\mathsf{z},\boldsymbol\varphi) = \sum\limits_{T\in\textup{Sym}_{2}(\mathbb{Q})}\partial\textup{Eis}_{T}(\mathsf{z},\boldsymbol\varphi)
\end{equation*}
\par
On the geometric side, there is a regular integral model of the modular curve $\mathcal{Y}_{0}(N)_{\mathbb{C}}$ over $\mathbb{Z}$ defined by Katz and Mazur: for any scheme $S$, the groupoid $\mathcal{Y}_{0}(N)(S)$ consists of objects $(E\stackrel{\pi}\longrightarrow E^{\prime})$ where $E$, $E^{\prime}$ are elliptic curves over $S$ and $\pi$ is a cyclic isogeny such that $\pi^{\vee}\circ\pi=N$. They proved that $\mathcal{Y}_{0}(N)$ is 2-dimensional regular flat Deligne-Mumford stack (cf. \cite[Theorem 5.1.1]{KM85}), but $\mathcal{Y}_{0}(N)$ is not proper. There is a moduli stack $\mathcal{X}_{0}(N)$ defined in \cite{Ces17} which serves as a ``compactification'' of $\mathcal{Y}_{0}(N)$. It is a proper regular flat 2-dimensional Deligne-Mumford stack which contains $\mathcal{Y}_{0}(N)$ as an open substack, so we can consider the arithmetic intersection theory on $\mathcal{X}_{0}(N)$ following the lines in \cite{Gil09}.
\par
The key concept is that of a special cycle. A typical special cycle is of the following form: $\mathcal{Z}(T,\boldsymbol\varphi)$ where $T$ is a $2\times2$ symmetric matrix and $\boldsymbol\varphi\in\mathscr{S}(\mathbb{V}_{f}^{2})$ is the characteristic function of some open compact subset of $\mathbb{V}_{f}^{2}$. It is a Deligne-Mumford stack finite unramified over $\mathcal{X}_{0}(N)$. For an object $(E\stackrel{\pi}\longrightarrow E^{\prime})$ of $\mathcal{Y}_{0}(N)(S)$, the special cycle $\mathcal{Z}(T,\boldsymbol\varphi)$ parameterizes pairs of isogenies between $E$ and $E^{\prime}$ with inner product matrix $T$ and orthogonal to the cyclic isogeny $\pi$, along with some level structures given by the Schwartz function $\boldsymbol\varphi$ (cf. Definition \ref{spe}). For every nonsingular $T\in \textup{Sym}_{2}(\mathbb{Q})$ and prime number $l$, we say $T$ is represented by $\Delta(N)\otimes\mathbb{Q}_{l}$ if there exist two vectors $x_{1},x_{2}\in\Delta(N)\otimes\mathbb{Q}_{l}$ such that $T=\frac{1}{2}((x_{i},x_{j}))$. Define the following diffenence set,
\begin{align*}
    \text{Diff}(T,\Delta(N)) & = \{l \text{ is a finite prime}:T\text{ is not represented by $\Delta(N)\otimes\mathbb{Q}_{l}$} \}
\end{align*}
\par
Let $\widehat{\textup{CH}}^{2}_{\mathbb{C}}(\mathcal{X}_{0}(N))$ be the codimension 2 arithmetic Chow group with complex coefficients of the regular flat Deligne-Mumford stack $\mathcal{X}_{0}(N)$. We also construct arithmetic special cycles on the stack $\mathcal{X}_{0}(N)$, they are elements of the form $\widehat{\mathcal{Z}}(T,\mathsf{y},\boldsymbol\varphi)\in \widehat{\textup{CH}}^{2}_{\mathbb{C}}(\mathcal{X}_{0}(N))$ where $T\in\textup{Sym}_{2}(\mathbb{Q})$ is nonsingular, $\mathsf{y}\in\textup{Sym}_{2}(\mathbb{R})$ is positive definite and $\boldsymbol{\varphi}\in\mathscr{S}(\mathbb{V}_{f}^{2})$ is a Schwartz function. Let $\boldsymbol{\varphi}\in\mathscr{S}(\mathbb{V}_{f}^{2})$ be a $T$-admissible Schwartz function, which roughly means $\boldsymbol\varphi$ is invariant under the action of $\Gamma_{0}(N)(\hat{\mathbb{Z}})$ and for every $p\in\text{Diff}(T,\Delta(N))$, $\boldsymbol\varphi=\boldsymbol\varphi^{p}\otimes\boldsymbol\varphi_{p}$ where $\boldsymbol\varphi^{p}\in\mathscr{S}((\mathbb{V}_{f}^{p})^{2})$ and $\boldsymbol\varphi_{p}=c\cdot\boldsymbol{1}_{\delta_{p}(N)^{2}}\in\mathscr{S}(\mathbb{V}_{p}^{2})$ for some $c\in\mathbb{C}$. Our main goal is to relate $\widehat{\textup{deg}}(\widehat{\mathcal{Z}}(T,\mathsf{y},\boldsymbol\varphi))$ to derivatives of the Fourier coefficients of a genus 2 Siegel Eisenstein series, where $\widehat{\textup{deg}}: \widehat{\textup{CH}}^{2}_{\mathbb{C}}(\mathcal{X}_{0}(N))\rightarrow\mathbb{C}$ is the arithmetic degree map (cf. (\ref{degreemap})).
\begin{theorem}
Let $N$ be a positive integer. Let $T\in \textup{Sym}_{2}(\mathbb{Q})$ be a nonsingular symmetric matrix. Let $\boldsymbol{\varphi}\in\mathscr{S}(\mathbb{V}_{f}^{2})$ be a $T$-admissible Schwartz function, then
\begin{equation*}
     \widehat{\textup{deg}}(\widehat{\mathcal{Z}}(T,\mathsf{y},\boldsymbol{\varphi}))q^{T}= \frac{\psi(N)}{24}\cdot\partial\textup{Eis}_{T}(\mathsf{z},\boldsymbol\varphi),
\end{equation*}
for any $\mathsf{z}=\mathsf{x}+i\mathsf{y}\in\mathbb{H}_{2}$. Here $\psi(N)= N\cdot\prod\limits_{l\vert N}(1+l^{-1})$, $q^{T}=e^{2\pi i \,\textup{tr}\,(T\mathsf{z})}$.
\label{MainGlobal}
\end{theorem}

\subsubsection{The local arithmetic Siegel-Weil formula with level $N$} 
Fix a prime number $p$. Let $\mathbb{F}$ be the algebraic closure of $\mathbb{F}_{p}$. Let $W$ be the integer ring of the completion of the maximal unramified extension of $\mathbb{Q}_{p}$. 
\par
On the geometry side, let $\mathbb{X}$ be a $p$-divisible group over $\mathbb{F}$ of dimension 1 and height 2. Let $\mathbb{B}$ be the unique division quaternion algebra over $\mathbb{Q}_{p}$, then $\textup{End}^{0}(\mathbb{X})\simeq\mathbb{B}$ as quadratic spaces. The Rapoport-Zink space associated to $\mathcal{X}_{0}(N)$ is the following deformation space $\mathcal{N}_{0}(N)$: for a $W$-scheme $S$ where $p$ is locally nilpotent, and an element $x_{0}\in\mathbb{B}$ such that $x_{0}^{\vee}\circ x_{0}=N$, the set $\mathcal{N}_{0}(N)(S)$ consists of elements $(X\stackrel{\pi}\longrightarrow X^{\prime})$ where $X$, $X^{\prime}$ are deformations over $S$ of $\mathbb{X}$ with certain restrictions on polarizations (cf. $\S$\ref{locdeform}), the morphism $\pi$ is a cyclic isogeny deforming $x_{0}$ and $\pi^{\vee}\circ\pi=N$.
\par
Let $\mathbb{W}=\{x_{0}\}^{\bot}\subset\mathbb{B}$ be the subspace of quasi-isogenies which are orthogonal to $x_{0}$. For any $x\in\mathbb{W}$, there is a closed formal subscheme $\mathcal{Z}(x)$ of $\mathcal{N}_{0}(N)$ over which the quasi-isogeny $x$ lifts to an isogeny. This is an example of special cycle (cf. Definition \ref{esp2}) on $\mathcal{N}_{0}(N)$. For a rank $2$ lattice $M\subset\mathbb{W}$, we choose a $\mathbb{Z}_{p}$-basis $\{x_{1},x_{2}\}$ of $M$, then define the local arithmetic intersection number of $M$ on $\mathcal{N}_{0}(N)$ to be
\begin{equation*}
    \textup{Int}_{\mathcal{N}_{0}(N)}(M) = \chi(\mathcal{N}_{0}(N),\mathcal{O}_{\mathcal{Z}(x_{1})}\otimes^{\mathbb{L}}_{\mathcal{O}_{\mathcal{N}_{0}(N)}}\mathcal{O}_{\mathcal{Z}(x_{2})}).
\end{equation*}
This number is independent of the choice of the basis $\{x_{1},x_{2}\}$ of $M$.
\par
On the analytic side, for any two integral quadratic $\mathbb{Z}_{p}$-lattices $L$ and $M$. Let $\textup{Rep}_{M,L}$ be the scheme of integral representations, an $\mathbb{Z}_{p}$-scheme such that for any $\mathbb{Z}_{p}$-algebra $R$, $\textup{Rep}_{M,L}(R)=\textup{QHom}(L\otimes_{\mathbb{Z}_{p}}R, M\otimes_{\mathbb{Z}_{p}}R)$, where $\textup{QHom}$ denotes the set of quadratic module homomorphisms. The local density of integral representations is defined to be
\begin{equation*}
    \textup{Den}(M,L)=\lim\limits_{d\rightarrow\infty}\frac{\#\textup{Rep}_{M,L}(\mathbb{Z}_{p}/p^{d})}{p^{d\cdot \textup{dim}(\textup{Rep}_{M,L})_{\mathbb{Q}_{p}}}}.
\end{equation*}
\par
Let $H_{2}^{+}=\mathbb{Z}_{p}^{2}$ be the rank 2 quadratic $\mathbb{Z}_{p}$-lattice equipped with the quadratic form $q_{H_{2}^{+}}(x,y)=xy$. For any $k\geq0$, let $H_{2k}^{+}\coloneqq (H_{2}^{+})^{\obot n}$ be a rank $2k$ quadratic $\mathbb{Z}_{p}$-lattice. For any $\mathbb{Z}_{p}$-lattice $M\subset\mathbb{W}$ of rank $2$, define the local density of $M$ with level $N$ to be the polynomial $\textup{Den}^{+}_{\Delta(N)}(X, M)$ such that for all $k\geq0$, 
\begin{equation}
    \textup{Den}^{+}_{\Delta(N)}(X,M)\,\big\vert_{X=p^{-k}} =\left\{\begin{aligned}
    &\frac{\textup{Den}(\delta_{p}(N)\obot H_{2k}^{+},M)}{\textup{Nor}^{+}(p^{-k},1)}, & \textup{when $p\,\vert\, N$};\\
    &\frac{\textup{Den}(\delta_{p}(N)\obot H_{2k}^{+},M)}{\textup{Nor}^{(N,p)_{p}}(p^{-k},2)}, & \textup{when $p\nmid N$}.
    \end{aligned}\right.
    \label{pcase}
\end{equation}
where $(\cdot,\cdot)_{p}$ is the Hilbert symbol at $p$, the polynomials $\textup{Nor}^{\varepsilon}(X,n)$ are normalizing factors defined in Definition \ref{normal1}. Then $\textup{Den}^{+}_{\Delta(N)}(1,M)=0$ since $M$ can't be isometrically embedded into the quadratic lattice $\delta_{p}(N)$, we define the derived local density of $M$ with level $N$ to be
\begin{equation*}
    \partial\textup{Den}^{+}_{\Delta(N)}(M) = -\frac{\textup{d}}{\textup{d}X}\bigg|_{X=1}\textup{Den}^{+}_{\Delta(N)}(X,M).
\end{equation*}
\par
The local arithmetic Siegel-Weil formula with level $N$ is an exact identity between the two integers just defined.
\begin{theorem}
    Let $M\subset\mathbb{W}$ be a $\mathbb{Z}_{p}$-lattice of rank 2. Then
\begin{equation*}
    \textup{Int}_{\mathcal{N}_{0}(N)}(M) = \partial\textup{Den}^{+}_{\Delta(N)}(M).
\end{equation*}
\label{MainLocal}
\end{theorem}
We refer to $\textup{Int}_{\mathcal{N}_{0}(N)}(M)$ as the geometric side of the identity (related to the geometry of Rapoport–Zink spaces and Shimura varieties) and $\partial\textup{Den}^{+}_{\Delta(N)}(M)$ the analytic side (related to the derivatives of Eisenstein series and $L$-functions).

\subsubsection{Formal uniformization}
For any prime $p$, let $\mathcal{X}_{0}(N)_{\mathbb{F}_{p}}^{ss}$ be the supersingular locus of the stack $\mathcal{X}_{0}(N)$, i.e., those $\mathbb{F}$-points of $\mathcal{X}_{0}(N)$ which is isogenous to a supersingular elliptic curve. Let $B$ be the unique quaternion algebra which ramifies exactly at $p$ and $\infty$. Let $\hat{\mathcal{X}}_{0}(N)/_{(\mathcal{X}_{0}(N)_{\mathbb{F}_{p}}^{ss})}$ be the completion of the stack $\mathcal{X}_{0}(N)$ along the closed substack $\mathcal{X}_{0}(N)_{\mathbb{F}_{p}}^{ss}$. Let $\Gamma_{0}(N)(\hat{\mathbb{Z}}^{p})$ be the group $\prod\limits_{v\neq\infty,p}\Gamma_{0}(N)(\mathbb{Z}_{v})$. We have the following formal uniformization theorem of the stack $\mathcal{X}_{0}(N)$.
\begin{proposition}
There is an isomorphism of formal stacks over $W$,
\begin{equation*}
    \hat{\mathcal{X}}_{0}(N)/_{(\mathcal{X}_{0}(N)_{\mathbb{F}_{p}}^{ss})}\underset{\sim}{\stackrel{\Theta_{\mathcal{X}_{0}(N)}}\longrightarrow} B^{\times}(\mathbb{Q})_{0}\backslash[\mathcal{N}_{0}(N)\times  \textup{GL}_{2}(\mathbb{A}_{f}^{p})/\Gamma_{0}(N)(\hat{\mathbb{Z}}^{p})]
\end{equation*}
where $B^{\times}(\mathbb{Q})_{0}$ is the subgroup of $B^{\times}(\mathbb{Q})$ consisting of elements whose norm has $p$-adic valuation 0.
\end{proposition}
The proposition was previously known only in the case that $N$ is odd and square-free (see the work of Kim \cite[Theorem 4.7]{Kim18} for the case that $p\nmid N$, and the work of Oki \cite[Theorem 6.1]{Oki20} for the case that $p\,\vert\, N$). As a corollary, let $\hat{\mathcal{Z}}^{ss}(T,\boldsymbol{\varphi})$ be the completion of $\mathcal{Z}(T,\boldsymbol{\varphi})$ along its supersingular locus $\mathcal{Z}^{ss}(T,\boldsymbol{\varphi})\coloneqq \mathcal{Z}(T,\boldsymbol{\varphi})\times_{\mathcal{X}_{0}(N)}\mathcal{X}_{0}(N)_{\mathbb{F}_{p}}^{ss}$. Let $\Delta(N)^{(p)}$ be the unique quadratic space over $\mathbb{Q}$ (up to isometry) such that: 1. It is positive definite at $\infty$; 2. For finite prime $l\neq p$, $\Delta(N)^{(p)}\otimes\mathbb{Q}_{l}$ is isometric to $\delta_{l}(N)\otimes\mathbb{Q}_{l}$; 3. $\Delta(N)^{(p)}\otimes\mathbb{Q}_{p}$ is isometric to $\mathbb{W}$. For a pair of vectors $\boldsymbol{x}=(x_{1},x_{2})\in (\Delta(N)^{(p)})^{2}$, let $T(\boldsymbol{x})=(\frac{1}{2}(x_{i},x_{j}))$ be the inner product matrix. We have the following formal uniformization theorem of the special cycle $\mathcal{Z}(T,\boldsymbol{\varphi})$.
\begin{corollary}
Let $T\in\textup{Sym}_{2}(\mathbb{Q})$ be a nonsingular symmetric matrix, and $\textup{Diff}(T,\Delta(N))=\{p\}$. Let $\boldsymbol{\varphi}\in\mathscr{S}(\mathbb{V}_{f}^{2})$ be a $T$-admissible Schwartz function. Let $K_{0}^{\prime}(\hat{\mathcal{X}}_{0}(N)/_{(\mathcal{X}_{0}(N)_{\mathbb{F}_{p}}^{ss})})$ be the Grothendieck group of coherent sheaves of $\mathcal{O}_{\hat{\mathcal{X}}_{0}(N)/_{(\mathcal{X}_{0}(N)_{\mathbb{F}_{p}}^{ss})}}$-modules. Then we have the following identity in $K_{0}^{\prime}(\hat{\mathcal{X}}_{0}(N)/_{(\mathcal{X}_{0}(N)_{\mathbb{F}_{p}}^{ss})})$,
\begin{equation*}
    \hat{\mathcal{Z}}^{ss}(T,\boldsymbol{\varphi}) = \sum\limits_{\substack{\boldsymbol{x}\in B^{\times}(\mathbb{Q})_{0}\backslash (\Delta(N)^{(p)})^{2}\\ T(\boldsymbol{x})= T}}\sum\limits_{g\in B^{\times}_{\boldsymbol{x}}(\mathbb{Q})_{0}\backslash \textup{GL}_{2}(\mathbb{A}_{f}^{p})/\Gamma_{0}(N)(\hat{\mathbb{Z}}^{p})}\boldsymbol\varphi(g^{-1}\boldsymbol{x})\cdot \Theta_{\mathcal{X}_{0}(N)}^{-1}(\mathcal{Z}(\boldsymbol{x}),g),
\end{equation*}
where $B^{\times}_{\boldsymbol{x}}\subset B^{\times}$ is the stabilizer of $\boldsymbol{x}\in(\Delta(N)^{(p)})^{2}$.
\label{uniZ}
\end{corollary}

\subsection{Strategy of the proof of main results}
\subsubsection{Difference formula at the geometric side}
Let $\mathcal{N}$ be the following deformation functor: for a $W$-scheme $S$ where $p$ is locally nilpotent, the set $\mathcal{N}(S)$ consists of elements $(X, X^{\prime})$ where both $X$ and $X^{\prime}$ are deformations over $S$ of $\mathbb{X}$ with certain restrictions on polarizations (cf. $\S$\ref{locdeform}). For a nonzero integral element $x\in\mathbb{B}$, i.e., $0\leq\nu_{p}(x^{\vee}\circ x)<\infty$, there is a closed formal subscheme $\mathcal{Z}^{\sharp}(x)$ of $\mathcal{N}$ over which the quasi-isogeny $x$ lifts to an isogeny. This is an example of special cycle (cf. Definition \ref{specialcycle}) on $\mathcal{N}$. \par
For a rank $3$ lattice $L\subset\mathbb{B}$, we choose a $\mathbb{Z}_{p}$-basis $\{x_{1},x_{2},x_{3}\}$ of $L$, then define the local arithmetic intersection number of $L$ on $\mathcal{N}$ to be
\begin{equation*}
    \textup{Int}^{\sharp}(L) = \chi(\mathcal{N},\mathcal{O}_{\mathcal{Z}^{\sharp}(x_{1})}\otimes^{\mathbb{L}}_{\mathcal{O}_{\mathcal{N}}}\mathcal{O}_{\mathcal{Z}^{\sharp}(x_{2})}\otimes^{\mathbb{L}}_{\mathcal{O}_{\mathcal{N}}}\mathcal{O}_{\mathcal{Z}^{\sharp}(x_{3})}).
\end{equation*}
This number is independent of the choice of the basis $\{x_{1},x_{2},x_{3}\}$ of the lattice $L$. 
\par
The special cycle $\mathcal{Z}^{\sharp}(x)$ is cut out by a single element $f_{x}\in\mathfrak{m}=(p,t_{1},t_{2})\subset W[[t_{1},t_{2}]]$, and when $\nu_{p}(x^{\vee}\circ x)\geq2$, $f_{p^{-1}x}\vert f_{x}$. We define $d_{x}=f_{x}/f_{p^{-1}x} \in W[[t_{1},t_{2}]]$ when $\nu_{p}(x^{\vee}\circ x)\geq2$, $d_{x}=f_{x}$ when $\nu_{p}(x^{\vee}\circ x)=0$ or $1$. The following divisor
\begin{equation*}
    \mathcal{D}(x)\coloneqq\textup{Spf}\,W[[t_{1},t_{2}]]/d_{x},
\end{equation*}
is called the difference divisor associated to $x$ (cf. Definition \ref{diff}), which is originally introduced by Terstiege in \cite{Ter11}.
\par
Fix $x_{0}\in\mathbb{B}$ such that $x_{0}^{\vee}\circ x_{0}=N$, recall that we have defined the deformation function $\mathcal{N}_{0}(N)$. In Theorem \ref{keydecomgeo}, we prove that $\mathcal{N}_{0}(N)$ is identified with the difference divisor $\mathcal{D}(x_{0})$, i.e., there is an isomorphism of formal schemes,
\begin{equation*}
    \mathcal{D}(x_{0})\stackrel{\sim}\longrightarrow \mathcal{N}_{0}(N).
\end{equation*}
Let $x_{0}^{\textup{univ}}:X^{\textup{univ}}\rightarrow X^{\prime\textup{univ}}$ be the universal isogeny deforming $x_{0}$ over the special cycle $\mathcal{Z}^{\sharp}(x_{0})$. We will prove that the base change of $x_{0}^{\textup{univ}}$ to $\mathcal{D}(x_{0})$ is cyclic, therefore there is a natural morphism $\mathcal{D}(x_{0})\rightarrow \mathcal{N}_{0}(N)$. The natural morphism is an isomorphism because both sides of the morphism are closed formal subschemes of $\mathcal{N}$ and are represented by 2-dimensional regular local rings. The identification of $\mathcal{D}(x_{0})$ and $\mathcal{N}_{0}(N)$ implies the following difference formula of local arithmetic intersection numbers,
\begin{theorem}
For any rank 2 lattice $M\subset\mathbb{W}$, the following identity holds,
\begin{equation*}
    \textup{Int}_{\mathcal{N}_{0}(N)}(M) = 
    \textup{Int}^{\sharp}(M\obot\mathbb{Z}_{p}\cdot x_{0})-\textup{Int}^{\sharp}(M\obot\mathbb{Z}_{p}\cdot p^{-1}x_{0}).
\end{equation*}\label{intro1}
\end{theorem}
We refer to this formula as the difference formula at the geometric side.

\subsubsection{Difference formula at the analytic side}
For any rank 3 quadratic $\mathbb{Z}_{p}$-lattice $L\subset\mathbb{B}$, define the local density of $L$ to be the polynomial $\textup{Den}^{+}(X,L)\in\mathbb{Z}[X]$ such that for all $k\geq0$,
\begin{equation*}
    \textup{Den}^{+}(X,L)\,\big\vert_{X=p^{-k}}=\frac{\textup{Den}(H_{2k+4}^{+},L)}{\textup{Nor}^{+}(p^{-k},3)}.
\end{equation*}
then $\textup{Den}^{+}(1,L)=0$ since $L$ can't be isometrically embedded into the quadratic lattice $H_{4}^{+}$, we define the derived local density of $L$ to be
\begin{equation*}
    \partial\textup{Den}^{+}(L)\coloneqq-\frac{\textup{d}}{\textup{d}X}\bigg\vert_{X=1}\textup{Den}^{+}(X,L).
\end{equation*}
\begin{theorem}
For any rank 2 lattice $M\subset\mathbb{W}$, the following identity holds,
\begin{equation*}
    \textup{Den}_{\Delta(N)}^{+}(X,M)=\textup{Den}^{+}(X, M\obot\mathbb{Z}_{p}\cdot x_{0})-X^{2}\cdot\textup{Den}^{+}(X, M\obot\mathbb{Z}_{p}\cdot p^{-1}x_{0}).
\end{equation*}
Since the lattice $M\obot\mathbb{Z}_{p}\cdot x_{0}$ can't be isometrically embedded into the lattice $H_{4}^{+}$,
\begin{equation*}
    \partial\textup{Den}^{+}_{\Delta(N)}(M)= \partial\textup{Den}^{+}(M\obot\mathbb{Z}_{p}\cdot x_{0})-\partial\textup{Den}^{+}(M\obot\mathbb{Z}_{p}\cdot p^{-1}x_{0}).
\end{equation*}
\label{intro2}
\end{theorem}
The theorem is proved in a more general form in Theorem \ref{anadiff}. We refer to this formula as the difference formula at the analytic side.

\subsubsection{Proof of Theorem \ref{MainGlobal}}
The following local arithmetic Siegel-Weil formula is proved in \cite[$\S$2.16]{Wed07} (see also \cite[Theorem 1.2.1]{LZ22} when $p$ is odd).
\begin{theorem}
    For any rank 3 lattice $L\subset\mathbb{B}$, the following identity holds,
    \begin{equation*}
        \textup{Int}^{\sharp}(L)=\partial\textup{Den}^{+}(L).
    \end{equation*}
\end{theorem}
For a rank 2 lattice $M\subset\mathbb{W}$, let $L=M\obot\mathbb{Z}_{p}\cdot x_{0}\subset\mathbb{B}$, the local arithmetic Siegel-Weil formula with level $N$ in Theorem \ref{MainLocal} follows immediately from $\textup{Int}^{\sharp}(L)=\partial\textup{Den}^{+}(L)$ and two difference formulas we stated before (Theorem \ref{intro1} and Theorem \ref{intro2}).\\
\subsubsection{Proof of Theorem \ref{MainLocal}}
Let $T\in\textup{Sym}_{2}(\mathbb{Q})$ be a nonsingular matrix. Let $\boldsymbol{\varphi}\in\mathscr{S}(\mathbb{V}_{f}^{2})$ be a $T$-admissible function. When $T$ is not positive definite, the arithmetic special cycle $\widehat{\mathcal{Z}}(T,\mathsf{y},\boldsymbol\varphi)$ is essentially an $(1,1)$-current on the proper smooth complex curve $\mathcal{X}_{0}(N)_{\mathbb{C}}$, the number $\widehat{\textup{deg}}(\widehat{\mathcal{Z}}(T,\mathsf{y},\boldsymbol\varphi))$ has been computed explicitly in \cite[Theorem 4.10]{SSY22}. 
\par
We focus on the case that $T$ is positive definite. In this case, $\widehat{\mathcal{Z}}(T,\mathsf{y},\boldsymbol{\varphi})=[(\mathcal{Z}(T,\varphi),0)]$ where $\mathcal{Z}(T,\boldsymbol\varphi)$ is a cycle of codimension 2 on $\mathcal{X}_{0}(N)$. Moreover, $\mathcal{Z}(T,\boldsymbol\varphi)=\varnothing$ only if $\text{Diff}(T,\Delta(N))=\{p\}$ for some prime number $p$, in this case the special cycle $\mathcal{Z}(T,\boldsymbol\varphi)$ is concentrated in the supersingular locus of $\mathcal{X}_{0}(N)$ in characteristic $p$. Suppose that the $2\times2$ matrix $T$ has diagonal element $t_{1}$ and $t_{2}$, $\boldsymbol\varphi=\varphi_{1}\times\varphi_{2}\in\mathscr{S}(\mathbb{V}_{f}^{2})$, where $\varphi_{i}\in\mathscr{S}(\mathbb{V}_{f})$. We will show that
\begin{equation*}
    \widehat{\textup{deg}}(\widehat{\mathcal{Z}}(T,\mathsf{y},\boldsymbol{\varphi})) = \chi(\mathcal{Z}(T,\boldsymbol{\varphi}), \mathcal{O}_{\mathcal{Z}(t_{1},\varphi_{1})}\otimes^{\mathbb{L}}\mathcal{O}_{\mathcal{Z}(t_{2},\varphi_{2})})\cdot\textup{log}(p),
\end{equation*}
By the formal uniformization of the special cycle $\mathcal{Z}(T,\boldsymbol{\varphi})$ in Corollary \ref{uniZ}, the Euler characteristic $\chi(\mathcal{Z}(T,\boldsymbol{\varphi}), \mathcal{O}_{\mathcal{Z}(t_{1},\varphi_{1})}\otimes^{\mathbb{L}}\mathcal{O}_{\mathcal{Z}(t_{2},\varphi_{2})})$ is a weighted linear combination of local arithmetic intersection numbers on $\mathcal{N}_{0}(N)$, Theorem \ref{MainGlobal} follows from the local arithmetic Siegel-Weil formula with level $N$ at the place $p$ and the classical local Siegel-Weil formula at other places.

\subsection{Supplement}
By the Windows theory developed by Zink in \cite{Zin01}, if $\nu_{p}(N)\geq1$, we prove that the special fiber $\mathcal{Z}(x_{0})_{p}$ of $\mathcal{Z}(x_{0})$ has the following explicit description (cf. Theorem \ref{spci}, Corollary \ref{spcci}),
\begin{equation*}
    \mathcal{Z}(x_{0})_{p}\simeq\textup{Spf}\,\mathbb{F}[[t_{1},t_{2}]]/\left(\prod\limits_{a+b=n\atop a,b\geq0}(t_{1}^{p^{a}}-t_{2}^{p^{b}})\right).
\end{equation*}
Based on the isomorphism $\mathcal{D}(x_{0})\stackrel{\sim}\rightarrow \mathcal{N}_{0}(N)$, the special fiber $\mathcal{N}_{0}(N)_{p}$ of $\mathcal{N}_{0}(N)$ can be described by
\begin{equation*}
    \mathcal{N}_{0}(N)_{p}\simeq\textup{Spf}\,\mathbb{F}[[t_{1},t_{2}]]/\left((t_{1}-t_{2}^{p^{n}})\cdot(t_{2}-t_{1}^{p^{n}})\cdot\prod\limits_{a+b=n\atop a,b\geq1}(t_{1}^{p^{a-1}}-t_{2}^{p^{b-1}})^{p-1}\right).
\end{equation*}
Both these two isomorphisms are proved in \cite[Theorem 13.4.6, Theorem 13.4.7]{KM85} by a totally different method.
\\
\par
\noindent\textbf{Acknowledgements.}\,\,First, I would like to express my deep gratitude to Professor Chao Li, who suggested I consider the question, and had countless instructive conversations with me during this project, without which this article would never exist. Next, I would like to thank Michael Rapoport, Tonghai Yang and Qiao He for their careful reading of the manuscript and many helpful comments. Moreover, I thank Cathy Chen for her unwavering support. The author is supported by the Department of Mathematics at Columbia University in the city of New York.

\section{Quadratic lattices and local densities}
\label{2}
\subsection{Notations on quadratic lattices}
Let $p$ be a prime number. Let $F$ be a nonarchimedean local field of residue characteristic $p$, with ring of integers $\mathcal{O}_{F}$, residue field $\kappa = \mathbb{F}_{q}$ of size $q$, and uniformizer $\pi$. Let $\nu_{\pi} : F \rightarrow \mathbb{Z}\cup\{\infty\}$ be the valuation on $F$ and $\vert\cdot\vert: F \rightarrow \mathbb{R}_{\geq0}$ be the normalized absolute value on $F$. Let $(\cdot,\cdot)_{F}$ be the Hilbert symbol on the local field $F$. Consider the following character:
\begin{align*}
    \chi_{F}:F^{\times}/(F^{\times})^{2}&\rightarrow\{\pm1,0\},\\
    x&\mapsto\begin{cases}
        (x,\pi)_{F}, &\textup{when $\nu_{\pi}(x)$ is even;}\\
        0, &\textup{when $\nu_{\pi}(x)$ is odd.}
    \end{cases}
\end{align*}
the character $\chi_{F}$ is independent of the choice of the uniformizer $\pi$.
\par
A quadratic space $(U,q_{U})$ over $F$ is a finite dimensional vector space $U$ over $F$ equipped with a quadratic form $q_{U}:U\rightarrow F$, the quadratic form $q_{U}$ induces a symmetric bilinear form given by 
\begin{align}
    (\cdot,\cdot):U\times U&\rightarrow F,\notag\\
    (x,y)&\mapsto q_{U}(x+y)-q_{U}(x)-q_{U}(y).\label{to}
\end{align}
An isometry between two quadratic spaces $(U,q_{U})$ and $(U^{\prime},q_{U^{\prime}})$ is a linear isomorphism $\phi:U\rightarrow U^{\prime}$ preserving quadratic forms, i.e., $q_{U^{\prime}}(\phi(x))=q_{U}(x)$ for any $x\in U$. In that case, we say $U$ and $U^{\prime}$ are isometric.
\par
A quadratic lattice $(L,q_{L})$ is a finite free $\mathcal{O}_{F}$-module equipped with a quadratic form $q_{L}:L\rightarrow F$. The quadratic form $q_{L}$ also induces a symmetric bilinear form $L\times L\stackrel{(\cdot,\cdot)}\longrightarrow F$ by similar formula (\ref{to}). Let $L^{\vee}=\{x\in L\otimes_{\mathcal{O}_{F}}F:(x,L)\subset \mathcal{O}_{F}\}$. We say a quadratic lattice is integral if $q_{L}(x)\in\mathcal{O}_{F}$ for all $x\in L$, is self-dual if it is integral and $L$ = $L^{\vee}$.
\par
Let's assume that $\textup{dim}_{F}U=n$ and the symmetric bilinear form $(\cdot,\cdot)$ is nondegenerate. Let $\{x_{i}\}_{i=1}^{n}$ be a basis of $U$, and $t_{ij}=\frac{1}{2}(x_{i},x_{j})$, we define the discriminant of the quadratic space $U$ to be:
\begin{equation*}
    \textup{disc}(U) = (-1)^{n\choose 2}\textup{det}\left((t_{ij})\right)\in F^{\times}/(F^{\times})^{2}.
\end{equation*}
If $\{x_{i}\}_{i=1}^{n}$ is an orthogonal basis of $U$ then $t_{ij}=0$ if $i\neq j$ and $t_{ii}\neq 0$ by the nondegeneracy of $(\cdot,\cdot)$. The Hasse invariant of the quadratic space $U$ is
\begin{equation*}
    \epsilon(U) = \prod\limits_{i<j}(t_{ii},t_{jj})_{F},
\end{equation*}
\par
For a quadratic lattice $L$, we use $\textup{disc}(L)$ and $\epsilon(L)$ to denote the corresponding invariants on the quadratic space $L_{F}=L\otimes_{\mathcal{O}_{F}}F$. Recall that when $p$ is odd, quadratic spaces $U$ over $F$ are classified by the following three invariants:
\begin{equation*}
    \textup{dim}_{F}\,U,\,\,\,\,\,\chi_{F}(U)\coloneqq\chi_{F}(\textup{disc}(U)),\,\,\,\,\,\epsilon(U).
\end{equation*}
i.e., two quadratic spaces $U$ and $U^{\prime}$ are isometric if and only if the above three invariants for $U$ and $U^{\prime}$ are the same. 
\par
When $p$ is odd, the quadratic space $U$ admits a self-dual sub-lattice if and only if $\epsilon(U)=+1$ and $\chi_{F}(U)\neq0$, we will use $H_{k}^{\varepsilon}$ to denote the unique self-dual lattice of rank $k$ and 
\begin{equation*}
    \chi_{F}(H_{k}^{\varepsilon})\coloneqq\chi_{F}(\textup{disc}(H_{k}^{\varepsilon}))=\varepsilon.
\end{equation*}
When $p=2$, let $H_{2n}^{+}=(H_{2}^{+})^{\obot n}$ be a self-dual lattice of rank $2n$, where the quadratic form on $H_{2}^{+}=\mathcal{O}_{F}^{2}$ is given by $(x,y)\in \mathcal{O}_{F}^{2}\mapsto xy$.
\label{quadratic}
\begin{example}
Let $N\in\mathcal{O}_{F}$. Let $\delta_{F}(N)$ be the following rank 3 quadratic lattice over $\mathcal{O}_{F}$,
\begin{equation*}
    \delta_{F}(N) = \left\{x=\begin{pmatrix}
    -Na & b\\
    c & a
    \end{pmatrix}:\,\, a,b,c\in\mathcal{O}_{F}\right\}.
\end{equation*}
equipped with the quadratic form induced by $x\mapsto \textup{det}(x)$. Under the following basis of $\delta_{F}(N)$,
\begin{equation*}
    e_{1}=\begin{pmatrix}
    -N & \\
     & 1
    \end{pmatrix},\,\,e_{2}=\begin{pmatrix}
    \,  & 1\\
     \, & 
    \end{pmatrix},\,\,e_{3}=\begin{pmatrix}
     \, & \,\\
     1& \,
    \end{pmatrix}.
\end{equation*}
the quadratic form can be represented by the following symmetric matrix,
\begin{equation*}
    T=\begin{pmatrix}
    -N & 0 & 0 \\
      0 & 0 & -\frac{1}{2}\\
      0 & -\frac{1}{2} & 0
    \end{pmatrix}.
\end{equation*}
therefore $\textup{disc}(\delta_{F}(N))=-\frac{1}{4}N\equiv -N$, $\epsilon(\delta_{F}(N))=(N,-1)_{F}$. Moreover,
\begin{equation*}
    \delta_{F}(N)^{\vee} = \left\{x=\begin{pmatrix}
    -Na & b\\
    c & a
    \end{pmatrix}:\,\, a\in\frac{1}{2N}\mathcal{O}_{F},b,c\in\mathcal{O}_{F}\right\}.
\end{equation*}
therefore $\delta_{F}(N)^{\vee}/\delta_{F}(N)\simeq\mathcal{O}_{F}/2N$.
\par
Throughout this article, we will mainly focus on the case that $F=\mathbb{Q}_{p}$. In this case, we simply use $\delta_{p}(N)$ to denote the lattice $\delta_{\mathbb{Q}_{p}}(N)$ (as we did in the introduction).
\label{lambda}
\end{example}

\subsection{Local densities of quadratic lattices}
\begin{definition}
Let $L, M$ be two quadratic $\mathcal{O}_{F}$-lattices. Let $\textup{Rep}_{M,L}$ be the scheme of integral representations, an $\mathcal{O}_{F}$-scheme such that for any $\mathcal{O}_{F}$-algebra $R$,
\begin{equation*}
    \textup{Rep}_{M,L}(R)=\textup{QHom}(L\otimes_{\mathcal{O}_{F}}R, M\otimes_{\mathcal{O}_{F}}R),
\end{equation*}
where \textup{QHom} denotes the set of injective module homomorphisms which preserve the quadratic forms. The local density of integral representations is defined to be
\begin{equation*}
    \textup{Den}(M,L)=\lim\limits_{d\rightarrow\infty}\frac{\#\textup{Rep}_{M,L}(\mathcal{O}_{F}/\pi^{d})}{q^{d\cdot \textup{dim}(\textup{Rep}_{M,L})_{F}}}.
\end{equation*}
\end{definition}
\begin{remark}
If $L,M$ have rank $n,m$ respectively and the generic fiber $(\textup{Rep}_{M,L})_{F}\neq\varnothing$, then $n\leq m$ and 
\begin{equation*}
    \textup{dim}(\textup{Rep}_{M,L})_{F}=\textup{dim} \,\textup{O}_{m}-\textup{dim}\, \textup{O}_{m-n}=\tbinom{m}{2}-\tbinom{m-n}{2}=mn-\frac{n(n+1)}{2}.
\end{equation*}
\end{remark}
\begin{definition}
Let $L, M$ be two quadratic $\mathcal{O}_{F}$-lattices. Let $\textup{PRep}_{M,L}$ be the $\mathcal{O}_{F}$-scheme of primitive integral representations such that for any $\mathcal{O}_{F}$-algebra $R$,
\begin{equation*}
    \textup{PRep}_{M,L}(R)=\{\phi\in\textup{Rep}_{M,L}(R):\textup{$\phi$ is an isomorphism between $L_{R}$ and a direct summand of $M_{R}$}\}.
\end{equation*}
where $L_{R}$ (resp. $M_{R}$) is $L\otimes_{\mathcal{O}_{F}}R$ (resp. $M\otimes_{\mathcal{O}_{F}}R$). The primitive local density is defined to be
\begin{equation*}
    \textup{Pden}(M,L)=\lim\limits_{d\rightarrow\infty}\frac{\#\textup{PRep}_{M,L}(\mathcal{O}_{F}/\pi^{d})}{q^{d\cdot \textup{dim}(\textup{Rep}_{M,L})_{F}}}.
\end{equation*}
\end{definition}
\begin{remark}
For any positive integer $d$, a homomorphism $\phi\in\textup{Rep}_{M,L}(\mathcal{O}_{F}/\pi^{d})$ or $\textup{Rep}_{M,L}(\mathcal{O}_{F})$ is primitive if and only if $\overline{\phi} \coloneqq \phi\,\,\textup{mod}\,\pi\in\textup{PRep}(\mathcal{O}_{F}/\pi)$, which is equivalent to $\textup{dim}_{\mathbb{F}_{q}}(\phi(L)+\pi\cdot M)/\pi\cdot M = \textup{rank}_{\mathcal{O}_{F}}(L)$. 
\end{remark}
\begin{lemma}
Let $H$ be a self-dual quadratic lattice. Let L be a quadratic $\mathcal{O}_{F}$-lattice, k is any positive integer, then we have the following stratification,
\begin{equation*}
    \textup{Rep}_{H,L}(\mathcal{O}_{F}) = \bigsqcup_{L\subset L^{\prime}\subset L^{\prime\vee}}\textup{PRep}_{H,L^{\prime}}(\mathcal{O}_{F}).
\end{equation*}
\label{stratification}
\end{lemma}
\begin{proof}
This is proved by Cho and Yamauchi in \cite[(3.1)]{CY20}.
\end{proof}
\begin{definition}
Let $n\geq0$, for $\varepsilon\in\{\pm 1\}$, we define the normalizing factors to be
    \begin{align*}
    \textup{Nor}^{\varepsilon}(X,n)&=(1-\frac{1+(-1)^{n+1}}{2}\cdot\varepsilon q^{-(n+1)/2}X)\prod\limits_{1\leq i<(n+1)/2}(1-q^{-2i}X^{2})
\end{align*}   
\label{normal1}
\end{definition}
\begin{lemma}
Let $L$ be a quadratic lattice of rank $n$. When $p$ is odd, there exist polynomials $\textup{Den}^{\varepsilon}(X,L)$ and $\textup{Den}^{\flat\varepsilon}(X,L)\in\mathbb{Z}[X]$, such that
\begin{equation*}
    \textup{Den}^{\varepsilon}(X,L)\,\big\vert_{X=q^{-k}} = \frac{\textup{Den}(H_{2k+n+1}^{\varepsilon}, L)}{\textup{Nor}^{\varepsilon}(q^{-k},n)},\,\,\,\,\,\,\textup{Den}^{\flat\varepsilon}(X,L)\,\big\vert_{X=q^{-k}}=\frac{\textup{Den}(H_{2k+n}^{\varepsilon},L)}{\textup{Nor}^{\varepsilon}(q^{-k},n-1)}\cdot(1-\varepsilon\chi_{F}(L)q^{-k}),
\end{equation*}
we define the derived local density of $L$ to be
\begin{equation*}
    \partial\textup{Den}^{\varepsilon}(L) = -\frac{\textup{d}}{\textup{d}X}\bigg|_{X=1}\textup{Den}^{\varepsilon}(X,L).
\end{equation*}
There also exists a polynomial $\textup{Pden}^{\varepsilon}(X,L)$, such that
\begin{equation*}
    \textup{Pden}^{\varepsilon}(X,L)\,\big\vert_{X=q^{-k}} = \frac{\textup{Pden}(H_{2k+n+1}^{\varepsilon}, L)}{\textup{Nor}^{\varepsilon}(q^{-k},n)}.
\end{equation*}\par
When $p=2$ and $n$ is odd, the polynomials $\textup{Den}^{+}(X,L)$ and $\textup{Pden}^{+}(X,L)$ with the same evaluation formulas at $X=q^{-k}$ exist, the derived local density $\partial\textup{Den}^{+}(L)$ of $L$ is defined similarly.
\label{poly}
\end{lemma}
\begin{proof}
This is a combination of \cite[Lemma 3.3.2, Definition 3.4.1, Definition 3.5.2]{LZ22}
\end{proof}
\par
For any $N\in\mathcal{O}_{F}$, let $(\langle N\rangle,q_{\langle N\rangle})$ be the rank 1 $\mathcal{O}_{F}$-lattice with an $\mathcal{O}_{F}$ generator $l_{N}$ such that $q_{\langle N\rangle}(l_{N})=N$. The following induction formula of local density is proved in \cite{LZ22},
\begin{lemma}
Assume $p$ is odd. Let $L$ be a quadratic lattice of rank $n$. Let $N\in\mathcal{O}_{F}$, when $\nu_{\pi}(N)$ is sufficiently large,
\begin{equation*}
    \textup{Den}^{\varepsilon}(X,L\obot\langle N\rangle)=X^{2}\cdot\textup{Den}^{\varepsilon}(X,L\obot\langle \pi^{-2}N\rangle)+(1-\varepsilon\chi_{F}(L)X)^{-1}(1-X^{2})\cdot\textup{Den}^{\flat\varepsilon}(X,L).
\end{equation*}
Moreover, if $L\obot\langle N\rangle$ can't be isometrically embedded into the quadratic space $H_{n+2}^{\varepsilon}\otimes_{\mathcal{O}_{F}}F$, then 
\begin{equation*}
    \partial\textup{Den}^{\varepsilon}(L\obot\langle N\rangle)-\partial\textup{Den}^{\varepsilon}(L\obot\langle \pi^{-2}N\rangle)=\begin{cases}
    -\varepsilon\chi_{F}(L)\cdot\textup{Den}^{\flat\varepsilon}(1,L), & \textup{if $\chi_{F}(L)\neq 0$;}\\
    2\cdot \textup{Den}^{\flat\varepsilon}(1,L), &\textup{if $\chi_{F}(L)=0$.}\\
    \end{cases}
\end{equation*}
\label{induction}
\end{lemma}
\begin{proof}
This is proved in \cite[Theorem 3.6.1]{LZ22}
\end{proof}

\section{Incoherent Eisenstein series and the main theorem}
\label{5}
\subsection{Incoherent Eisenstein series}
Let $\mathsf{W}$ be the standard symplectic space over $\mathbb{Q}$ of dimension 4. Let $P=MN\subset \textup{Sp}(\mathsf{W})$ be the standard Siegel parabolic subgroup, which take the following form under the standard basis of $\mathsf{W}$,
\begin{align*}
    M(\mathbb{Q}) & = \left\{m(a)=\begin{pmatrix}a & 0\\0 & ^{t}a^{-1}\end{pmatrix}: a\in \textup{GL}_{2}(\mathbb{Q})\right\},\\
    N(\mathbb{Q}) & = \left\{n(b)=\begin{pmatrix}1_{2} & b\\0 & 1_{2}\end{pmatrix}: b\in \textup{Sym}_{2}(\mathbb{Q})\right\}.
\end{align*}
Let $\mathbb{A}$ be the ad$\grave{\textup{e}}$le ring over $\mathbb{Q}$. Let $\textup{Mp}(\mathsf{W}_{\mathbb{A}})$ be the following metaplectic extension of $\textup{Sp}(\mathsf{W})(\mathbb{A})$,
\begin{equation*}
    1\rightarrow \mathbb{C}^{1}\rightarrow \textup{Mp}(\mathsf{W}_{\mathbb{A}})\rightarrow \textup{Sp}(\mathsf{W})({\mathbb{A}})\rightarrow 1,
\end{equation*}
where $\mathbb{C}^{1}=\{z\in\mathbb{C}^{\times}:\vert z\vert=1\}$. There is an isomorphism $\textup{Mp}(\mathsf{W}_{\mathbb{A}})\stackrel{\sim}\rightarrow \textup{Sp}(\mathsf{W})({\mathbb{A}})\times \mathbb{C}^{1}$ with the multiplication on the latter is given by the global Rao cycle, therefore we can write an element of $\textup{Mp}(\mathsf{W}_{\mathbb{A}})$ as $(g,t)$ where $g\in \textup{Sp}(\mathsf{W})({\mathbb{A}})$ and $t\in \mathbb{C}^{1}$.
\par
Let $P(\mathbb{A}) = M(\mathbb{A})N(\mathbb{A})$ be the standard Siegel parabolic subgroup of $\textup{Mp}(\mathsf{W}_{\mathbb{A}})$ where 
\begin{align*}
    M(\mathbb{A}) & = \left\{(m(a),t): a\in \textup{GL}_{2}(\mathbb{A}), t\in\mathbb{C}^{1}\right\},\\
    N(\mathbb{A}) & = \left\{n(b): b\in \textup{Sym}_{2}(\mathbb{A})\right\}.
\end{align*}
Recall the following incoherent collection of rank $3$ quadratic spaces $\mathbb{V}=\{\mathbb{V}_{v}\}$ over $\mathbb{A}$ we defined in (\ref{incoherent}),
\begin{equation*}
    \mathbb{V}_{v}=\delta_{v}(N)\otimes\mathbb{Q}_{v}\,\,\textup{if $v<\infty$, and $\mathbb{V}_{\infty}$ is positive definite.}
\end{equation*}
then we can verify immediately that $\prod\limits_{v}\epsilon(\mathbb{V}_{v})=-1$. 
\par
Let $\chi:\mathbb{A}^{\times}/\mathbb{Q}^{\times}\rightarrow\mathbb{C}^{\times}$ be the quadratic character given by $\chi(x)=\prod\limits_{v\leq\infty}\chi_{v}(x_{v}) =\prod\limits_{v\leq\infty}(x_{v},-N)_{v}$ for all $x=(x_{v})\in\mathbb{A}^{\times}$. Fix the standard additive character $\psi:\mathbb{A}/\mathbb{Q}\rightarrow\mathbb{C}^{\times}$ such that $\psi_{\infty}(x)=e^{2\pi ix}$. We may view $\chi$ as a character on $M(\mathbb{A})$ by
\begin{equation*}
    \chi(m(a),t)=\chi(\textup{det}(a))\cdot\gamma(\textup{det}(a),\psi)^{-1}\cdot t.
\end{equation*}
and extend it to $P(\mathbb{A})$ trivially on $N(\mathbb{A})$. Here $\gamma(\textup{det}(a),\psi)$ is the Weil index (see the work of Kudla \cite[p. 548]{Kud97}). We define the degenerate principal series to be the unnormalized smooth induction
\begin{equation*}
    I(s,\chi)\coloneqq\textup{Ind}_{P(\mathbb{A})}^{\textup{Mp}(\mathsf{W}_{\mathbb{A}})}(\chi\cdot\vert\cdot\vert_{\mathbb{Q}}^{s+3/2}),\,\,\,s\in\mathbb{C}.
\end{equation*}
For a standard section $\Phi(-,s)\in I(s,\chi)$ (i.e., its restriction to the standard maximal compact subgroup of $\textup{Mp}(\mathsf{W}_{\mathbb{A}})$ is independent of $s$), we define the associated Siegel Eisenstein series
\begin{equation*}
    E(g,s,\Phi)=\sum\limits_{\gamma\in P(\mathbb{Q})\backslash \textup{Sp}(\mathbb{Q})}\Phi(\gamma g,s),
\end{equation*}
which converges for $\textup{Re}(s)\gg 0$ and admits meromorphic continuation to $s\in\mathbb{C}$.
\par
Recall that $\mathscr{S}(\mathbb{V}^{2})$ is the space of Schwartz functions on $\mathbb{V}^{2}$. The fixed choice of $\chi$ and $\psi$ gives a Weil representation $\omega = \omega_{\chi,\psi}$ of $\textup{Mp}(\mathsf{W}_{\mathbb{A}})\times\textup{O}(\mathbb{V})$ on $\mathscr{S}(\mathbb{V}^{2})$. For $\tilde{\boldsymbol\varphi}\in\mathscr{S}(\mathbb{V}^{2})$, define a function
\begin{equation*}
    \Phi_{\tilde{\boldsymbol\varphi}}(g)\coloneqq\omega(g)\tilde{\boldsymbol\varphi}(0),\,\,\,g\in\textup{Mp}(\mathsf{W}_{\mathbb{A}}).
\end{equation*}
Then $\Phi_{\tilde{\boldsymbol\varphi}}(g)\in I(0,\chi)$. Let $\Phi_{\tilde{\boldsymbol\varphi}}(-,s)\in I(s,\chi)$ be the associated standard section, known as the standard Siegel–Weil section associated to $\tilde{\boldsymbol\varphi}$. For $\tilde{\boldsymbol\varphi}\in \mathscr{S}(\mathbb{V}^{2})$, we write $E(g,s,\tilde{\boldsymbol\varphi})\coloneqq E(g,s,\Phi_{\tilde{\boldsymbol\varphi}})$.

\subsection{Fourier coefficients and derivatives}
We have a Fourier expansion of the Siegel Eisenstein series defined above,
\begin{equation*}
    E(g,s,\Phi)=\sum\limits_{T\in\textup{Sym}_{2}(\mathbb{Q})}E_{T}(g,s,\Phi),
\end{equation*}
where
\begin{equation*}
    E_{T}(g,s,\Phi)=\int_{\textup{Sym}_{2}(\mathbb{Q})\backslash \textup{Sym}_{2}(\mathbb{A})}E(n(b)g,s,\Phi)\psi(-\textup{tr}(Tb))\textup{d}n(b),
\end{equation*}
the Haar measure $\textup{d}n(b)$ is normalized to be self-dual with respect to $\psi$. When $T$ is nonsingular, for factorizable $\Phi = \otimes_{v}\Phi_{v}$, we have a factorization of the Fourier coefficient into a product
\begin{equation*}
     E_{T}(g,s,\Phi)=\prod_{v}W_{T,v}(g_{v},s,\Phi_{v}),
\end{equation*}
where the product ranges over all places $v$ of $\mathbb{Q}$ and the local Whittaker function $W_{T,v}(g_{v},s,\Phi_{v})$ is defined by
\begin{equation}
    W_{T,v}(g_{v},s,\Phi_{v})=\int_{\textup{Sym}_{2}(\mathbb{Q}_{v})}\Phi_{v}(w^{-1}n(b)g_{v},s)\cdot\psi_{v}(-\textup{tr}(Tb))\textup{d}n(b),\,\,\,w=\begin{pmatrix}0 & 1_{2}\\-1_{2} & 0\end{pmatrix},
    \label{lowhit}
\end{equation}
it has analytic continuation to $s \in \mathbb{C}$. Thus we have a decomposition of the derivative of a nonsingular Fourier coefficient,
\begin{equation*}
    E_{T}^{\prime}(g,s,\Phi) = \sum\limits_{v} E_{T,v}^{\prime}(g,s,\Phi),
\end{equation*}
where
\begin{equation}
    E_{T,v}^{\prime}(g,s,\Phi)=W_{T,v}^{\prime}(g_{v},s,\Phi_{v})\cdot\prod\limits_{v^{\prime}\neq v}W_{T,v^{\prime}}(g_{v^{\prime}},s,\Phi_{v^{\prime}}).
    \label{dewhit}
\end{equation}

\subsection{Whittaker functions and local densities}
Let $v$ be a finite place of $\mathbb{Q}$. Define the local degenerate principal series to be the unnormalized smooth induction
\begin{equation*}
     I_{v}(s,\chi_{v})\coloneqq\textup{Ind}_{P(\mathbb{Q}_{v})}^{\textup{Mp}(\mathsf{W}_{v})}(\chi_{v}\cdot\vert\cdot\vert_{v}^{s+3/2}),\,\,\,s\in\mathbb{C}.
\end{equation*}
\par
The fixed choice of $\chi_{v}$ and $\psi_{v}$ gives a local Weil representation $\omega_{v}=\omega_{\chi_{v},\psi_{v}}$ of $\textup{Mp}(\mathsf{W}_{v})\times\textup{O}(\mathbb{V}_{v})$ on the Schwartz function space $\mathscr{S}(\mathbb{V}_{v}^{2})$. We define the local Whittaker function associated to $\boldsymbol\varphi_{v}$ and $T\in\textup{Sym}_{2}(\mathbb{Q}_{v})$ to be
\begin{equation*}
    W_{T,v}(g_{v},s,\boldsymbol\varphi_{v})\coloneqq W_{T,v}(g_{v},s,\Phi_{\boldsymbol\varphi_{v}}),
\end{equation*}
where $\Phi_{\boldsymbol\varphi_{v}}(g_{v})\coloneqq\omega_{v}(g_{v})\boldsymbol\varphi_{v}(0)\in I_{v}(0,\chi_{v})$ and $\Phi_{\boldsymbol\varphi_{v}}(-,s)$ is the associated standard section.

\par
The relationship between Whittaker functions and local densities is encoded in the following proposition.
\begin{proposition}
Suppose $v\neq\infty$. Let $M$ be an integral $\mathbb{Z}_{v}$-quadratic lattice of rank 3 contained in $\mathbb{V}_{v}$. Let $L$ be an integral quadratic $\mathbb{Z}_{v}$-lattice of rank 2. Suppose that the quadratic form of $L$ is represented by a matrix $T\in \textup{Sym}_{2}(\mathbb{Q}_{v})$ after a choice of $\mathbb{Z}_{v}$-basis of $L$, we have the following identity,
\begin{equation}
    W_{T,v}(1,k,1_{M^{2}})=\vert  M^{\vee}/M\vert_{v}\cdot\gamma(\mathbb{V}_{v})^{2}\cdot\vert 2\vert_{v}^{1/2}\cdot \textup{Den}(M\obot H_{2k}^{+}, L).
\end{equation} 
where the constant $\gamma(\mathbb{V}_{v})=\gamma(\textup{det}(\mathbb{V}_{v}),\psi_{v})^{-1}\cdot\epsilon(\mathbb{V}_{v})\cdot\gamma(\psi_{v})^{-3}$, $\gamma(\textup{det}(\mathbb{V}_{v}),\psi_{v})$ and $\gamma(\psi_{v})$ are Weil indexes (cf. Appendix of \cite{Rao93}).
\label{non-homo}
\end{proposition}
\begin{proof}
This is proved in \cite[Lemma 5.7.1]{KRY06}.    
\end{proof}

\subsection{Classical incoherent Eisenstein series}
The hermitian symmetric domain for $\textup{Sp}(\mathsf{W})$ is the Siegel upper half space
\begin{equation*}
    \mathbb{H}_{2}=\{\mathsf{z}=\mathsf{x}+i\mathsf{y}\,\,\vert\,\,\mathsf{x}\in\textup{Sym}_{2}(\mathbb{R}),\mathsf{y}\in\textup{Sym}_{2}(\mathbb{R})_{>0}\}.
\end{equation*}
Let $\mathsf{z}=\mathsf{x}+i\mathsf{y}\in\mathbb{H}_{2}$ with $\mathsf{x},\mathsf{y}\in\textup{Sym}_{2}(\mathbb{R})$ and $\mathsf{y}={^{t}a}\cdot a$ is positive definite. Define the classical incoherent Eisenstein series to be
\begin{equation*}
    E(\mathsf{z},s,\tilde{\boldsymbol\varphi})=\chi_{\infty}(m(a))^{-1}\vert\textup{det}(m(a))\vert^{-3/2}\cdot E(g_{\mathsf{z}},s,\tilde{\boldsymbol\varphi}),\,\,g_{\mathsf{z}}=n(x)m(a)\in\textup{Mp}(\mathsf{W}_{\mathbb{A}}).
\end{equation*}
Notice that $E(\mathsf{z},s,\tilde{\boldsymbol\varphi})$ doesn't depend on the choice of $\chi$. We write the central derivatives as,
\begin{equation}
    \partial\textup{Eis}(\mathsf{z},\tilde{\boldsymbol\varphi}) \coloneqq E^{\prime}(\mathsf{z},0,\tilde{\boldsymbol\varphi}),\,\,\partial\textup{Eis}_{T}(\mathsf{z},\tilde{\boldsymbol\varphi}) \coloneqq E_{T}^{\prime}(\mathsf{z},0,\tilde{\boldsymbol\varphi}).
    \label{classi}
\end{equation}
Then we have a Fourier expansion
\begin{equation*}
     \partial\textup{Eis}(\mathsf{z},\tilde{\boldsymbol\varphi})=\sum\limits_{T\in\textup{Sym}_{2}(\mathbb{Q})}\partial\textup{Eis}_{T}(\mathsf{z},\tilde{\boldsymbol\varphi}).
\end{equation*}
For the open compact subgroup $\Gamma_{0}(N)(\hat{\mathbb{Z}})\subset \textup{GL}_{2}(\mathbb{A}_{f})$ , we will choose
\begin{equation*}
    \tilde{\boldsymbol\varphi} = \boldsymbol\varphi\otimes\boldsymbol\varphi_{\infty}\in\mathscr{S}(\mathbb{V}^{2}).
\end{equation*}
such that $\boldsymbol\varphi\in \mathscr{S}(\mathbb{V}_{f}^{2})$ is $\Gamma_{0}(N)(\hat{\mathbb{Z}})$-invariant and $\boldsymbol\varphi_{\infty}$ is the Gaussian function
\begin{equation*}
    \boldsymbol\varphi_{\infty}(x)=e^{-\pi\,\textup{tr}\,T(x)}.
\end{equation*}
For our fixed choice of Gaussian $\boldsymbol\varphi_{\infty}$, we write
\begin{equation}
    E(\mathsf{z},s,\boldsymbol\varphi)=E(\mathsf{z},s,\boldsymbol\varphi\otimes\boldsymbol\varphi_{\infty}),\,\,\,\partial\textup{Eis}(\mathsf{z},\boldsymbol\varphi)=\partial\textup{Eis}(\mathsf{z},\boldsymbol\varphi\otimes\boldsymbol\varphi_{\infty}),\,\,\,\partial\textup{Eis}_{T}(\mathsf{z},\boldsymbol\varphi) = \partial\textup{Eis}_{T}(\mathsf{z},\boldsymbol\varphi\otimes\boldsymbol\varphi_{\infty}).
    \label{eis}
\end{equation}
and so on for short.
\label{inco}

\section{The modular curve $\mathcal{X}_{0}(N)$ and special cycles}
\label{3}
\subsection{Cyclic group schemes}
Let $S$ be a scheme. Let $G/S$ be a finite locally free group scheme over $S$. On every connected component of $S$, the rank of $G$ is a constant, if the rank is a same number $N$ for every connected component, we say that $G$ has order $N$.
\par
Let $\mathcal{O}_{S}$ be the structure sheaf of the scheme $S$. Let $G/S$ be a finite locally free group scheme of order $N$, then the structure sheaf $\mathcal{O}_{G}$ of $G$ is finite locally free of rank $N$ as an $\mathcal{O}_{S}$-module. Any element $f\in \mathcal{O}_{G}$ acts on itself by left multiplication, this defines an $\mathcal{O}_{S}$-linear endomorphism of $\mathcal{O}_{G}$, the characteristic polynomial of this endomorphism
\begin{equation*}
    \textup{det}(T-f) = T^{N}-\textup{tr}(f)T^{N-1}+\cdot\cdot\cdot+(-1)^{N}\textup{N}(f),
\end{equation*}
is a monic polynomial in $\mathcal{O}_{S}[T]$ of degree $N$.
\begin{definition}
We say that a set of $N$ not necessarily distinct points $\{P_{i}\}_{i=1}^{N}$ in $G(S)$ is a full set of sections of $G/S$ if the following condition is fulfilled: for any element $f\in \mathcal{O}_{G}$, the following equality of polynomials with coefficients in $\mathcal{O}_{S}$ holds,
\begin{equation*}
    \textup{det}(T-f) = \prod\limits_{i=1}^{N}(T-f(P_{i})).
\end{equation*}
\end{definition}
\begin{definition}
We say a finite locally free group scheme $G/S$ of rank $N$ is cyclic over $S$ if there exits a section $P\in G(S)$ such that $\{aP\}_{a=1}^{N}$ forms a full set of sections of $G/S$, we say $P$ is a generator of $G$ over $S$. We say $G/S$ is \textup{cyclic} if $G_{T}$ is cyclic over $T$ after some fppf covering by some scheme $T\rightarrow S$.
\label{cycdef}
\end{definition}
\begin{remark}
The cyclicity of a group scheme is preserved under base change by the definition, i.e., if $G/S$ is cyclic, then for any morphism $S^{\prime}\rightarrow S$, the base change group scheme $G\times_{S}S^{\prime}/S^{\prime}$ is also cyclic.
\end{remark}
\begin{proposition}
Let $S$ be a scheme. Let $E/S$ be an elliptic curve over $S$. Let $G\subset E[N]$ be a finite locally free group scheme of order $N$ over $S$. Then there exists a closed subscheme $S^{\textup{cyc}}\subset S$ which is universal for the condition ``$G$ is cyclic", in the sense that for any morphism $T\rightarrow S$, the base change $G_{T}/T$ is cyclic if and only if the morphism $T\rightarrow S$ factors through the closed subscheme $S^{\textup{cyc}}$.
\label{cyclo}
\end{proposition}
\begin{proof}
This is proved in \cite[Theorem 6.4.1]{KM85}.
\end{proof} 
\begin{lemma}
Let $W$ be a discrete valuation ring with residue characteristic $p$ and uniformizer $\pi$. Let $S$ be a reduced, noetherian, quasi-separated and flat scheme over $W$. Let $G$ be a finite locally free group scheme of order $p^{n}$ over $S$, which is also embedded into an elliptic curve $E/S$. If for every generic point $\xi$ of $S$, $G_{\xi}$ doesn't factor through the multiplication-by-$p$ morphism of $E_{\xi}$, then $G$ is a cyclic group scheme.
\label{cyclemma}
\end{lemma}
\begin{proof} Since $S$ is quasi-separated, quasi-compact and flat over $W$, then $S[\pi^{-1}]$ is dense in $S$ since the scheme-theoretic image commutes with flat base change, therefore every generic point $\xi$ lies in the open dense subscheme $S[\pi^{-1}]$. Let $\kappa(\xi)$ be the residue field of $\xi$, it has characteristic 0. 
\par
The group scheme $G_{\xi}$ is of order $p^{n}$ over the characteristic 0 field $\kappa(\xi)$, hence $G_{\xi}\simeq\prod\limits_{i=1}^{k}\mathbb{Z}/p^{a_{i}}\mathbb{Z}$ where $\sum\limits_{i=1}^{k}a_{i}=n$. Now the fact that $G_{\xi}$ doesn't factor through multiplication-by-$p$ morphism of $E_{\xi}$ implies the only possibility is $k=1$ and $G_{\xi}\simeq\mathbb{Z}/p^{n}\mathbb{Z}$.
\par
Let $S^{\textup{cyc}}$ be the closed subscheme described by Proposition \ref{cyclo}, we know that every generic point is contained in the closed subscheme $S^{\textup{cyc}}$, hence $S^{\textup{cyc}}=S$ since $S$ is reduced.
\end{proof}
\begin{corollary}
Let $W$ be a discrete valuation ring with residue characteristic $p$ and uniformizer $\pi$. Let $S$ be an integral noetherian scheme, quasi-separated and flat over $W$. Let $G$ be a finite locally free group scheme of order $p^{n}$ over $S$, which is also embedded into an elliptic curve $E/S$. If the isogeny $\pi_{G}:E\rightarrow E/G$ doesn't factor through multiplication-by-$p$ morphism of $E$, then $G$ is a cyclic group scheme.
\label{corcyc}
\end{corollary}
\begin{proof}
The isogeny $\pi_{G}:E\rightarrow E/G$ factors through the multiplication-by-$p$ morphism of $E$ if and only if $\textup{ker}([p]_{E})$ is contained (as a Cartier divisor on $E$) in $G$, this is a closed condition on the base scheme $S$ by \cite[Lemma 1.3.4]{KM85}. We use $\mathcal{I}\neq 0$ (since the morphism $\pi_{G}$ doesn't factor through the multiplication-by-$p$ morphism of $E$) to denote the ideal sheaf of this closed subscheme of $S$, it is functorial with respect to the base change of $S$.
\par
Let $\xi$ be the only generic point of $S$, then $G_{\xi}$ doesn't factor through the multiplication-by-$p$ morphism because otherwise $\mathcal{I}_{\xi}=0$, but the injection $\mathcal{I}\rightarrow\mathcal{I}_{\xi}$ will imply that $\mathcal{I}=0$, which is a contradiction. Then the corollary follows from Lemma \ref{cyclemma}.
\end{proof}

\subsection{$\Gamma_{0}(N)$-structures on elliptic curves}
Let $S$ be a scheme. We say a scheme $C$ over $S$ is a smooth curve over $S$ if the structure morphism $C\rightarrow S$ is a smooth proper morphism  of relative dimension 1. 
\begin{definition}
A closed immersion $i: D\rightarrow C$ is called an effective Cartier divisor if the following conditions hold,\\
(i) The closed subscheme $D$ is flat over $S$;\\
(ii) The ideal sheaf $\mathcal{I}(D)$ defining $D$ is an invertible $\mathcal{O}_{C}$-module.
\end{definition}
\begin{lemma}
If C/S is a smooth curve, then any section $s\in C(S)$ defines an effective Cartier divisor on C, denoted by $[s]$.
\end{lemma}
\begin{proof}
This is proved in \cite[Lemma 1.2.2]{KM85}.
\end{proof}
Given two effective Cartier divisors $D$ and $D^{\prime}$ on $C/S$, we can define their sum $D+D^{\prime}$. It is an effective Cartier divisor on $C/S$ defined locally by the product of the defining equations of $D$ and $D^{\prime}$. Explicitly, if $S=\textup{Spec}\,R$ and if over an affine open subscheme $\textup{Spec}\,A$ of $C$, the Cartier divisor $D$ (resp. $D^{\prime}$) is defined by an element $f\in A$ (resp. $g\in A$), then the Cartier divisor $D+D^{\prime}$ is defined by the equation $fg$.
\begin{lemma}
Suppose $E/S$ and $E^{\prime}/S$ are two elliptic curves over $S$, $\pi: E\rightarrow E^{\prime}$ is an isogeny, i.e., $\pi$ is surjective and $\textup{ker}(\pi)$ is a finite flat group scheme locally of finite presentation over $S$. Then $\textup{ker}(\pi)\rightarrow E$ is an effective Cartier divisor.
\end{lemma}
\begin{proof}
By cancellation theorem of morphisms of locally finite presentation, any morphism between abelian schemes are locally of finite presentation. Hence $\pi$ is locally of finite presentation, and therefore $\textup{ker}(\pi)$ is also locally of finite presentation over $S$. Then the lemma follows from  of \cite[Lemma 1.2.3]{KM85}.\end{proof}
\begin{definition}
We say an isogeny $\pi:E\rightarrow E^{\prime}$ between two elliptic curves $E$ and $E^{\prime}$ is a cyclic $N$-isogeny if $\pi^{\vee}\circ\pi=N$, and there exists an fppf covering of $S$ by a scheme $T\rightarrow S$ with a point $P\in \textup{ker}(\pi)(T)$ such that the following equality of Cartier divisors on $E_{T}$ holds:
\begin{equation*}
    \textup{ker}(\pi)_{T} = \sum\limits_{a=1}^{N}\,[aP].
\end{equation*}
A $\Gamma_{0}(N)$-structure on an elliptic curve $E/S$ is a cyclic $N$-isogeny $E\stackrel{\pi}\longrightarrow E^{\prime}$.
\end{definition}
\begin{lemma}
Let $\pi:E\rightarrow E^{\prime}$ be an isogeny between two elliptic curves $E$ and $E^{\prime}$, the isogeny $\pi$ is a $N$-cyclic isogeny if and only if $\textup{ker}(\pi)$ is a cyclic group scheme of order $N$.
\end{lemma}
\begin{proof} By \cite[Theorem 1.10.1]{KM85}, the set $\{aP\}_{a=1}^{N}$ (where $P\in \textup{ker}(\pi)(S)$) forms a full set of sections of $\textup{ker}(\pi)$ if and only if we have the following equality of effective Cartier divisors in $E/S$,
\begin{equation*}
    \textup{ker}(\pi) = \sum\limits_{a=1}^{N}\,[aP],
\end{equation*}
which is exactly the definition of the cyclicity of a $N$-isogeny.
\end{proof}
\begin{example}
(a) Suppose $\tau=x+iy\in\mathbb{H}_{1}^{+}$, we consider the elliptic curve $E_{\tau}=\mathbb{C}/\mathbb{Z}+\mathbb{Z}\tau$, and a finite subgroup $K$ generated by $1/N$ inside $E_{\tau}$, then $\pi:E_{\tau}\rightarrow E_{\tau}/K$ is a cyclic isogeny.\\
(b) Suppose $E/S$ is an elliptic curve over a $\mathbb{F}_{p}$-scheme $S$, then for any $n\geq1$, the $n^{\textup{th}}$ iterated relative Frobenius
\begin{equation*}
    F^{n}: E\rightarrow E^{(p^{n})},
\end{equation*}
is a cyclic $p^{n}$-isogeny. The origin $P=0$ is a generator of $\textup{ker}(F^{n})$ because $\textup{ker}(F^{n})\simeq\mathcal{O}_{S}[T]/(T^{p^{n}})$ Zariski locally (cf. \cite[Lemma 12.2.1]{KM85}).
\end{example}
\par
Let $\mathcal{E}ll$ be the stack of elliptic curves, i.e., for an arbitrary scheme $S$, $\mathcal{E}ll(S)$ is a groupoid whose objects are elliptic curves $p: E\rightarrow S$ and morphisms are isomorphisms of elliptic curves over $S$. We use $\mathcal{Y}_{0}(N)$ to denote the stack which consists of all the $\Gamma_{0}(N)$-structures on elliptic curves, i.e., for a scheme $S$, $\mathcal{Y}_{0}(N)(S)$ is a groupoid whose objects are cyclic $N$-isogenies $(E\stackrel{\pi}\rightarrow E^{\prime})$ where $E$ and $E^{\prime}$ are elliptic curves over $S$, a morphism between two cyclic isogenies $(E_{1}\stackrel{\pi_{1}}\rightarrow E_{1}^{\prime})$ and $(E_{2}\stackrel{\pi_{2}}\rightarrow E_{2}^{\prime})$ is a pair of isomorphisms of elliptic curves $a:E_{1}\stackrel{\sim}\rightarrow E_{2}$ and $a^{\prime}:E_{1}^{\prime}\stackrel{\sim}\rightarrow E_{2}^{\prime}$ such that $a^{\prime}\circ\pi_{1}  = \pi_{2}\circ a$. We have the following functors,
\begin{align*}
     s:\mathcal{Y}_{0}(N)&\longrightarrow\mathcal{E}ll,\\
    (E/S\stackrel{\pi}\rightarrow E^{\prime}/S)&\longmapsto E/S.
\end{align*}
\par
\begin{lemma}
Both $\mathcal{Y}_{0}(N)$ and $\mathcal{E}ll$ are 2-dimensional Deligne-Mumford stacks. The functor $s:\mathcal{Y}_{0}(N)\rightarrow \mathcal{E}ll$ is finite flat of degree $\psi(N)=N\cdot\prod\limits_{l\vert N}(1+l^{-1})$, and representable by schemes, $s$ is also étale over $\textup{Spec}\,\mathbb{Z}[1/N]$.
\label{fiet}
\end{lemma}
\begin{proof} This is proved in \cite[Theorem 5.1.1]{KM85}. The key input is that a finite order group scheme is automatically étale if the order is invertible in the base scheme.\end{proof}
\par
For a $\mathbb{Z}_{(p)}$-scheme $S$, a geometric point $\overline{s}$ of $S$ and an elliptic curve $E$ over $S$. Let $E_{\overline{s}}$ be the base change of $E$ to $\overline{s}$. Let $T^{p}(E_{\overline{s}})$ (resp. $V^{p}(E_{\overline{s}})$) be the integral (resp. rational) Tate module of the elliptic curve $E_{\overline{s}}$. A $\mathbb{Z}_{(p)}^{\times}$-isogeny $f:E\rightarrow E^{\prime}$ over $S$ is a quasi-isogeny and there exists a prime-to-$p$ number $M$, such that $M\circ f$ is an isogeny. Let $V^{p}(f)$ be the homomorphism on rational Tate modules induced by $f$.
\begin{lemma}
Let $\mathcal{E}ll_{(p)}$ be the localization of the stack $\mathcal{E}ll$ to $\textup{Spec}\,\mathbb{Z}_{(p)}$. Then $\mathcal{E}ll_{(p)}$ can be described by the following stack: for every $\mathbb{Z}_{(p)}$-scheme S, $\mathcal{E}ll_{(p)}(S)$ is a groupoid whose objects are pairs $(E/S,\overline{\eta^{p}})$, where $\overline{\eta^{p}}$ is a $\pi_{1}(S,\overline{s})$-invariant $\textup{GL}_{2}(\hat{\mathbb{Z}}^{p})$-equivalence class of isomorphism
\begin{equation*}
    \eta^{p}: V^{p}(E_{\overline{s}})\stackrel{\simeq}\longrightarrow (\mathbb{A}_{f}^{p})^{2}
\end{equation*}
A morphism between two objects $(E/S,\overline{\eta^{p}})$ and $(E^{\prime}/S,\overline{\eta^{\prime p}})$ is a $\mathbb{Z}_{(p)}^{\times}$-isogeny $f: E\rightarrow E^{\prime}$ over S such that $\overline{\eta^{p}}=\overline{V^{p}(f)\circ \eta^{\prime p}}$.
\end{lemma}
\begin{proof} 
We temporarily use $\mathcal{E}ll^{\prime}$ to denote the stack described in the lemma. It suffices to show that for a connected scheme $S$ over $\textup{Spec}\,\mathbb{Z}_{(p)}$, there is a category equivalence between $\mathcal{E}ll(S)$ and $\mathcal{E}ll^{\prime}(S)$. We first construct a functor $F$ from $\mathcal{E}ll(S)$ to $\mathcal{E}ll^{\prime}(S)$. Given an elliptic curve $E$ over $S$, and a geometric point $\overline{s}$ of $S$, we choose an isomorphism
\begin{equation*}
    \eta^{p}: T^{p}(E_{\overline{s}})\simeq (\hat{\mathbb{Z}}^{p})^{2}
\end{equation*}
then clearly the $\textup{GL}_{2}(\hat{\mathbb{Z}}^{p})$-orbit of $\overline{\eta^{p}}$ is $\pi_{1}(S,\overline{s})$-invariant (because $\pi_{1}(S,\overline{s})$ acts linear on $T^{p}(E_{\overline{s}})$). We define $F(E)=(E, \overline{\eta^{p}})$, this functor is independent of the choice of $\eta^{p}$.\\
\indent
Now we prove that this functor is essentially surjective and fully faithful. For essential surjectivity, we pick an arbitrary object $(E/S, \overline{\eta^{p}})$ of $\mathcal{E}ll^{\prime}(S)$, by the work of Lan \cite[Corollary 1.3.5.4]{Lan13}, there is a $\mathbb{Z}_{(p)}^{\times}$-isogeny $f: E^{\prime}\rightarrow E$, such that $\eta^{\prime p}=\eta^{p}\circ V^{p}(f): V^{p}(E_{\overline{s}}^{\prime})\stackrel{\simeq}\longrightarrow (\mathbb{A}_{f}^{p})^{2}$ maps $T^{p}(E_{\overline{s}}^{\prime})$ to $(\hat{\mathbb{Z}}^{p})^{2}$. Therefore the object $(E/S, \overline{\eta^{p}})$ is isomorphic to $(E^{\prime}/S, \overline{\eta^{\prime p}})$, which is the essential image of $E^{\prime}\in Ob\,\mathcal{E}ll(S)$.\\
\indent
Next we show that there is an isomorphism,
\begin{equation}
    \textup{Hom}_{\mathcal{E}ll(S)}(E, E^{\prime})\simeq \textup{Hom}_{\mathcal{E}ll^{\prime}(S)}((E,\overline{\eta^{p}}), (E^{\prime},\overline{\eta^{\prime p}}))
    \label{ff}
\end{equation}
This is clearly injective by the above discussion. Now we pick an arbitrary element $f$ from the right hand side, then $f$ is a $\mathbb{Z}_{(p)}^{\times}$-isogeny, and $\eta^{\prime p}=\eta^{p}\circ V^{p}(f)$. There exists an integer $M$ prime to $p$, such that $\Tilde{f}=M\circ f$ is an isogeny from $E$ to $E^{\prime}$. We claim that this isogeny factors through the multiplication-by-$M$ map, i.e., $f$ itself is an isogeny. By the relation $\eta^{\prime p}=\eta^{p}\circ V^{p}(f)$ and the construction above, $V^{p}(f)$ maps $T^{p}(E_{\overline{s}})$ isomorphically to $T^{p}(E_{\overline{s}}^{\prime})$, then obviously $\Tilde{f}$ maps $E^{\prime}_{\overline{s}}[M]\simeq E^{\prime}[M]_{\overline{s}}$ to $0$, this holds for every geometric point $\overline{s}$ of $S$, then since $S$ is a $\mathbb{Z}_{(p)}$-scheme and rigidity result proved by Mumford and Fogarty \cite[Proposition 6.1]{MF82}, we know the isogeny $\Tilde{f}$ vanishes on $E^{\prime}[M]$, hence $f$ itself is an isogeny. Now $\textup{ker}(f)$ is a finite flat group scheme over $S$ of order prime to $p$, but since $V^{p}(f)$ maps $T^{p}(E_{\overline{s}})$ isomorphically to $T^{p}(E_{\overline{s}}^{\prime})$, this group scheme must be trivial, i.e., $f$ is an isomorphism, therefore it comes from an element of the left hand side of (\ref{ff}).
\end{proof}
\begin{remark}
We consider the following Deligne-Mumford stack, 
\begin{equation*}
    \mathcal{H}=\mathcal{E}ll\times_{\mathbb{Z}}\mathcal{E}ll
\end{equation*}
it parametrizes equivalence classes of pairs of elliptic curves. For any prime $p$, we use $\mathcal{H}_{(p)}$ to denote the localization of $\mathcal{H}$ to $\textup{Spec}\,\mathbb{Z}_{(p)}$. 
\par
There is a similar description of the stack $\mathcal{H}_{(p)}$ as follows, for any $\mathbb{Z}_{(p)}$-scheme $S$, the groupoid $\mathcal{H}_{(p)}(S)$ consists of pairs $((E,E^{\prime}),(\overline{\eta^{p}}, \overline{\eta^{\prime p}}))$, where $\overline{\eta^{p}}$ (resp. $\overline{\eta^{\prime p}}$) is a $\pi_{1}(S,\overline{s})$-invariant $\textup{GL}_{2}(\hat{\mathbb{Z}}^{p})$-equivalence class of isomorphism $V^{p}(E_{\overline{s}})\stackrel{\simeq}\longrightarrow (\mathbb{A}_{f}^{p})^{2}$ (resp. $V^{p}(E_{\overline{s}}^{\prime})\stackrel{\simeq}\longrightarrow (\mathbb{A}_{f}^{p})^{2}$). 
\label{HHH}
\end{remark}
\par
For any $N\in\mathbb{Z}_{>0}$, let $w_{N}$ be the following $2\times2$ matrix,
\begin{equation*}
    w_{N} = \begin{pmatrix}
    N & 0\\
    0 & 1
    \end{pmatrix}.
\end{equation*}
We consider the following stack $\mathcal{Y}_{0}(N)_{(p)}^{\prime}$ over $\textup{Spec}\,\mathbb{Z}_{(p)}$: for every $\mathbb{Z}_{(p)}$-scheme $S$, $\mathcal{Y}_{0}(N)_{(p)}^{\prime}(S)$ is a groupoid whose objects are pairs $(E\stackrel{\pi}\rightarrow E^{\prime},\overline{(\eta^{p},\eta^{\prime p})})$, where $E\stackrel{\pi}\rightarrow E^{\prime}$ is a cyclic $N$-isogeny and  $\overline{(\eta^{p},\eta^{\prime p})}$ is a pair of $\pi_{1}(S,\overline{s})$-invariant $\Gamma_{0}(N)(\hat{\mathbb{Z}}^{p})$-equivalence class (we will specify the action of $\Gamma_{0}(N)(\hat{\mathbb{Z}}^{p})$ in (\ref{action})) of isomorphisms
\begin{equation*}
    \eta^{p}: V^{p}(E_{\overline{s}})\stackrel{\simeq}\longrightarrow (\mathbb{A}_{f}^{p})^{2},\,\,\,\eta^{\prime p}: V^{p}(E_{\overline{s}}^{\prime})\stackrel{\simeq}\longrightarrow (\mathbb{A}_{f}^{p})^{2}.
\end{equation*}
which maps $T^{p}(E_{\overline{s}})$ and $T^{p}(E_{\overline{s}}^{\prime})$ to $(\hat{\mathbb{Z}}^{p})^{2}$, and the isomorphism $\eta^{\prime p}$ is determined by $\eta^{p}$ by the following commutative diagram,
\begin{equation}
    \xymatrix{
    V^{p}(E_{\overline{s}})\ar[d]^{V^{p}(\pi)}\ar[r]^{\eta^{p}}&(\mathbb{A}_{f}^{p})^{2}\ar[d]^{w_{N}}\\
    V^{p}(E_{\overline{s}}^{\prime})\ar[r]^{\eta^{\prime p}}&(\mathbb{A}_{f}^{p})^{2}_{.}}
    \label{diagram}
\end{equation}
\par
A morphism from $(E_{1}\stackrel{\pi_{1}}\rightarrow E_{1}^{\prime},\overline{(\eta_{1}^{p},\eta_{1}^{\prime p})})$ to $(E_{2}\stackrel{\pi_{2}}\rightarrow E_{2}^{\prime},\overline{(\eta_{2}^{p},\eta_{2}^{\prime p})})$ is a pair $(f,f^{\prime})$ of isomorphisms $f:E_{1}\rightarrow E_{2}$ and $f^{\prime}:E_{1}^{\prime}\rightarrow E_{2}^{\prime}$ such that $f^{\prime}\circ \pi_{1}=\pi_{2}\circ f$ and $\overline{(\eta_{1}^{p},\eta_{1}^{\prime p})} = \overline{(\eta_{2}^{p}\circ V^{p}(f),\eta_{2}^{\prime p}\circ V^{p}(f^{\prime}))}$ as $\Gamma_{0}(N)(\hat{\mathbb{Z}}^{p})$-orbits. The action of $\Gamma_{0}(N)(\hat{\mathbb{Z}}^{p})$ on the pair $(\eta^{p},\eta^{\prime p})$ is given by
\begin{equation}
    g\cdot((\eta^{p},\eta^{\prime p})) = (g\circ\eta^{p},w_{N}gw_{N}^{-1}\circ\eta^{\prime p}).
    \label{action}
\end{equation}
\begin{lemma}
Let $\mathcal{Y}_{0}(N)_{(p)}$ be the localization of $\mathcal{Y}_{0}(N)$ to $\mathbb{Z}_{(p)}$. There is an isomorphism $G:\mathcal{Y}_{0}(N)_{(p)}\rightarrow\mathcal{Y}_{0}(N)_{(p)}^{\prime}$ of stacks over $\textup{Spec}\,\mathbb{Z}_{(p)}$.
\label{YYY}
\end{lemma}
\begin{proof}
Let $S$ be a scheme over $\textup{Spec}\,\mathbb{Z}_{(p)}$, and an object $(E\stackrel{\pi}\rightarrow E^{\prime})$ in the groupoid $\mathcal{Y}_{0}(N)_{(p)}(S)$. Let $\overline{s}$ is a geometric point of $S$, the cyclicity of $\pi$ implies that $\pi_{\overline{s}}$ is also cyclic, since $l$ is invertible in $\textup{Spec}\,\kappa(\overline{s})$ if $l\neq p$, there exist isomorphisms $\eta^{p}: T^{p}(E_{\overline{s}})\simeq (\hat{\mathbb{Z}}^{p})^{2}$ and $\eta^{\prime p}: T^{p}(E^{\prime}_{\overline{s}})\simeq (\hat{\mathbb{Z}}^{p})^{2}$ such that $\omega_{N}\circ\eta^{p}=\eta^{\prime p}\circ T^{p}(\pi)$. Now we consider a different choice of $(\eta^{p},\eta^{\prime p})$, say $(\tilde{\eta}^{p},\tilde{\eta}^{\prime p})$, satisfying the above conditions. Then $\tilde{\eta}^{\prime p}$ differs $\eta^{p}$ by an element $g\in\textup{GL}_{2}(\hat{\mathbb{Z}}^{p})$, i.e. $\tilde{\eta}^{p}=g\circ \eta^{p}$, correspondingly $\tilde{\eta}^{\prime p}=\omega_{N}g\omega_{N}^{-1}\circ \eta^{\prime p}$. However, $\omega_{N}g\omega_{N}^{-1}\in \textup{GL}_{2}(\hat{\mathbb{Z}}^{p})$ since both $\eta^{\prime p}$ and $\tilde{\eta}^{\prime p}$ give isomorphisms from $T^{p}(E_{\overline{s}})$ to $(\hat{\mathbb{Z}}^{p})^{2}$, therefore $g\in\textup{GL}_{2}(\hat{\mathbb{Z}}^{p})\cap \omega_{N}^{-1}\textup{GL}_{2}(\hat{\mathbb{Z}}^{p})\omega_{N}=\Gamma_{0}(N)(\hat{\mathbb{Z}}^{p})$, thus the $\Gamma_{0}(N)(\hat{\mathbb{Z}}^{p})$-orbit $\overline{(\eta^{p},\eta^{\prime p})}$ is well-defined. We define $G((E\stackrel{\pi}\rightarrow E^{\prime}))=((E\stackrel{\pi}\rightarrow E^{\prime}),\overline{(\eta^{p},\eta^{\prime p})})$. For a pair of isomorphisms $(f,f^{\prime})$, where $f:E_{1}\rightarrow E_{1}^{\prime}$ and $f^{\prime}:E_{2}\rightarrow E_{2}^{\prime}$, define $G((f,f^{\prime}))=(f,f^{\prime})$.
\par
It suffices to show that for a connected scheme $S$ over $\textup{Spec}\,\mathbb{Z}_{(p)}$, the functor $G(S): \mathcal{Y}_{0}(N)_{(p)}(S)\rightarrow\mathcal{Y}_{0}(N)_{(p)}(S)$ is an equivalence of categories. This functor is essentially surjective by definition, now we show that it is fully faithful, i.e., the following morphism between sets is bijective,
\begin{align*}
    \textup{Hom}_{\mathcal{Y}_{0}(N)_{(p)}(S)}((E_{1}\stackrel{\pi_{1}}\rightarrow E_{1}^{\prime}),((E_{2}\stackrel{\pi_{2}}\rightarrow E_{2}^{\prime})))&\stackrel{G}\longrightarrow \textup{Hom}_{\mathcal{Y}_{0}(N)_{(p)}(S)}((E_{1}\stackrel{\pi_{1}}\rightarrow E_{1}^{\prime},\overline{(\eta_{1}^{p},\eta_{1}^{\prime p})}),(E_{2}\stackrel{\pi_{2}}\rightarrow E_{2}^{\prime},\overline{(\eta_{2}^{p},\eta_{2}^{\prime p})})),\\
    (f,f^{\prime})&\longmapsto (f,f^{\prime}).
\end{align*}
but this is clearly bijective by the definition.
\end{proof}
There is a natural morphism from $\mathcal{Y}_{0}(N)_{(p)}$ to $\mathcal{H}_{(p)}$, i.e., $(E\stackrel{\pi}\longrightarrow E^{\prime})\longrightarrow (E,E^{\prime})$. By Remark \ref{HHH} and Lemma \ref{YYY}, we can also describe it as follows
\begin{align}
    \mathcal{Y}_{0}(N)_{(p)}&\longrightarrow\mathcal{H}_{(p)},\notag\\
    (E\stackrel{\pi}\longrightarrow E^{\prime},\overline{(\eta^{p},\eta^{\prime p})})&\longmapsto((E,E^{\prime}),(\overline{\eta^{p}}, \overline{\eta^{\prime p}})).
    \label{keymap}
\end{align}\\
\subsubsection{Compactification of $\mathcal{Y}_{0}(N)$} Next we introduce the compactification of the moduli stack $\mathcal{Y}_{0}(N)$. Let $S$ be a scheme, we first introduce the notion of N$\acute{\textup{e}}$ron $n$-gons.
\begin{definition}
For any integer $n\geq1$ and an scheme $S$, the N$\acute{\textup{e}}$ron $n$-gon over $S$ is the coequalizer of
\begin{equation*}
    \bigsqcup\limits_{i\in\mathbb{Z}/n\mathbb{Z}}S\rightrightarrows \bigsqcup\limits_{i\in\mathbb{Z}/n\mathbb{Z}}\mathbb{P}_{S}^{1}.
\end{equation*}
where the top (resp. the bottom) closed immersion includes the $i^{\textup{th}}$ copy of $S$ as the 0 (resp. the $\infty$) section of the $i^{\textup{th}}$ (resp. $(i+1)^{\textup{st}}$) copy of $\mathbb{P}_{S}^{1}$.
\end{definition}
\begin{definition}
A generalized elliptic curve over a scheme $S$ is the data of \\
$\bullet$ A proper, flat, finitely presented morphism $E\rightarrow S$ each of whose geometric fibers is either a smooth connected curve of genus 1 or a N$\acute{\textup{e}}$ron $n$-gon for some $n\geq1$;\\
$\bullet$ An $S$-morphism $E^{\textup{sm}}\times_{S}E\stackrel{+}\rightarrow E$ that restricts to a commutative $S$-group scheme structure on $E^{\textup{sm}}$ for which $+$ becomes an $S$-group action such that via the pullback of line bundles the action $+$ induces the trivial action of $E^{\textup{sm}}$ on $\textup{Pic}^{0}_{E/S}$.
\end{definition}
We will use $\mathcal{X}$ to denote the moduli stack consisting of generalized elliptic curves whose degenerate fibers are N$\acute{\textup{e}}$ron $1$-gons, i.e., for a scheme $S$, $\mathcal{X}(S)$ is a groupoid whose objects are generalized elliptic curves $E$ over $S$ and whose geometric fibers are either elliptic curves or N$\acute{\textup{e}}$ron $1$-gons. The following result is proved in \cite{Ces17}.
\begin{lemma}
$\mathcal{X}$ is a proper smooth 2-dimensional Deligne-Mumford stack.
\label{smooth}
\end{lemma}
\begin{proof}
This is proved in Theorem 3.1.6 of \cite{Ces17}.
\end{proof}
We have a natural morphism of Deligne-Mumford stacks $\mathcal{E}ll\rightarrow\mathcal{X}$, which sends an elliptic curve $E$ over $S$ to itself. This morphism is an open immersion, i.e., the stack $\mathcal{E}ll$ is an open substack of the stack $\mathcal{X}$. Recall that we have a finite flat representable morphism $\mathcal{Y}_{0}(N)\rightarrow \mathcal{E}ll$ by Lemma \ref{fiet}, let $\mathcal{X}_{0}(N)$ be the normalization of $\mathcal{Y}_{0}(N)$ with respect to this morphism. A moduli description of $\mathcal{X}_{0}(N)$ in terms of level structures on the generalized elliptic curves can be found in \cite[section 5.9]{Ces17}. The stack $\mathcal{Y}_{0}(N)$ can be realized as an open substack of the stack $\mathcal{X}_{0}(N)$ based on this description. We also have the following theorem
\begin{theorem}
$\mathcal{X}_{0}(N)$ is a regular proper 2-dimensional Deligne-Mumford stack, it is finite flat over $\mathcal{X}$.
\end{theorem}
\begin{proof}
This is proved in Theorem 5.13 of \cite{Ces17}.
\end{proof}

\subsubsection{Integral models of GSpin Shimura varities}
Both the stacks $\mathcal{Y}_{0}(N)$ and $\mathcal{H}$ are examples of integral models of GSpin Shimura varieties, we will explain this and discuss the general case. This part will not be used in the following chapters, readers can skip on first reading.
\par
Let $L$ be an integral quadratic lattice over $\mathbb{Z}$ of finite rank $n\geq3$, i.e., it is equipped with a quadratic form $q_{L}$ valued in $\mathbb{Z}$. Let $(\cdot,\cdot)$ be the associated symmetric bilinear form, assume that the bilinear form is non-degenerate, we define $\textup{disc}(L)\coloneqq L^{\vee}/L$, where $L^{\vee}=\{x\in L\otimes\mathbb{Q}:(x,y)\in\mathbb{Z},\,\text{for any $y\in L$}\}$. Let $C\coloneqq C(L)$ be the associated Clifford algebra over $\mathbb{Z}$.
\par
Assume that the quadratic space $L\otimes\mathbb{R}$ over $\mathbb{R}$ has signature $(n-2,2)$. Let $G=\textup{GSpin}(L\otimes\mathbb{Q})$ and $X$ be the space of oriented negative definite planes in $L\otimes\mathbb{R}$. Let the level group $K(L)\coloneqq G(\mathbb{A}_{f})\cap C(L\otimes\hat{\mathbb{Z}})^{\times}$, where $C(L\otimes\hat{\mathbb{Z}})^{\times}$ is the unit group of the Clifford algebra of $L\otimes\hat{\mathbb{Z}}$. The $\mathbb{Q}$-algebraic stack $\textup{Sh}_{K(L)}(G,X)\coloneqq G(\mathbb{Q})\backslash [X\times G(\mathbb{A}_{f})/K(L)]$ is the Shimura variety associated to the pair $(G,X)$ with level structure $K(L)$, or simply speaking, associated to the quadratic lattice $L$.
\begin{theorem}
    Assume that the integer $\#\textup{disc}(L)$ is square-free. Then, over $\mathbb{Z}[\frac{1}{2}]$, the Shimura variety $\textup{Sh}_{K(L)}(G,X)$ admits a regular canonical integral model $\mathscr{S}(L)$ with a regular compactification.
    \label{integralmodel}
\end{theorem}
Theorem \ref{integralmodel} is the main result of Madapusi \cite{Mad16}. Roughly speaking, the integral model $\textup{Sh}_{K(L)}(G,X)$ is constructed by building a regular integral model $\mathscr{S}(L)_{(p)}$ of $\textup{Sh}_{K(L)}(G,X)$ over $\mathbb{Z}_{(p)}$ for every odd prime $p$. Let $L_{p}=L\otimes\mathbb{Z}_{p}$. The level group has the following decomposition: $K(L)=\prod\limits_{p}K_{p}$, where $K_{p}= G(\mathbb{Q}_{p})\cap C(L_{p})^{\times}$. When $p\nmid\textup{disc}(L)$, the quadratic lattice $L_{p}$ is self-dual and the group $K_{p}$ is hyperspecial, Kisin \cite{Kis10} has already constructed a smooth integral canonical model $\mathscr{S}(L)_{(p)}$ of $\textup{Sh}_{K(L)}(G,X)$ over $\mathbb{Z}_{(p)}$. In general, we can embed $L$ into a bigger quadratic lattice $\tilde{L}$ of rank $n+1$ over $\mathbb{Z}$ that is self-dual at $p$ and the associated quadratic space $\tilde{L}\otimes{\mathbb{R}}$ over $\mathbb{R}$ has signature $(n-1, 2)$. Let $\tilde{G}=\textup{GSpin}(\tilde{L}\otimes\mathbb{Q})$ and $\tilde{X}$ be the space of oriented negative definite planes in $\tilde{L}\otimes{\mathbb{R}}$. Let the level group $K(\tilde{L})\coloneqq \tilde{G}(\mathbb{A}_{f})\cap C(\tilde{L}\otimes\hat{\mathbb{Z}})^{\times}$, here again $C(\tilde{L}\otimes\hat{\mathbb{Z}})^{\times}$ is the unit group of the Clifford algebra of $\tilde{L}\otimes\hat{\mathbb{Z}}$. Let $\textup{Sh}_{K(\tilde{L})}(\tilde{G},\tilde{X})$ be the Shimura variety attached to the pair $(\tilde{G},\tilde{X})$ with level structure $K(\tilde{L})$, it has a smooth integral model $\mathscr{S}(\tilde{L})_{(p)}$ over $\mathbb{Z}_{(p)}$ for every odd prime $p$ since the quadratic lattice $\tilde{L}\otimes\mathbb{Z}_{p}$ is self-dual. Then the regular integral model $\mathscr{S}(L)_{(p)}$ of $\textup{Sh}_{K(L)}(G,X)$ is constructed as a divisor on $\mathscr{S}(\tilde{L})_{(p)}$ (cf. \cite[Chapter 7]{Mad16}). We will explain Theorem \ref{integralmodel} and constructions above by the following example.
\par
Let $\tilde{L}=\textup{M}_{2}(\mathbb{Z})$ be the rank 4 lattice equipped with the quadratic form $x\mapsto \textup{det}(x)$. It can be checked immediately that the quadratic lattice $\tilde{L}$ is self-dual. Let $N$ be a positive integer. Recall that $w_{N} = \textup{diag}\{N,1\}\in\tilde{L}$. Let $\Delta(N)=\{w_{N}\}^{\bot}\subset\tilde{L}$, then $\Delta(N)$ is the following quadratic lattice of rank 3 over $\mathbb{Z}$,
\begin{equation}
    \Delta(N) = \left\{x=\begin{pmatrix}
    -Na & b\\
    c & a
    \end{pmatrix}:\,\, a,b,c\in\mathbb{Z}\right\}.
\end{equation}
Similiar computations in Example \ref{lambda} implies that $\textup{disc}(\Delta(N))\simeq\mathbb{Z}/2N$. When $N$ is odd and square-free, the Shimura vairty associated to the quadratic lattice $\Delta(N)$ admits a regular canonical model with a regular compactification by Theorem \ref{integralmodel}. The following lemma will imply that the integral model is $\mathcal{Y}_{0}(N)$, and the regular compactification is $\mathcal{X}_{0}(N)$.
\begin{lemma}
The following diagram is commutative,
\begin{equation*}
     \xymatrix{
    \textup{GSpin}(\Delta(N)\otimes\mathbb{Q})\ar[r]\ar[d]_{\iota}&\textup{GSpin}(\tilde{L}\otimes\mathbb{Q})\ar[d]^{\tilde{\iota}}\\
    \textup{GL}_{2}\ar[r]&\textup{GL}_{2}\times_{\mathbb{G}_{m}}\textup{GL}_{2}.}
\end{equation*}
here $\iota$ and $\tilde{\iota}$ are isomorphisms of algebraic groups, and the bottom arrow is given by $g\in\textup{GL}_{2}\mapsto (g,w_{N}gw_{N}^{-1})\in \textup{GL}_{2}\times_{\mathbb{G}_{m}}\textup{GL}_{2}$. An element $g\in\textup{GSpin}(\Delta(N)\otimes\mathbb{Q})$ acts on the quadratic space $\Delta(N)\otimes\mathbb{Q}$ by the following formula,
\begin{equation*}
    g(x)=w_{N}\iota(g)w_{N}^{-1}\cdot x\cdot \iota(g^{-1}),\,\,\,\forall x\in \Delta(N)\otimes\mathbb{Q}.
\end{equation*}
Moreover, for every prime number $p$, we consider the following subgroup of $\textup{GL}_{2}(\mathbb{Z}_{p})$,
\begin{equation*}
    \Gamma_{0}(N)(\mathbb{Z}_{p}) = \left\{x=\begin{pmatrix}
    a & b\\
    Nc & d
    \end{pmatrix}\in \textup{GL}_{2}(\mathbb{Z}_{p})\,\,:\,\,a,b,c,d\in\mathbb{Z}_{p}\right\}.
\end{equation*}
then the isomorphism $\iota$ induces the following isomorphism,
\begin{equation*}
    \textup{GSpin}(\delta_{p}(N)\otimes\mathbb{Q}_{p})\cap C(\delta_{p}(N))^{\times}\simeq \Gamma_{0}(N)(\mathbb{Z}_{p}),
\end{equation*}
and the isomorphism $\tilde{\iota}$ induces the following isomorphism,
\begin{equation*}
    \textup{GSpin}(\tilde{L}\otimes\mathbb{Q}_{p})\cap C(\tilde{L}\otimes\mathbb{Z}_{p})^{\times}\simeq \textup{GL}_{2}(\mathbb{Z}_{p})\times_{\mathbb{Z}_{p}^{\times}}\textup{GL}_{2}(\mathbb{Z}_{p}).
\end{equation*}
\label{cliffordal}
\end{lemma}
\begin{proof}
We choose the following basis of $\Delta(N)$ over $\mathbb{Z}$, 
\begin{equation*}
    e_{1}=\begin{pmatrix}
    -N & \\
     & 1
    \end{pmatrix},\,\,e_{2}=\begin{pmatrix}
    \,  & 1\\
     \, & 
    \end{pmatrix},\,\,e_{3}=\begin{pmatrix}
     \, & \,\\
     1& \,
    \end{pmatrix}.
\end{equation*}
the Clifford algebra $C(\Delta(N)\otimes\mathbb{Q})$ is an associated algebra over $\mathbb{Q}$ generated by $e_{1}.e_{2},e_{3}$ subject to the relations that $v\cdot v=(v,v)$ for any $v\in\Delta(N)\otimes\mathbb{Q}$ where $(\cdot,\cdot)$ is the quadratic form on $\Delta(N)\otimes\mathbb{Q}$. The even part of the $C(\Delta(N)\otimes\mathbb{Q})^{+}$ is spanned over $\mathbb{Q}$ by
\begin{equation*}
    1,\,\,u_{1}=e_{1}e_{2},\,\,u_{2}=e_{1}e_{3},\,\,u_{3}=e_{2}e_{3},
\end{equation*}
and the even part of the $C(\tilde{L}\otimes\mathbb{Q})^{+}$ is spanned over $\mathbb{Q}$ by
\begin{equation*}
    1,\,\,u_{1},\,\,u_{2},\,\,u_{3},\,\,w_{N}e_{1},\,\,w_{N}e_{2},\,\,w_{N}e_{3},\,\,\lambda=w_{N}e_{1}(e_{2}+e_{3})(e_{2}-e_{3}). 
\end{equation*}
We easily verify that $\lambda^{2}=N^{2}$. By the work of Kudla and Rapoport \cite[Chapter 0]{KR99}, there is an algebra isomorphism $C(\tilde{L}\otimes\mathbb{Q})^{+}\simeq C(\Delta(N)\otimes\mathbb{Q})^{+}\otimes\mathbb{Q}(\lambda)$, where $\mathbb{Q}(\lambda)\simeq\mathbb{Q}\times\mathbb{Q}$ is the center of $C(\tilde{L}\otimes\mathbb{Q})^{+}$. The following is an isomorphism of algebras over $\mathbb{Q}$,
\begin{align*}
    \iota: C(\Delta(N)\otimes\mathbb{Q})^{+}&\longrightarrow \textup{M}_{2}(\mathbb{Q}),\\
    1\mapsto 1_{2}\coloneqq\begin{pmatrix}
    1 & \\
    \, & 1
    \end{pmatrix},\,\,
    u_{1}\mapsto \begin{pmatrix}
     & 1\\
    \, & 
    \end{pmatrix},\,\,u_{2}&\mapsto \begin{pmatrix}
     & \,\\
    -N & \,
    \end{pmatrix},\,\,u_{3}\mapsto \begin{pmatrix}
    -1 & \\
     & \,
    \end{pmatrix},
\end{align*}
under this isomorphism, the main involution $g\mapsto g^{t}$ on the algebra $C(\Delta(N)\otimes\mathbb{Q})^{+}$ transfers to the adjoint involution of $\textup{M}_{2}(\mathbb{Q}): g\mapsto g^{\ast}$, here $g^{\ast}$ is the adjoint matrix of $g$. Therefore $\iota$ induces an isomorphism of algebras $\iota\otimes\mathbb{Q}(\lambda):C(\tilde{L}\otimes\mathbb{Q})^{+}\simeq \textup{M}_{2}(\mathbb{Q})\times \textup{M}_{2}(\mathbb{Q})$, where the element $\lambda$ is mapped to $(-N\cdot1_{2},N\cdot1_{2})$. The main involution $\tilde{g}\mapsto\tilde{g}^{t}$ on the algebra $C(\tilde{L}\otimes\mathbb{Q})^{+}$ transfers to the following involution of $\textup{M}_{2}(\mathbb{Q})\times \textup{M}_{2}(\mathbb{Q}):(g_{1},g_{2})\mapsto (g_{1}^{\ast},g_{2}^{\ast})$. By definition of the spinor group, for any $\mathbb{Q}$-algebra $R$,
\begin{equation*}
    \textup{GSpin}(\Delta(N)\otimes\mathbb{Q})(R)=\{g\in C(\Delta(N)\otimes\mathbb{Q})^{+}\otimes_{\mathbb{Q}}R: g\cdot g^{t}\in R^{\times}\},
\end{equation*}
\begin{equation*}
    \textup{GSpin}(\tilde{L}\otimes\mathbb{Q})(R)=\{\tilde{g}\in C(\tilde{L}\otimes\mathbb{Q})^{+}\otimes_{\mathbb{Q}}R: \tilde{g}\cdot\tilde{g}^{t}\in R^{\times}\}.
\end{equation*}
Therefore $\iota$ induces isomorphisms of algebraic groups over $\mathbb{Q}$, $\iota:\textup{GSpin}(\Delta(N)\otimes\mathbb{Q})\simeq \textup{GL}_{2}$ and $\iota\otimes\mathbb{Q}(\lambda): \textup{GSpin}(\tilde{L}\otimes\mathbb{Q})\simeq\textup{GL}_{2}\times_{\mathbb{G}_{m}}\textup{GL}_{2}$.
\par
By definition, $\textup{GSpin}(\delta_{p}(N)\otimes\mathbb{Q}_{p})\cap C(\delta_{p}(N))^{\times}$ is the subgroup of $\textup{GSpin}(\delta_{p}(N)\otimes\mathbb{Q}_{p})$ whose elements are $\mathbb{Z}_{p}$-linear combinations of $1,u_{1},u_{2},u_{3}$, and invertible. Therefore we have the following identification under the isomorphism $\iota$,
\begin{equation*}
    \textup{GSpin}(\delta_{p}(N)\otimes\mathbb{Q}_{p})\cap C(\delta_{p}(N))^{\times}\simeq\left\{g=\begin{pmatrix}
    a & b\\
    -Nc & d
    \end{pmatrix}\in \textup{GL}_{2}(\mathbb{Z}_{p})\,\,:\,\,a,b,c,d\in\mathbb{Z}_{p}\right\}=\Gamma_{0}(N)(\mathbb{Z}_{p}).
\end{equation*}
Let $\textup{Ad}_{w_{N}}$ be the conjugation action on $\textup{M}_{2}(\mathbb{Q})$, i.e., $\textup{Ad}_{w_{N}}(g)=w_{N}gw_{N}^{-1}$ for $g\in\textup{M}_{2}(\mathbb{Q})$. Let $\tilde{\iota}=(\textup{id}\times\textup{Ad}_{w_{N}})\circ(\iota\otimes\mathbb{Q}(\lambda))$. By definition, $\textup{GSpin}(\tilde{L}\otimes\mathbb{Q}_{p})\cap C(\tilde{L})^{\times}$ is the subgroup of $\textup{GSpin}(\tilde{L}\otimes\mathbb{Q}_{p})$ whose elements are $\mathbb{Z}_{p}$-linear combinations of $1,\frac{N-w_{N}e_{1}}{2N}, \frac{w_{N}e_{2}+u_{1}}{2}, \frac{w_{N}e_{3}+u_{2}}{2},\frac{w_{N}e_{2}-u_{1}}{2N},\frac{w_{N}e_{3}-u_{2}}{2N},u_{3}$, and invertible. Under the isomorphism $\tilde{\iota}$,
\begin{equation*}
    \textup{GSpin}(\tilde{L}\otimes\mathbb{Q}_{p})\cap C(\tilde{L}\otimes\mathbb{Z}_{p})^{\times}\simeq \textup{GL}_{2}(\mathbb{Z}_{p})\times_{\mathbb{Z}_{p}^{\times}}\textup{GL}_{2}(\mathbb{Z}_{p}).
\end{equation*}
\end{proof}
By Lemma \ref{cliffordal}, the level group $K(\tilde{L})\simeq\prod\limits_{p<\infty}\left(\textup{GL}_{2}(\mathbb{Z}_{p})\times_{\mathbb{Z}_{p}^{\times}}\textup{GL}_{2}(\mathbb{Z}_{p})\right)$ and $\linebreak K(\Delta(N))\simeq\prod\limits_{p<\infty}\Gamma_{0}(N)(\mathbb{Z}_{p})$. These two isomorphisms are compatible with the homomorphism $(\textup{id},\textup{Ad}_{w_{N}}):\textup{GL}_{2}\rightarrow\textup{GL}_{2}\times_{\mathbb{G}_{m}}\textup{GL}_{2}$. The GSpin Shimura variety associated to $\tilde{L}$ is $\textup{Sh}_{K(\tilde{L})}(\textup{GL}_{2}\times_{\mathbb{G}_{m}}\textup{GL}_{2},\mathbb{H}_{1}^{\pm}\times\mathbb{H}_{1}^{\pm})\simeq\mathcal{Y}_{0}(1)_{\mathbb{C}}\times\mathcal{Y}_{0}(1)_{\mathbb{C}}\simeq\mathcal{H}_{\mathbb{C}}$, the smooth Deligne-Mumford stack $\mathcal{H}=\mathcal{E}ll\times_{\mathbb{Z}}\mathcal{E}ll$ is the canonical integral model $\mathscr{S}(\tilde{L})$ of this Shimura variety, its compactification is $\mathcal{X}\times_{\mathbb{Z}}\mathcal{X}$, which is also a smooth Deligne-Mumford stack by Lemma \ref{smooth}. The GSpin Shimura variety associated to $\Delta(N)$ is $\textup{Sh}_{K(\Delta(N))}(\textup{GL}_{2},\mathbb{H}_{1}^{\pm})\simeq\mathcal{Y}_{0}(N)_{\mathbb{C}}$, the regular Deligne-Mumford stack $\mathcal{Y}_{0}(N)$ is the canonical integral model $\mathscr{S}(\Delta(N))$ of this Shimura variety, whose compactification is the regular Deligne-Mumford stack $\mathcal{X}_{0}(N)$. 
\begin{remark}
Although Theorem \ref{integralmodel} requires the integer $N$ to be odd and square-free, all the computations in Lemma \ref{cliffordal} and the discussions about the canonical integral model of the Shimura variety associated to the quadratic lattice $\Delta(N)$ don't need this requirement. Moreover, canonical integral models $\mathcal{H}$ and $\mathcal{Y}_{0}(N)$ for the lattice $\tilde{L}=\textup{M}_{2}(\mathbb{Z})$ and $\Delta(N)$ respectively are both regular over $\mathbb{Z}$, not only over $\mathbb{Z}[\frac{1}{2}]$.
\end{remark}
For every prime number $p$, the natural morphism from $\mathcal{Y}_{0}(N)_{(p)}$ to $\mathcal{H}_{(p)}$ is described in the formula (\ref{keymap}), this morphism realizes the regular integral model $\mathscr{S}(\Delta(N))_{(p)}\simeq \mathcal{Y}_{0}(N)_{(p)}$ as a divisor on the smooth integral model $\mathscr{S}(\tilde{L})_{(p)}\simeq \mathcal{H}_{(p)}$ (cf. \cite[6.12, 7.17]{Mad16}) when $N$ is odd and square-free.

\subsection{Special cycles on $\mathcal{H}$ and $\mathcal{X}_{0}(N)$}
Let $p$ be a prime number, we first define the special cycles on the stack $\mathcal{H}_{(p)}$.
\begin{definition}
For every symmetric $n\times n$ matrix $T=(T_{ik})$. Let $\tilde{\boldsymbol\varphi}^{p}$ be the characteristic function of an open compact subset $\tilde{\boldsymbol\omega}^{p}$ of $\textup{M}_{2}(\mathbb{A}_{f}^{p})^{n}$ invariant under the action of $\textup{GL}_{2}(\hat{\mathbb{Z}}^{p})\times \textup{GL}_{2}(\hat{\mathbb{Z}}^{p})$. We consider the stack $\mathcal{Z}^{\sharp}(T,\tilde{\boldsymbol\varphi}^{p})$, whose fibered category over a $\mathbb{Z}_{(p)}$-scheme $S$ consists of the following objects,
\begin{equation*}
   ((E,E^{\prime}), (\overline{\eta^{p}},\overline{\eta^{\prime p}}), \boldsymbol{j}),
\end{equation*}
where $((E, E^{\prime}), (\overline{\eta^{p}},\overline{\eta^{\prime p}}))$ is an object in $\mathcal{H}_{(p)}(S)$, $\boldsymbol{j}=(j_{1},j_{2},\cdot\cdot\cdot,j_{n})\in (\textup{Hom}(E,E^{\prime})\otimes_{\mathbb{Z}}\mathbb{Z}_{(p)})^{n}$ and $\eta^{\ast}(\boldsymbol{j})\coloneqq\eta^{\prime p}\circ  V^{p}(\boldsymbol{j})\circ(\eta^{p})^{-1}\in\tilde{\boldsymbol{\omega}}^{p}$. Moreover, 
\begin{equation*}
    T_{ik}=\frac{1}{2}\left(\textup{deg}(j_{i}+j_{k})-\textup{deg}(j_{i})-\textup{deg}(j_{k})\right)=\frac{1}{2}(j_{i}\circ j_{k}^{\vee}+j_{k}\circ j_{i}^{\vee}).
\end{equation*}
The special cycle $\mathcal{Z}^{\sharp}(T,\tilde{\boldsymbol{\varphi}}^{p})$ may be empty.
\end{definition}
For every symmetric $n\times n$ matrix $T$, we have a natural finite unramified morphism $i_{n}^{\sharp}: \mathcal{Z}^{\sharp}(T,\tilde{\boldsymbol{\varphi}}^{p})\rightarrow \mathcal{H}_{(p)}$ by forgetting the morphisms $\boldsymbol{j}$ of an object $((E,E^{\prime}), (\overline{\eta^{p}},\overline{\eta^{\prime p}}), \boldsymbol{j})$ of $\mathcal{Z}^{\sharp}(T,\tilde{\boldsymbol\varphi}^{p})$. We recall the following definition of generalized Cartier divisor appeared in the work of Howard and Madapusi \cite[Definition 2.4.1]{HM22}.
\begin{definition}
Suppose $D\rightarrow X$ is any finite, unramified, and relatively representable morphism of Deligne-Mumford stacks, then there is an étale cover $U\rightarrow X$ by a scheme such that the pullback $D_{U}\rightarrow U$ is a finite disjoint union
\begin{equation*}
    D_{U}=\bigsqcup\limits_{i}D_{U}^{i}
\end{equation*}
with each map $D_{U}^{i} \rightarrow U$ a closed immersion. If each of these closed immersions is an effective Cartier divisor on $U$ in the usual sense (the corresponding ideal sheaves are invertible), then we call $D \rightarrow X$ a generalized Cartier divisor.
\end{definition}
\begin{proposition}
    Let $\tilde{\varphi}^{p}$ be the characteristic function of an open compact subset $\tilde{\omega}^{p}$ of $\textup{M}_{2}(\mathbb{A}_{f}^{p})$ invariant under the action of $\textup{GL}_{2}(\hat{\mathbb{Z}}^{p})\times \textup{GL}_{2}(\hat{\mathbb{Z}}^{p})$. For any positive number $d\in\mathbb{Q}$, the finite unramified morphism $i_{1}^{\sharp}: \mathcal{Z}^{\sharp}(d,\tilde{\varphi}^{p})\rightarrow \mathcal{H}_{(p)}$ is a generalized Cartier divisor
    \label{divisorH}
\end{proposition}
\begin{proof}
    This is proved by Howard and Madapusi in \cite[Proposition 6.5.2]{HM20} (see also \cite[Proposition 2.4.3]{HM22}).
\end{proof}
Now let's come to the special cycles on the stack $\mathcal{X}_{0}(N)_{(p)}$ and $\mathcal{Y}_{0}(N)_{(p)}$, we first introduce the notion of special morphisms for the moduli stack $\mathcal{Y}_{0}(N)_{(p)}$.
\begin{definition}
Let $S$ be a scheme over $\textup{Spec}\,\mathbb{Z}_{(p)}$, for an object $((E\stackrel{\pi}\longrightarrow E^{\prime}),\overline{(\eta^{p},\eta^{\prime p})})$ in $\mathcal{Y}_{0}(N)_{(p)}(S)$, a special morphism of this object is an element $j\in \textup{Hom}(E,E^{\prime})\otimes_{\mathbb{Z}}\mathbb{Z}_{(p)}$ satisfying
\begin{equation*}
    j\circ\pi^{\vee}+\pi\circ j^{\vee} = 0.
\end{equation*}
we denote this space by $S(E,\pi)$.
\label{spend}
\end{definition}
\begin{definition}
For every symmetric $n\times n$ matrix $T=(T_{ik})$. Let $\boldsymbol\varphi^{p}$ be the characteristic function of an open compact subset $\boldsymbol\omega^{p}$ of $(\mathbb{V}_{f}^{p})^{n}$ invariant under the action of $\Gamma_{0}(N)(\hat{\mathbb{Z}}^{p})$. We consider the stack $\mathcal{Z}(T,\boldsymbol{\varphi}^{p})$, whose fibered category over a $\mathbb{Z}_{(p)}$-scheme $S$ consists of the following objects,
\begin{equation*}
   ((E\stackrel{\pi}\longrightarrow E^{\prime}), \overline{(\eta^{p},\eta^{\prime p})}, \boldsymbol{j}),
\end{equation*}
where $((E\stackrel{\pi}\longrightarrow E^{\prime}), \overline{(\eta^{p},\eta^{\prime p})})$ is an object in $\mathcal{Y}_{0}(N)_{(p)}(S)$, $\boldsymbol{j}=(j_{1},j_{2},\cdot\cdot\cdot,j_{n})\in S(E,\pi)^{n}$ and $\eta^{\ast}(\boldsymbol{j})\coloneqq\eta^{\prime p}\circ  V^{p}(\boldsymbol{j})\circ(\eta^{p})^{-1}\in\boldsymbol{\omega}^{p}$. Moreover, 
\begin{equation*}
    T_{ik}=\frac{1}{2}\left(\textup{deg}(j_{i}+j_{k})-\textup{deg}(j_{i})-\textup{deg}(j_{k})\right)=\frac{1}{2}(j_{i}\circ j_{k}^{\vee}+j_{k}\circ j_{i}^{\vee}).
\end{equation*}
The special cycle $\mathcal{Z}(T,\boldsymbol{\varphi}^{p})$ may be empty.
\label{higherco}
\label{spe}
\end{definition}
For every symmetric $n\times n$ matrix $T$, we have a natural morphism $i_{n}: \mathcal{Z}(T,\boldsymbol{\varphi}^{p})\rightarrow \mathcal{Y}_{0}(N)_{(p)}$ by forgetting the special morphisms. 
\begin{remark}
    Let $T\in\textup{Sym}_{n}(\mathbb{Q})$, let $\tilde{\boldsymbol\varphi}^{p}$ be the characteristic function of an open compact subset $\tilde{\boldsymbol\omega}^{p}$ of $\textup{M}_{2}(\mathbb{A}_{f}^{p})^{n}$ invariant under the action of $\textup{GL}_{2}(\hat{\mathbb{Z}}^{p})\times \textup{GL}_{2}(\hat{\mathbb{Z}}^{p})$, and let $\boldsymbol\varphi^{p}$ be the restriction of $\tilde{\boldsymbol\varphi}^{p}$ to the subspace $(\mathbb{V}_{f}^{p})^{n}$ of $\textup{M}_{2}(\mathbb{A}_{f}^{p})^{n}$, then $\boldsymbol\varphi^{p}$ is the characteristic function of an open compact subset $\boldsymbol\omega^{p}$ of $(\mathbb{V}_{f}^{p})^{n}$ invariant under the action of $\Gamma_{0}(N)(\hat{\mathbb{Z}}^{p})$ by Lemma \ref{cliffordal}, then the following diagram is Cartesian by definition,
    \begin{equation*}
     \xymatrix{
    \mathcal{Z}(T,\boldsymbol{\varphi}^{p})\ar[d]\ar[r]^{i_{n}}&\mathcal{Y}_{0}(N)_{(p)}\ar[d]\\
    \mathcal{Z}^{\sharp}(T,\tilde{\boldsymbol{\varphi}}^{p})\ar[r]^{i^{\sharp}_{n}}&\mathcal{H}_{(p)}.}
 \end{equation*}
Therefore the morphisms $i_{n}: \mathcal{Z}(T,\boldsymbol{\varphi}^{p})\rightarrow \mathcal{Y}_{0}(N)_{(p)}$ are also finite unramified. In particular, when $n=1$, let $T=d\in\mathbb{Q}_{>0}$, the morphism $i_{1}: \mathcal{Z}(d,\boldsymbol{\varphi}^{p})\rightarrow \mathcal{Y}_{0}(N)_{(p)}$ is a generalized Cartier divisor by Proposition \ref{divisorH}.
\label{abasechange} 
\end{remark}
We will show next that the composite $\tilde{i}_{n}:\mathcal{Z}(T,\boldsymbol{\varphi}^{p})\stackrel{i_{n}}\rightarrow\mathcal{Y}_{0}(N)_{(p)}\rightarrow\mathcal{X}_{0}(N)_{(p)}$ is also finite unramified. We start with the case that $n=1$.
\begin{proposition}
Let $\varphi^{p}$ be the characteristic function of an open compact subset $\omega^{p}$ of $\mathbb{V}_{f}^{p}$ invariant under the action of $\Gamma_{0}(N)(\hat{\mathbb{Z}}^{p})$. For any positive number $d\in\mathbb{Q}$, the morphism $\tilde{i}_{1}:\mathcal{Z}(d,\varphi^{p})\rightarrow\mathcal{X}_{0}(N)_{(p)}$ is finite unramified, and $\mathcal{Z}(d,\varphi^{p})$ is a generalized Cartier divisor.
\label{nocusp}
\end{proposition}
\begin{proof}
The morphism $\tilde{i}_{1}$ is unramified since $i_{1}$ is unramified and the open immersion $\mathcal{Y}_{0}(N)_{(p)}\rightarrow\mathcal{X}_{0}(N)_{(p)}$ is also unramified, therefore we only need to show the finiteness of $\tilde{i}_{1}$.
\par
We first prove that the stack $\mathcal{Z}(d,\varphi^{p})$ is flat over $\mathbb{Z}_{(p)}$, since the morphism $\mathcal{Z}(d,\varphi^{p})\mathcal{Y}_{0}(N)_{(p)}$ is a generalized Cartier divisor by Remark \ref{abasechange}, the flatness of $\mathcal{Z}(d,\varphi^{p})$ is equivalent to the fact that its local equation is not divisible by $p$ since the stack $\mathcal{Y}_{0}(N)_{(p)}$ is flat over $\mathbb{Z}_{(p)}$. We assume the converse and suppose that there exists a point $z\in\mathcal{Z}(d,\varphi^{p})(\overline{\mathbb{F}}_{p})$ such that the equation of $\mathcal{Z}(d,\varphi^{p})$ in the étale local ring $\mathcal{O}_{\mathcal{Y}_{0}(N),z}^{\textup{ét}}$ is divisible by $p$, then the stack $\mathcal{Z}(d,\varphi^{p})$ contains an irreducible component of $\mathcal{Y}_{0}(N)_{\mathbb{F}_{p}}$ in an étale neighbourhood of $z$, let $(E\stackrel{\pi}\rightarrow E^{\prime},\overline{(\eta^{p},\eta^{\prime p})})$ be the object corresponding to the generic point of this irreducible component, then $\textup{End}(E)\simeq\mathbb{Z}$ since the $j$-invariant of $E$ must be transcendental over $\mathbb{F}_{p}$ (by the description of the stack $\mathcal{Y}_{0}(N)_{\mathbb{F}_{p}}$ in \cite[Proposition 13.4.5, Theorem 13.4.7]{KM85}). There also exists an isogeny $j\in \textup{Hom}(E,E^{\prime})\otimes_{\mathbb{Z}}\mathbb{Z}_{(p)}$ such that $j^{\vee}\circ \pi+\pi^{\vee}\circ j = 0$. Let $\alpha=j^{-1}\circ\pi\in\textup{End}^{\circ}(E)\coloneqq \textup{End}(E)\otimes\mathbb{Q}\simeq \mathbb{Q}$, then $\alpha^{2}=-Nd^{-1}<0$, this contradicts to the fact that $\textup{End}^{\circ}(E)\simeq\mathbb{Q}$.
\par
Therefore the stack $\mathcal{Z}(d,\varphi^{p})$ is flat over $\mathbb{Z}_{(p)}$, hence equals to the flat closure of its generic fiber $\mathcal{Z}(d,\varphi^{p})_{\mathbb{Q}}\coloneqq\mathcal{Z}(d,\varphi^{p})\times_{\mathbb{Z}_{(p)}}\mathbb{Q}$. The stack $\mathcal{Z}(d,\varphi^{p})_{\mathbb{Q}}$ consists of finitely many points whose residue fields are finite extensions of $\mathbb{Q}$, therefore the structure sheaf $\mathcal{O}_{\mathcal{Z}(d,\varphi^{p})}$ of $\mathcal{Z}(d,\varphi^{p})$ is a finite product of subrings of the integer rings of these residue fields, hence the stack $\mathcal{Z}(d,\varphi^{p})$ is finite over $\mathbb{Z}_{(p)}$, then $\tilde{i}_{1}:\mathcal{Z}(d,\varphi^{p})\rightarrow\mathcal{X}_{0}(N)_{(p)}$ is proper since $\mathcal{X}_{0}(N)_{(p)}$ is proper over $\mathbb{Z}_{(p)}$. The morphism $\tilde{i}_{1}$ is obviously quasi-finite by the finiteness of $\mathcal{Z}(d,\varphi^{p})$ over $\mathbb{Z}_{(p)}$, hence $\tilde{i}_{1}$ is finite. 
\par
We already know that the morphism $\tilde{i}_{1}$ is a generalized Cartier divisor over the open substack $\mathcal{Y}_{0}(N)_{(p)}$ of $\mathcal{X}_{0}(N)_{(p)}$. Moreover, étale locally around a cusp point of $\mathcal{X}_{0}(N)_{(p)}$, the stack $\mathcal{Z}(d,\varphi^{p})$ is cut out by 1 since $\tilde{i}_{1}$ factors through the non-cuspidal locus $\mathcal{Y}_{0}(N)_{(p)}$, hence the finite unramified morphism $\tilde{i}_{1}:\mathcal{Z}(d,\varphi^{p})\rightarrow \mathcal{X}_{0}(N)_{(p)}$ is a generalized Cartier divisor on the stack $\mathcal{X}_{0}(N)_{(p)}$.
\end{proof}
\begin{corollary}
    Let $\boldsymbol\varphi^{p}=\prod\limits_{i=1}^{n}\varphi_{i}^{p}$ be the characteristic function of an open compact subset $\boldsymbol\omega^{p}$ of $(\mathbb{V}_{f}^{p})^{n}$ invariant under the action of $\Gamma_{0}(N)(\hat{\mathbb{Z}}^{p})$. For any matrix $T\in\textup{Sym}_{n}(\mathbb{Q})_{>0}$, the morphism $\tilde{i}_{n}:\mathcal{Z}(T,\boldsymbol{\varphi}^{p})\rightarrow\mathcal{X}_{0}(N)_{(p)}$ is finite unramified.
    \label{nocusp2}
\end{corollary}
\begin{proof}
Suppose the diagonal elements of $T$ are $d_{1},\cdot\cdot\cdot,d_{n}$. Proposition \ref{nocusp} implies that the morphism $\mathcal{Z}(d_{1},\varphi_{1}^{p})\times_{\mathcal{X}_{0}(N)_{(p)}}\times\cdot\cdot\cdot\times_{\mathcal{X}_{0}(N)_{(p)}}\mathcal{Z}(d_{n},\varphi_{n}^{p})\rightarrow\mathcal{X}_{0}(N)_{(p)}$ is finite unramified. The stack $\mathcal{Z}(T,\boldsymbol{\varphi}^{p})$ is a connected component of $\mathcal{Z}(d_{1},\varphi_{1}^{p})\times_{\mathcal{X}_{0}(N)_{(p)}}\times\cdot\cdot\cdot\times_{\mathcal{X}_{0}(N)_{(p)}}\mathcal{Z}(d_{n},\varphi_{n}^{p})$, hence the morphism $\tilde{i}_{n}:\mathcal{Z}(T,\boldsymbol{\varphi}^{p})\rightarrow\mathcal{X}_{0}(N)_{(p)}$ is finite unramified.
\end{proof}
We will mainly focus on the case that $T$ is a nonsingular $2\times2$ symmetric matrix with coefficients in $\mathbb{Q}$. For every such matrix $T$, recall that we have defined the difference set to be $\text{Diff}(T,\Delta(N)) = \{l\text{ is a finite prime}:\text{$T$ is not represented by $\Delta(N)\otimes\mathbb{Q}_{l}$.}\}$.
\begin{proposition}
Let $T\in\textup{Sym}_{2}(\mathbb{Q})$ be a nonsingular matrix. If $\mathcal{Z}(T,\boldsymbol{\varphi}^{p})(\overline{\mathbb{F}}_{p})\neq\varnothing$ for some prime p, then $T$ is positive definite, and
\begin{equation*}
    \textup{Diff}(T,\Delta(N)) = \{p\}.
\end{equation*}
Moreover, in this case, the special cycle $\mathcal{Z}(T,\boldsymbol{\varphi}^{p})$ is supported in the supersingular locus of the special fiber $\mathcal{Y}_{0}(N)_{\mathbb{F}_{p}}$.
\label{onlyone}
\end{proposition}
\begin{proof}
Since $\mathcal{Z}(T,\boldsymbol{\varphi}^{p})(\overline{\mathbb{F}}_{p})\neq\varnothing$, Corollary \ref{nocusp2} implies that there are two elliptic curves $\mathbb{E}$ and $\mathbb{E}^{\prime}$ over $\overline{\mathbb{F}}_{p}$, a cyclic isogeny $\pi\in \textup{Hom}(\mathbb{E},\mathbb{E}^{\prime})$, and two isogenies $x_{1}, x_{2}\in \textup{Hom}(\mathbb{E},\mathbb{E}^{\prime})_{(p)}$ such that
\begin{equation*}
    T = (\frac{1}{2}(x_{i},x_{j})),\,\,\textup{and}\,\,(x_{1},\pi) =(x_{2},\pi) = 0. 
\end{equation*}
therefore $T$ must be positive definite and both $\mathbb{E}$ and $\mathbb{E}^{\prime}$ are supersingular elliptic curves over $\overline{\mathbb{F}}_{p}$ since $\textup{dim}_{\mathbb{Q}}\,\textup{Hom}(\mathbb{E},\mathbb{E}^{\prime})\otimes\mathbb{Q}\geq 3$. The quadratic space $\textup{Hom}(\mathbb{E},\mathbb{E}^{\prime})\otimes\mathbb{Q}_{p}$ is isometric to the underlying quadratic space of the unique division quaternion algebra $\mathbb{B}$ over $\mathbb{Q}_{p}$.
\par
The isogenies $x_{1},x_{2}\in\{\pi\}^{\bot}\subset\textup{Hom}(\mathbb{E},\mathbb{E}^{\prime})\otimes\mathbb{Q}_{p}\simeq\mathbb{B}$ where $\pi^{\vee}\circ\pi=N$. However, $\{\pi\}^{\bot}$ and $\Delta(N)\otimes\mathbb{Q}_{p}$ have the same discriminant $-N$ but opposite Hasse invariants, therefore $p\in \textup{Diff}(T,\Delta(N))$. At the same time, by choosing some level structures on $\mathbb{E}$ and $\mathbb{E}^{\prime}$ away from $p$, we get that $T$ can be realized in $\Delta(N)\otimes\mathbb{Q}_{l}$ for any finite prime $l\neq p$. Therefore $p$ is the only prime in the set $\textup{Diff}(T,\Delta(N))$.
\end{proof}
\begin{remark}
    Proposition \ref{onlyone} implies that the special cycle $\mathcal{Z}(T,\boldsymbol{\varphi}^{p})$ is also finite unramified over the stack $\mathcal{X}_{0}(N)_{(p)}$ because the scheme-theoretic image $\tilde{\mathcal{Z}}(T,\boldsymbol{\varphi}^{p})$ of $\mathcal{Z}(T,\boldsymbol{\varphi}^{p})$ in $\mathcal{X}_{0}(N)_{(p)}$ is supported in the supersingular locus of the special fiber $\mathcal{X}_{0}(N)_{\mathbb{F}_{p}}$, which equals the supersingular locus of the special fiber $\mathcal{Y}_{0}(N)_{\mathbb{F}_{p}}$, hence $\tilde{\mathcal{Z}}(T,\boldsymbol{\varphi}^{p})$ is contained in $\mathcal{Y}_{0}(N)_{(p)}$, and therefore equals to the scheme-theoretic image of $\mathcal{Z}(T,\boldsymbol{\varphi}^{p})$ in $\mathcal{Y}_{0}(N)_{(p)}$, over which $\mathcal{Z}(T,\boldsymbol{\varphi}^{p})$ is finite unramified.
\end{remark}
\par
For any nonsingular $2\times2$ symmetric matrix $T\in\textup{Sym}_{2}(\mathbb{Q})$, we say a Schwartz function $\boldsymbol\varphi=\underset{v<\infty}{\bigotimes}\boldsymbol\varphi_{v}\in\mathscr{S}(\mathbb{V}_{f}^{2})$ is $T$-admissible if $\boldsymbol\varphi$ is invariant under the action of $\Gamma_{0}(N)(\hat{\mathbb{Z}})$, $\boldsymbol\varphi=\varphi_{1}\times\varphi_{2}$ where $\varphi_{i}\in\mathscr{S}(\mathbb{V}_{f})$ and\\
\indent$\bullet$\,\,$T$ is not positive definite, or\\
\indent$\bullet$\,\,$T$ is positive definite, and $\vert \textup{Diff}(T,\Delta(N))\vert\neq1$, or\\
\indent$\bullet$\,\,$T$ is positive definite, $\textup{Diff}(T,\Delta(N)) = \{p\}$ for some prime number $p$, and $\boldsymbol\varphi=\boldsymbol\varphi^{p}\otimes\boldsymbol\varphi_{p}$ where $\boldsymbol\varphi^{p}\in\mathscr{S}((\mathbb{V}_{f}^{p})^{2})$ and $\boldsymbol\varphi_{p}=c\cdot\boldsymbol{1}_{\delta_{p}(N)^{2}}$ for some $c\in\mathbb{C}$.
\begin{definition}
For a nonsingular $2\times2$ matrix $T\in\textup{Sym}_{2}(\mathbb{Q})$ and a $T$-admissible Schwartz function $\boldsymbol\varphi\in\mathscr{S}(\mathbb{V}_{f}^{2})$ which is also a characteristic function of a $\Gamma_{0}(N)(\hat{\mathbb{Z}})$-invariant open compact subset $\boldsymbol\omega$ of $\mathbb{V}_{f}^{2}$, we define a stack finite unramified over $\mathcal{X}_{0}(N)$ as follows,
\begin{equation*}
    \mathcal{Z}(T,\boldsymbol{\varphi}) \coloneqq \mathcal{Z}(T,\boldsymbol{\varphi}^{p})\rightarrow\mathcal{X}_{0}(N)_{(p)}\hookrightarrow\mathcal{X}_{0}(N),
\end{equation*}
where $p\in \textup{Diff}(T,\Delta(N))$. If $\vert\textup{Diff}(T,\Delta(N))\vert\neq1$, we define $ \mathcal{Z}(T,\boldsymbol{\varphi})=\varnothing$.
\label{codim22}
\end{definition}
\begin{remark}
By Proposition \ref{onlyone}, $\mathcal{Z}(T,\boldsymbol{\varphi})$ is nonempty only if $\vert \textup{Diff}(T,\Delta(N))\vert=1$, therefore the above definition makes sense.
\end{remark}
\begin{remark}
If we view $\mathcal{Z}(T,\boldsymbol\varphi)$ as an element in $\textup{CH}_{\mathbb{C}}^{2}(\mathcal{X}_{0}(N))$, we can drop the restrictions in Definition \ref{codim22} that the Schwartz functions $\boldsymbol\varphi$ is the characteristic function of an open compact subset of $\mathbb{V}_{f}^{2}$. Since any $T$-admissible Schwartz functions $\boldsymbol\varphi$ on $\mathbb{V}_{f}^{2}$ is a finite linear combination of $\Gamma_{0}(N)(\hat{\mathbb{Z}})$-invariant characteristic functions of some open compact subsets, we can define $\mathcal{Z}(T,\boldsymbol\varphi)$ as the corresponding linear combination of elements in $\textup{CH}_{\mathbb{C}}^{2}(\mathcal{X}_{0}(N))$.
\label{fuldef}
\end{remark}

\subsubsection{Comparison with \cite[$\S$2.2]{SSY22}}
Another kind of special cycles of $\mathcal{X}_{0}(N)$ is defined in \cite[$\S$2.2]{SSY22} as follows,
\begin{definition}
    For $m\in\mathbb{Z}$, let $\mathcal{Z}(m)$ denote the moduli stack whose $S$ points, for a base scheme $S$, are given by
    \begin{equation*}
        \mathcal{Z}(m)(S)\coloneqq\{(E\stackrel{\pi}\rightarrow E^{\prime},\alpha)\}
    \end{equation*}
    where $(E\stackrel{\pi}\rightarrow E^{\prime})\in\mathcal{Y}_{0}(N)(S)$ where $\alpha\in\textup{End}(E)$ satisfies the following conditions:\\
    (a)\,\,$\alpha^{\vee}\circ\alpha=mN$ and $\alpha^{\vee}+\alpha=0$;\\
    (b)\,\,$\alpha\circ \pi^{-1}\in\textup{Hom}(E^{\prime},E)$;\\
    (c)\,\,$\pi\circ\alpha\circ\pi^{-1}\in\textup{End}(E^{\prime})$.
\end{definition}
\begin{lemma}
    For every prime number $p$, let $\mathcal{Z}(m)_{(p)}\coloneqq\mathcal{Z}(m)\times_{\mathbb{Z}}\mathbb{Z}_{(p)}$, then we have an isomorphism of stacks as follows,
    \begin{align*}
        T: \mathcal{Z}(m)_{(p)}\stackrel{\sim}\longrightarrow & \mathcal{Z}(m,\mathbf{1}_{\Delta(N)\otimes\hat{\mathbb{Z}}^{p}}),\\
        (E\stackrel{\pi}\rightarrow E^{\prime},\alpha) \longmapsto &(E\stackrel{\pi}\rightarrow E^{\prime},\overline{(\eta^{p},\eta^{\prime p})},(\alpha\circ\pi^{-1})^{\vee}).
    \end{align*}
\end{lemma}
\begin{proof}
We first prove that $T$ is well-defined. For any connected $\mathbb{Z}_{(p)}$-scheme $S$, let $\overline{s}$ be a geometric point of $S$, we can choose trivilizations $\eta^{p}:V^{p}(E_{\overline{s}})\stackrel{\sim}\rightarrow (\mathbb{A}_{f}^{p})^{2}$ and $\eta^{\prime p}: V^{p}(E_{\overline{s}}^{\prime})\stackrel{\sim}\rightarrow (\mathbb{A}_{f}^{p})^{2}$ such that $T^{p}(E_{\overline{s}})$ and $T^{p}(E_{\overline{s}}^{\prime})$ are mapped isomorphically to $(\hat{\mathbb{Z}}^{p})^{2}$, and $\eta^{\prime p}\circ V^{p}(\pi)\circ (\eta^{p})^{-1}=w_{N}$ by the cyclicity of $\pi$. Moreover,
\begin{align*}
    (\alpha\circ\pi^{-1})^{\vee}\circ \pi+\pi^{\vee}\circ (\alpha\circ\pi^{-1})&=\frac{1}{N}\pi^{\vee}\circ\alpha^{\vee}\circ\pi+\frac{1}{N}\pi^{\vee}\circ\alpha\circ\pi\\
    &= \frac{1}{N}\pi^{\vee}\circ(\alpha^{\vee}+\alpha)\circ\pi=0.
\end{align*}
hence $(\alpha\circ\pi^{-1})^{\vee}\in S(E,\pi)$, then (b) implies that $\eta^{\prime p}\circ V^{p}((\alpha\circ\pi^{-1})^{\vee})\circ (\eta^{p})^{-1}\in \Delta(N)\otimes\hat{\mathbb{Z}}^{p}\subset\mathbb{V}_{f}^{p}$, therefore $(E\stackrel{\pi}\rightarrow E^{\prime},\overline{(\eta^{p},\eta^{\prime p})},(\alpha\circ\pi^{-1})^{\vee})\in \mathcal{Z}(m,\mathbf{1}_{\Delta(N)\otimes\hat{\mathbb{Z}}^{p}})(S)$.
\par
We define the following morphism,
\begin{align*}
        R: \mathcal{Z}(m,\mathbf{1}_{\Delta(N)\otimes\hat{\mathbb{Z}}^{p}})\longrightarrow &\mathcal{Z}(m)_{(p)},\\
        (E\stackrel{\pi}\rightarrow E^{\prime},\overline{(\eta^{p},\eta^{\prime p})},j)\longmapsto &(E\stackrel{\pi}\rightarrow E^{\prime},j^{\vee}\circ\pi).
    \end{align*}
We show that $R$ is well-defined: For any connected $\mathbb{Z}_{(p)}$-scheme $S$, an object $(E\stackrel{\pi}\rightarrow E^{\prime},\overline{(\eta^{p},\eta^{\prime p})},j)\in \mathcal{Z}(m,\mathbf{1}_{\Delta(N)\otimes\hat{\mathbb{Z}}^{p}})(S)$ means that $j\in\textup{Hom}(E,E^{\prime})\otimes\mathbb{Z}_{(p)}$ and $j^{\vee}\circ\pi+\pi^{\vee}\circ j=0$, the fact that $\eta^{\prime p}\circ V^{p}(j)\circ (\eta^{p})^{-1}\in \Delta(N)\otimes\hat{\mathbb{Z}}^{p}$ implies that $j\in\textup{Hom}(E,E^{\prime})$, then $j^{\vee}\circ\pi\in\textup{End}(E)$, $(j^{\vee}\circ\pi)^{\vee}\circ(j^{\vee}\circ\pi)=\pi^{\vee}\circ j\circ j^{\vee}\circ\pi=mN$ and $(j^{\vee}\circ\pi)^{\vee}+j^{\vee}\circ\pi=\pi^{\vee}\circ j+j^{\vee}\circ\pi=0$, which is exactly (a). Moreover, (b) and (c) are easily verified, hence $(E\stackrel{\pi}\rightarrow E^{\prime},j^{\vee}\circ\pi)\in \mathcal{Z}(m)(S)$, then the morphism $R$ is well-defined. It's easy to see that $T$ and $R$ are inverse to each other, therefore the lemma is proved.
\end{proof}

\subsubsection{Arithmetic special cycles on $\mathcal{X}_{0}(N)$}
We apply the arithmetic intersection theory developed by Henri Gillet in \cite{Gil82} and \cite{Gil09} to the regular proper flat Deligne-Mumford stack $\mathcal{X}_{0}(N)$. We obtain the following arithmetic Chow ring of $\mathcal{X}_{0}(N)$,
\begin{equation*}
    \widehat{\textup{CH}}^{\bullet}_{\mathbb{C}}(\mathcal{X}_{0}(N)) = \bigoplus\limits_{n=0}^{2}\widehat{\textup{CH}}^{n}_{\mathbb{C}}(\mathcal{X}_{0}(N)).
\end{equation*}
Roughly speaking, a class in $\widehat{\textup{CH}}^{n}_{\mathbb{C}}(\mathcal{X}_{0}(N))$ is represented by an arithmetic cycle $(\mathcal{Z},g_{\mathcal{Z}})$, where $\mathcal{Z}$ is a codimension $n$ closed substack of $\mathcal{X}_{0}(N)$, with $\mathbb{C}$-coefficients, and $g_{\mathcal{Z}}$ is a Green current for $\mathcal{Z}(\mathbb{C})$, i.e., $g_{\mathcal{Z}}$ is a current on the proper smooth complex curve $\mathcal{X}_{0}(N)_{\mathbb{C}}$ of degree $(n-1,n-1)$ for which there exists a smooth $\omega$ such that
\begin{equation*}
    dd^{c}(g) + \delta_{\zeta} =[\omega].
\end{equation*}
holds; here $[\omega]$ is the current defined by integration against the smooth form $\omega$. The rational arithmetic cycles are those of the form $\widehat{\textup{div}}(f)=(\textup{div}(f), \iota_{\ast}[-\textup{log}(\vert\tilde{f}\vert^{2})])$, where $f\in \kappa(\mathcal{Z})^{\times}$ is a rational function on a codimension $n-1$ integral substack $\iota:\mathcal{Z}\hookrightarrow \mathcal{X}_{0}(N)$, together with classes of the form $(0,\partial\eta+\overline{\partial}\eta^{\prime})$. By definition, the arithmetic Chow group $\textup{CH}^{n}_{\mathbb{C}}(\mathcal{X}_{0}(N))$ is the quotient of the space of arithmetic cycles by the $\mathbb{C}$-subspace spanned by those rational cycles.
\par
Let $\mathcal{Z}$ be an irreducible codimension $2$ cycle on $\mathcal{X}_{0}(N)$, then $\mathcal{Z}$ is a Deligne-Mumford stack over $\mathbb{F}_{p}$ for some prime number $p$ and the groupoid $\mathcal{Z}(\overline{\mathbb{F}}_{p})$ would be a singleton with a finite automorphism group $\textup{Aut}(\mathcal{Z})$, the rational function field $\kappa(\mathcal{Z})$ of $\mathcal{Z}$ is a finite extension of $\mathbb{F}_{p}$. Clearly $\delta_{\mathcal{Z}}=0$ because $\mathcal{Z}(\mathbb{C})=\varnothing$.
\par
Let $(\mathcal{Z},g)=(\sum\limits_{i}n_{i}[\mathcal{Z}_{i}],g)$ be an arithmetic cycle of codimension $2$ where each $\mathcal{Z}_{i}$ is an irreducible codimension $2$ cycle on $\mathcal{X}_{0}(N)$. We define the degree map as follows,
\begin{align}
    \widehat{\textup{deg}}: \widehat{\textup{CH}}^{2}_{\mathbb{C}}(\mathcal{X}_{0}(N))&\longrightarrow \mathbb{C},\label{degreemap}\\
    [(\mathcal{Z},g)]&\longmapsto \sum\limits_{i}n_{i}\cdot\frac{\textup{log}\,\vert \kappa(\mathcal{Z}_{i})\vert}{\vert\textup{Aut}(\mathcal{Z}_{i})\vert}+\frac{1}{2}\int\limits_{\mathcal{X}_{0}(N)(\mathbb{C})}g.\notag
\end{align}
here the integration $\int\limits_{\mathcal{X}_{0}(N)_\mathbb{C}}g$ is the integration of the constant function $1$ on $\mathcal{X}_{0}(N)_{\mathbb{C}}$ against the $(1,1)$-current $g$. It is a finite number since the stack $\mathcal{X}_{0}(N)$ is proper. This number is independent of the choice of representing element $(\mathcal{Z},g)$ as a consequence of the product formula (cf. \cite[3.4.3]{Gil82}).\\
\par
Now we are going to construct Green currents for the special cycle $\mathcal{Z}(T,\boldsymbol\varphi)$. Let $\mathbb{D}=\{z\in\Delta(N)\otimes_{\mathbb{Z}}\mathbb{C}:(z,z)=0,\,(z,\overline{z})<0\}\,/\,\mathbb{C}^{\ast}\subset\mathbb{P}(\Delta(N)\otimes\mathbb{C})$, we have the following $\textup{GL}_{2}(\mathbb{R})$-equivariant identification,
\begin{align*}
    \mathbb{H}_{1}^{\pm}&\stackrel{\sim}\longrightarrow \mathbb{D},\\
    \tau&\longmapsto \textup{span}_{\mathbb{C}}\left\{\begin{pmatrix}
    -N\tau & -N\tau^{2} \\
    1 & \tau
    \end{pmatrix}\right\}.
\end{align*}
\par
Next, we will associate to any nonsingular $T\in\textup{Sym}_{2}(\mathbb{Q})$ and a $T$-admissible Schwartz function $\boldsymbol\varphi\in\mathscr{S}(\mathbb{V}_{f}^{2})$ an element in $\widehat{\textup{CH}}^{2}_{\mathbb{C}}(\mathcal{X}_{0}(N))$. Let $\mathsf{y}={^{t}a}\cdot a\in\textup{Sym}_{2}(\mathbb{R})$ be a positive definite matrix, where $a\in\textup{GL}_{2}(\mathbb{R})$.\\
$\bullet$ For a positive definite $T$ and a $T$-admissible Schwartz function $\boldsymbol\varphi$, we consider the following element
\begin{equation*}
    \widehat{\mathcal{Z}}(T,\mathsf{y},\boldsymbol{\varphi})=(\mathcal{Z}(T,\boldsymbol{\varphi}),0)\in\widehat{\textup{CH}}^{2}_{\mathbb{C}}(\mathcal{X}_{0}(N)).
\end{equation*}
$\bullet$ For other nonsingular $T$ which is not positive definite, we apply the general machine developed by Garcia and Sankaran in \cite{GS19} which is made explicit in \cite{SSY22}. For any $x\in\mathbb{V}_{\infty}$ and $[z]\in\mathbb{D}$, let $\textup{R}(x,[z])=-\vert(x,z)\vert^{2}\cdot(z,\overline{z})^{-1}$. We define an element in $\mathscr{S}(\mathbb{V}_{\infty}^{2})\otimes\mathcal{A}^{1,1}(\mathbb{H}_{1}^{\pm})$ by the following description: for $\boldsymbol{x}=(x_{1},x_{2})\in\mathbb{V}_{\infty}^{2}$ and $[z]\in\mathbb{D}$,
\begin{equation*}
    \nu(\boldsymbol{x},[z])=\left(-\pi^{-1}+2\sum\limits_{i=1}^{2}(\textup{R}(x_{i},[z])+(x_{i},x_{i}))\right)e^{-2\pi(\sum\limits_{i=1}^{2}(\textup{R}(x_{i},[z])+\frac{1}{2}(x_{i},x_{i})))}\cdot\frac{dx\wedge dy}{y^{2}}.
\end{equation*}
then we define a smooth $(1,1)$-form $\mathfrak{g}(T,\mathsf{y},\boldsymbol{\varphi})$ on $\mathbb{D}$ as follows, its value at the point $[z]\in\mathbb{D}$ is
\begin{equation*}
    \mathfrak{g}(T,\mathsf{y},\boldsymbol{\varphi})([z])=\sum\limits_{\boldsymbol{x}\in(\Delta(N)\otimes\mathbb{Q})^{2}\atop T(\boldsymbol{x})=T}\boldsymbol\varphi(\boldsymbol{x})\cdot\int\limits_{1}^{\infty} \nu(t^{\frac{1}{2}}\boldsymbol{x}\cdot {^{t}a},[z])\cdot\frac{\textup{d}t}{t}.
\end{equation*}
the sum converges absolutely, and descends to a smooth $(1,1)$-form on the modular curve $\mathcal{Y}_{0}(N)_{\mathbb{C}}$. \cite{SSY22} proves the following
\begin{lemma}
For nonsingular $T\in\textup{Sym}_{2}(\mathbb{R})$ which is not positive definite. The form $\mathfrak{g}(T,\mathsf{y},\boldsymbol{\varphi})$ is absolutely integrable on $\mathcal{X}_{0}(N)_{\mathbb{C}}$, hence $\mathfrak{g}(T,\mathsf{y},\boldsymbol{\varphi})$ defines a $(1,1)$-current on $\mathcal{X}_{0}(N)_{\mathbb{C}}$.
\end{lemma}
\begin{proof}
This is proved in \cite[Lemma 2.9]{SSY22}.
\end{proof}
\par
To sum up, let $T\in \textup{Sym}_{2}(\mathbb{Q})$ be a nonsingular matrix, $\mathsf{y}\in\textup{Sym}_{2}(\mathbb{R})_{>0}$, $\boldsymbol\varphi\in\mathscr{S}(\mathbb{V}_{f}^{2})$ is a $T$-admissible Schwartz function, we define
\begin{equation}
     \widehat{\mathcal{Z}}(T,\mathsf{y},\boldsymbol{\varphi}) = \begin{cases}
   ([\mathcal{Z}(T,\boldsymbol{\varphi})],0), & \textup{when $T$ is positive definite;}\\
    (0,\mathfrak{g}(T,\mathsf{y},\boldsymbol{\varphi})), &\textup{when $T$ is not positive definite.}\\
    \end{cases}
    \label{codim2}
\end{equation}
it is an element in $\widehat{\textup{CH}}^{2}_{\mathbb{C}}(\mathcal{X}_{0}(N))$.
\subsection{Arithmetic Siegel-Weil formula on $\mathcal{X}_{0}(N)$}
Now we can state the main theorem of this article, which proves an identity between arithmetic intersection numbers on $\mathcal{X}_{0}(N)$ and derivatives of Fourier coefficients of Eisenstein series,
\begin{theorem}
Let $N$ be a positive integer. Let $T\in \textup{Sym}_{2}(\mathbb{Q})$ be a nonsingular symmetric matrix. Let $\boldsymbol{\varphi}\in\mathscr{S}(\mathbb{V}_{f}^{2})$ be a $T$-admissible Schwartz function, then
\begin{equation*}
     \widehat{\textup{deg}}(\widehat{\mathcal{Z}}(T,\mathsf{y},\boldsymbol{\varphi}))q^{T}= \frac{\psi(N)}{24}\cdot\partial\textup{Eis}_{T}(\mathsf{z},\boldsymbol\varphi),
\end{equation*}
for any $\mathsf{z}=\mathsf{x}+i\mathsf{y}\in\mathbb{H}_{2}$. Here $\psi(N)= N\cdot\prod\limits_{l\vert N}(1+l^{-1})$, $q^{T}=e^{2\pi i \,\textup{tr}\,(T\mathsf{z})}$.
\label{main}
\end{theorem}
The article \cite{SSY22} proves this formula in the case that $T$ is not positive definite without any restrictions on the the level $N$, and the case that $T$ is positive definite but with the restriction that $N$ is odd and square-free. We will give a proof of the case that $T$ is positive definite and $N$ is arbitrary in $\S$\ref{proof}.

\section{Rapoport-Zink spaces and special cycles}
\label{4}
\subsection{$\Gamma_{0}(N)$-structures on $p$-divisible groups}
Let $p$ be a prime number. Let $\mathbb{F}$ be the algebraic closure of $\mathbb{F}_{p}$. Let $W$ be the completion of the maximal unramified extension of $\mathbb{Q}_{p}$. Let $\textup{Nilp}_{W}$ be the category of schemes $S$ over $\textup{Spec}\,W$ such that $p$ is locally nilpotent on $S$, let $\overline{S}$ be the closed subscheme of $S$ defined by the ideal sheaf $p\mathcal{O}_{S}$. For a $p$-divisible group $X$, we use $X^{\vee}$ to denote the dual $p$-divisible group. We will introduce two Rapoport-Zink spaces in this chapter, they are essentially isomorphic to the completed local rings of supersingular points in characteristic $p$ of the moduli stacks $\mathcal{H}$ and $\mathcal{X}_{0}(N)$.
\par
Let $\mathbb{X}$ be a $p$-divisible group over $\mathbb{F}$ of dimension 1 and height 2, the associated filtered isocrystal $\mathbb{D}(\mathbb{X})_{\mathbb{Q}}$ has pure slope $\frac{1}{2}$, e.g., we can take $\mathbb{X}$ to be $\mathbb{E}[p^{\infty}]$ where $\mathbb{E}$ is a supersingular elliptic curve over $\mathbb{F}$. Let $\lambda_{0}:\mathbb{X}\stackrel{\sim}\rightarrow\mathbb{X}^{\vee}$ be a principal polarization. We consider the following functor $\mathcal{N}$ on the category $\textup{Nilp}_{W}$: for any $S\in \textup{Nilp}_{W}$, the set $\mathcal{N}(S)$ is the isomorphism classes of tuples $((X,\rho,\lambda),(X^{\prime},\rho^{\prime},\lambda^{\prime}))$, where\\
\par
(1) $X$ and $X^{\prime}$ are two $p$-divisible group over $S$, $\rho$ and $\rho^{\prime}$ are two quasi-isogenies between $p$-divisible groups $\rho: \mathbb{X}\times_{\mathbb{F}}\overline{S}\rightarrow X\times_{S}\overline{S}$,  $\rho^{\prime}: \mathbb{X}\times_{\mathbb{F}}\overline{S}\rightarrow X^{\prime}\times_{S}\overline{S}$.\\
\par
(2) $\lambda: X\rightarrow X^{\vee}$, $\lambda^{\prime}: X^{\prime}\rightarrow X^{\prime\vee}$ are two principal polarizations, such that Zariski locally on $\overline{S}$, we have 
\begin{equation*}
    \rho^{\vee}\circ\lambda\circ\rho = c(\rho)\cdot\lambda_{0},\,\,\,\,\,\rho^{\prime\vee}\circ\lambda\circ\rho^{\prime} = c(\rho^{\prime})\cdot\lambda_{0}.
\end{equation*}
for some $c(\rho), c(\rho^{\prime})\in\mathbb{Z}_{p}^{\times}$.\\
\par
\begin{proposition}
The functor $\mathcal{N}$ is represented by the formal scheme $\textup{Spf}\,W[[t_{1},t_{2}]]$ over $\textup{Spf}\,W$.
\label{H1}
\end{proposition}
\begin{proof}
When $p$ is odd, this is explained in \cite[Example 4.5.3 (ii)]{LZ22}. In general, the deformation space of the supersingular elliptic curve $\mathbb{E}$ is isomorphic to $\textup{Spf}\,W[[t]]$. By Serre-Tate theorem, this is also the deformation space of the $p$-divisible group $\mathbb{X}$ with certain restrictions on the polarization as in the definition of the deformation functor $\mathcal{N}$. Therefore $\mathcal{N}\simeq\textup{Spf}\,W[[t_{1}]]\times_{\textup{Spf}\,W}\textup{Spf}\,W[[t_{2}]]\simeq\textup{Spf}\,W[[t_{1},t_{2}]]$. 
\end{proof}
Let $((X^{\textup{univ}}, \rho^{\textup{univ}},\lambda^{\textup{univ}}),(X^{\prime\textup{univ}},\rho^{\prime\textup{univ}},\lambda^{\prime\textup{univ}}))$ be the universal $p$-divisible group over $\linebreak\mathcal{N}=\textup{Spf}\,W[[t_{1},t_{2}]]$. By Lemma \ref{equivalence} below, the category of $p$-divisible groups over $\textup{Spf}\,W[[t_{1},t_{2}]]$ is equivalent to the category of $p$-divisible groups over $\textup{Spec}\,W[[t_{1},t_{2}]]$, we still use $\linebreak((X^{\textup{univ}}, \rho^{\textup{univ}},\lambda^{\textup{univ}}),(X^{\prime\textup{univ}},\rho^{\prime\textup{univ}},\lambda^{\prime\textup{univ}}))$ to denote the corresponding $p$-divisible group over $\linebreak\textup{Spec}\,W[[t_{1},t_{2}]]$.
\par
Next we fix a $N$-isogeny $x_{0}:\mathbb{X}\rightarrow\mathbb{X}$, i.e., $x_{0}\circ x_{0}^{\vee} = N$. $\mathcal{N}_{0}(N)$ is a contravariant set-valued functor defined over $\textup{Nilp}_{W}$, for every $S\in \textup{Nilp}_{W}$, the set $\mathcal{N}_{0}(N)(S)$ consists of the isomorphism classes of elements of the following form $(X\stackrel{x}\rightarrow X^{\prime},(\rho,\rho^{\prime}),(\lambda,\lambda^{\prime}))$, where\\
\par
(1) $X$ and $X^{\prime}$ are two $p$-divisible group over $S$, $\rho$ and $\rho^{\prime}$ are two quasi-isogenies between $p$-divisible groups $\rho: \mathbb{X}\times_{\mathbb{F}}\overline{S}\rightarrow X\times_{S}\overline{S}$,  $\rho^{\prime}: \mathbb{X}\times_{\mathbb{F}}\overline{S}\rightarrow X^{\prime}\times_{S}\overline{S}$.\\
\par
(2) $\lambda: X\rightarrow X^{\vee}$, $\lambda^{\prime}: X^{\prime}\rightarrow X^{\prime\vee}$ are two principal polarizations, such that Zariski locally on $\overline{S}$, we have 
\begin{equation*}
    \rho^{\vee}\circ\lambda\circ\rho = c(\rho)\cdot\lambda_{0},\,\,\,\,\,\rho^{\prime\vee}\circ\lambda\circ\rho^{\prime} = c(\rho^{\prime})\cdot\lambda_{0}.
\end{equation*}
for some $c(\rho), c(\rho^{\prime})\in\mathbb{Z}_{p}^{\times}$.\\
\par
(3) $x:X\rightarrow X^{\prime}$ is a cyclic isogeny (i.e., $\textup{ker}(f)$ is a cyclic group scheme over $S$) lifting $\rho^{\prime}\circ x_{0}\circ\rho^{-1}$.\\
\par
We will prove it later that the functor $\mathcal{N}_{0}(N)$ is represented by a closed formal subscheme of $\textup{Spf}\,W[[t_{1},t_{2}]]$ cut out by a single equation (cf. Theorem \ref{keydecomgeo}).
\label{locdeform}

\subsection{Special cycles on $\mathcal{N}$ and $\mathcal{N}_{0}(N)$}
Now we give the definition of special cycles on the formal schemes $\mathcal{N}$ and $\mathcal{N}_{0}(N)$. Recall that $((X^{\textup{univ}}, \rho^{\textup{univ}},\lambda^{\textup{univ}}),(X^{\prime\textup{univ}},\rho^{\prime\textup{univ}},\lambda^{\prime\textup{univ}}))$ is the universal $p$-divisible group over $\mathcal{N}$, and $\mathbb{B}\simeq \textup{End}^{0}(\mathbb{X})$ is the unique division quaternion algebra over $\mathbb{Q}_{p}$, whose Hasse invariant as a quadratic space is $-1$. 
\begin{definition}
For any subset $L\subset\mathbb{B}$, define the special cycle $\mathcal{Z}^{\sharp}(L)\subset\mathcal{N}$ to be the closed formal subscheme cut out by the condition,
\begin{equation*}
    \rho^{\prime\textup{univ}}\circ x \circ (\rho^{\textup{univ}})^{-1} \in \textup{Hom}(X^{\textup{univ}}, X^{\prime\textup{univ}}).
\end{equation*}
for all $x\in L$.
\label{specialcycle}
\end{definition}
\begin{remark}
The special cycle $\mathcal{Z}^{\sharp}(L)$ only depends on the $\mathbb{Z}_{p}$-linear span of $L$ in $\mathbb{B}$, and is nonempty only when this span is an integral quadratic $\mathbb{Z}_{p}$-lattice in $\mathbb{B}$.
\end{remark}
\begin{proposition}
Let $x\in\mathbb{B}$ be a nonzero and integral element, i.e., $0\leq\nu_{p}(x^{\vee}\circ x)<\infty$. Then $\mathcal{Z}^{\sharp}(x)$ is a Cartier divisor on $\mathcal{N}$, i.e., it is defined by a single nonzero element $f_{x}\in W[[t_{1},t_{2}]]$. Moreover, $\mathcal{Z}^{\sharp}(x)$ is also flat over $\textup{Spf}\,W$, i.e., $p\nmid f_{x}$.
\label{equation}
\end{proposition}
\begin{proof}
When $p$ is odd, the formal scheme $\mathcal{N}$ is an example of GSpin Rapoport-Zink space (cf. \cite[Example 4.5.3 (ii)]{LZ22}), and the proposition has been proved for every GSpin Rapoport-Zink space in \cite[Proposition 4.10.1]{LZ22}. In general, this is proved in \cite[Theorem 6.8.1]{KM85}.
\end{proof}
Now let's come to the special cycles on $\mathcal{N}_{0}(N)$. Firstly, we give the definition of the space of special quasi-isogenies, recall that we have fixed a $N$-isogeny $x_{0}$ when we define the formal scheme $\mathcal{N}_{0}(N)$.
\begin{definition}
We call a quasi-isogeny $x\in \mathbb{B}=\textup{End}^{0}(\mathbb{X})$ is special to $x_{0}$ if the following condition holds,
\begin{equation*}
   x\circ x_{0}^{\vee}+x_{0}\circ x^{\vee} = 0.
\end{equation*}
\label{esp1}
\end{definition}
By definition, the space of quasi-isogenies special to $x_{0}$ is just the quadratic space $\mathbb{W}=\{x_{0}\}^{\bot}\subset\mathbb{B}$. By Witt's theorem, it is a 3-dimensional quadratic space over $\mathbb{Q}_{p}$ whose isometric class is independent of the choice of the $N$-isogeny $x_{0}$.
\begin{definition}
Let $(\Breve{X}\stackrel{\Breve{x}}\rightarrow \Breve{X}^{\prime},(\Breve{\rho},\Breve{\rho}^{\prime}),(\Breve{\lambda},\Breve{\lambda}^{\prime}))$ be the universal object over $\mathcal{N}_{0}(N)$. For any subset $M\subset \mathbb{W}$, define the special cycle $\mathcal{Z}(M)\subset \mathcal{N}_{0}(N)$ to be the closed formal subscheme cut out by the following conditions,
\begin{align*}
    \Breve{\rho}^{\prime}\circ x \circ \Breve{\rho}^{-1} \in \textup{Hom}(\Breve{X},\Breve{X}^{\prime}).
\end{align*}
for any $x\in M$.
\label{esp2}
\end{definition}
For any subset $M\subset \mathbb{W}\subset \mathbb{B}$, we have the following Cartesion diagram,
\begin{equation*}
     \xymatrix{
    \mathcal{Z}(M)\ar[d]\ar[r]&\mathcal{N}_{0}(N)\ar[d]\\
    \mathcal{Z}^{\sharp}(M)\ar[r]&\mathcal{N}.}
\end{equation*}

\subsection{Formal uniformization of $\mathcal{X}_{0}(N)$ and the special cycle $\mathcal{Z}(T,\boldsymbol{\varphi})$}
Let $B$ be the unique quaternion algebra over $\mathbb{Q}$ ramified exactly at $p$ and $\infty$. Then $B\otimes_{\mathbb{Q}}\mathbb{Q}_{p}\simeq\mathbb{B}$ is the unique division quaternion algebra over $\mathbb{Q}_{p}$. Let $\mathbb{E}$ be a supersingular elliptic curve over $\mathbb{F}$ and $\mathbb{X}=\mathbb{E}[p^{\infty}]$ is the $p$-divisible group of $\mathbb{E}$. Then $B\simeq\textup{End}^{0}(\mathbb{E})$ and $\mathbb{B}\simeq\textup{End}^{0}(\mathbb{X})$. Suppose $x_{0}\in\mathbb{B}$ comes from a cyclic $N$-isogeny of $\mathbb{E}$ under the above isomorphism $\textup{End}^{0}(\mathbb{E})\otimes_{\mathbb{Q}}\mathbb{Q}_{p}\simeq\mathbb{B}$.
\par
We first state and explain the formal uniformization theorem of the supersingular locus $\mathcal{H}_{\mathbb{F}_{p}}^{ss}$ of $\mathcal{H}_{\mathbb{F}_{p}}$. We use $\hat{\mathcal{H}}/_{(\mathcal{H}_{\mathbb{F}_{p}}^{ss})}$ to denote the completion of $\mathcal{H}$ along the closed substack $\mathcal{H}_{\mathbb{F}_{p}}^{ss}$.
\begin{theorem}
There is an isomorphism of formal stacks over $W$
\begin{equation}
    \hat{\mathcal{H}}/_{(\mathcal{H}_{\mathbb{F}_{p}}^{ss})}\underset{\sim}{\stackrel{\Theta_{\mathcal{H}}}\longrightarrow} B^{\times}(\mathbb{Q})_{0}^{2}\backslash[\mathcal{N}\times \textup{GL}_{2}(\mathbb{A}_{f}^{p})^{2}/\textup{GL}_{2}(\hat{\mathbb{Z}}^{p})^{2}],
    \label{HDI}
\end{equation}
where $B^{\times}(\mathbb{Q})_{0}$ is the subgroup of $B^{\times}(\mathbb{Q})$ consisting of elements whose norm has $p$-adic valuation 0.
\label{H}
\end{theorem}
Theorem \ref{H} is proved by Rapoport and Zink \cite[Theorem 6.24]{RZ96}. Here we only describe the isomorphism, especially the group action on the right hand side of (\ref{HDI}). Let $\eta_{0}^{p}: V^{p}(\mathbb{E})\stackrel{\sim}\rightarrow(\mathbb{A}_{f}^{p})^{2}$ be a prime-to-$p$ level structure of $\mathbb{E}$. Let $\tilde{\mathbb{E}}$ be a deformation of $\mathbb{E}$ to $W$, and let $\tilde{\mathbb{X}}\coloneqq\tilde{\mathbb{E}}[p^{\infty}]$ be the corresponding deformation of $\mathbb{X}$ to $W$. For some object $S\in \textup{Nilp}_{W}$, we pick an object $((X,\rho,\lambda),(X^{\prime},\rho^{\prime},\lambda^{\prime}),(g,g^{\prime}))\in \mathcal{N}(S)\times\textup{GL}_{2}(\mathbb{A}_{f}^{p})^{2}$. The quasi-isogeny $\rho$ (resp. $\rho^{\prime}$) gives rise to a quasi-isogeny $\tilde{\rho}: \tilde{\mathbb{X}}_{S}\rightarrow X$ (resp. $\tilde{\rho^{\prime}}: \tilde{\mathbb{X}}_{S}\rightarrow X^{\prime}$). Then there exists an elliptic curve $E$ (resp. $E^{\prime}$) up to prime-to-$p$ isogeny over $S$ and a quasi-isogeny $\rho_{E}: \tilde{\mathbb{E}}_{S}\rightarrow E$ (resp. $\rho_{E^{\prime}}: \tilde{\mathbb{E}}_{S}\rightarrow E^{\prime}$), such that $E[p^{\infty}]\simeq X$ (resp. $E^{\prime}[p^{\infty}]\simeq X^{\prime}$) and $\rho_{E}$ (resp. $\rho_{E^{\prime}}$) induces $\tilde{\rho}$ (resp. $\tilde{\rho^{\prime}}$) under this isomorphism. The object $((X,\rho,\lambda),(X^{\prime},\rho^{\prime},\lambda^{\prime}),(g,g^{\prime}))$ is mapped to 
\begin{equation*}
    ((E,E^{\prime}),(\overline{g^{-1}\eta_{0}^{p}\circ V^{p}(\rho_{E}^{-1})},\overline{g^{\prime-1}\eta_{0}^{p}\circ V^{p}(\rho_{E^{\prime}}^{-1})}))\in \mathcal{H}(S).
\end{equation*}
\par
The group action is given as follows, for a pair of elements $(b,b^{\prime})\in B^{\times}(\mathbb{Q})_{0}\times B^{\times}(\mathbb{Q})_{0}$, by the following map,
\begin{equation*}
    B(\mathbb{Q})\rightarrow B(\mathbb{Q}_{p})\simeq \textup{End}^{0}(\mathbb{X})\stackrel{\rho^{*}}\longrightarrow \textup{End}^{0}(X)\textup{ (resp. $\stackrel{\rho^{\prime*}}\longrightarrow \textup{End}^{0}(X^{\prime})$)},
\end{equation*}
and a fixed isomorphism $B(\mathbb{A}_{f}^{p})\simeq \textup{GL}_{2}(\mathbb{A}_{f}^{p})$. We obtain another triple
\begin{equation*}
    (b,b^{\prime})_{\ast}(((X,\rho,\lambda),(X^{\prime},\rho^{\prime},\lambda^{\prime}),(g,g^{\prime})))\coloneqq((X,\rho\circ b^{-1},\lambda),(X^{\prime},\rho^{\prime}\circ b^{\prime -1},\lambda^{\prime}),(bg,b^{\prime}g^{\prime})).
\end{equation*}
\par
Now let $\mathcal{X}_{0}(N)_{\mathbb{F}_{p}}^{ss}$ (resp. $\mathcal{Y}_{0}(N)_{\mathbb{F}_{p}}^{ss}$) be the supersingular locus of $\mathcal{X}_{0}(N)_{\mathbb{F}_{p}}$ (resp. $\mathcal{Y}_{0}(N)_{\mathbb{F}_{p}}$). Let $\hat{\mathcal{X}}_{0}(N)/_{(\mathcal{X}_{0}(N)_{\mathbb{F}_{p}}^{ss})}$ (resp. $\hat{\mathcal{Y}}_{0}(N)/_{(\mathcal{Y}_{0}(N)_{\mathbb{F}_{p}}^{ss})}$) be the completion of $\mathcal{X}_{0}(N)$ (resp. $\mathcal{Y}_{0}(N)$) along the closed substack $\mathcal{X}_{0}(N)_{\mathbb{F}_{p}}^{ss}$ (resp. $\mathcal{Y}_{0}(N)_{\mathbb{F}_{p}}^{ss}$). By the definition of $\mathcal{X}_{0}(N)$, we have $\mathcal{X}_{0}(N)_{\mathbb{F}_{p}}^{ss}=\mathcal{Y}_{0}(N)_{\mathbb{F}_{p}}^{ss}$ and therefore $\hat{\mathcal{X}}_{0}(N)/_{(\mathcal{X}_{0}(N)_{\mathbb{F}_{p}}^{ss})}\simeq\hat{\mathcal{Y}}_{0}(N)/_{(\mathcal{Y}_{0}(N)_{\mathbb{F}_{p}}^{ss})}$.
\begin{proposition}
There is an isomorphism of formal stacks over $W$,
\begin{equation}
    \hat{\mathcal{X}}_{0}(N)/_{(\mathcal{X}_{0}(N)_{\mathbb{F}_{p}}^{ss})}\underset{\sim}{\stackrel{\Theta_{\mathcal{X}_{0}(N)}}\longrightarrow} B^{\times}(\mathbb{Q})_{0}\backslash[\mathcal{N}_{0}(N)\times  \textup{GL}_{2}(\mathbb{A}_{f}^{p})/\Gamma_{0}(N)(\hat{\mathbb{Z}}^{p})],
    \label{modular}
\end{equation}
where $B^{\times}(\mathbb{Q})_{0}$ is the subgroup of $B^{\times}(\mathbb{Q})$ consisting of elements whose norm has $p$-adic valuation 0.
\label{uniformization}
\end{proposition}
\begin{proof}
The following diagram is Cartesian with all arrows closed immersions. 
\begin{equation*}
     \xymatrix{
    \mathcal{Y}_{0}(N)_{\mathbb{F}_{p}}^{ss}\ar[d]\ar[r]&\mathcal{H}_{\mathbb{F}_{p}}^{ss}\ar[d]\\
    \mathcal{Y}_{0}(N)\ar[r]&\mathcal{H}.}
\end{equation*}
this diagram gives a closed immersion $i:\hat{\mathcal{X}}_{0}(N)/_{(\mathcal{X}_{0}(N)_{\mathbb{F}_{p}}^{ss})}\simeq\hat{\mathcal{Y}}_{0}(N)/_{(\mathcal{Y}_{0}(N)_{\mathbb{F}_{p}}^{ss})}\rightarrow \hat{\mathcal{H}}/_{(\mathcal{H}_{\mathbb{F}_{p}}^{ss})}$. 
\par
Recall that we have the following isomorphism,
\begin{equation*}
    \hat{\mathcal{H}}/_{(\mathcal{H}_{\mathbb{F}_{p}}^{ss})}\underset{\sim}{\stackrel{\Theta_{\mathcal{H}}}\longrightarrow} B^{\times}(\mathbb{Q})_{0}^{2}\backslash[\mathcal{N}\times \textup{GL}_{2}(\mathbb{A}_{f}^{p})^{2}/\textup{GL}_{2}(\hat{\mathbb{Z}}^{p})^{2}].
\end{equation*}
Let $S$ be an object in $\textup{Nilp}_{W}$, let $(z,(g,g^{\prime}))\in \mathcal{N}(S)\times \textup{GL}_{2}(\mathbb{A}_{f}^{p})^{2}$ be a point in the closed formal substack $\mathcal{Y}_{0}(N)_{\mathbb{F}_{p}}^{ss}$, then clearly $z\in\mathcal{N}_{0}(N)(S)$. Suppose $z$ corresponds to a cyclic isogeny $E\stackrel{\pi}\longrightarrow E^{\prime}$ by our description of the isomorphism $\Theta_{\mathcal{H}}$, then $g^{\prime}$ is determined by $g$ by the following diagram,
\begin{equation}
    \begin{split}    
    \xymatrix{
    V^{p}(E_{\overline{s}})\ar[d]^{V^{p}(\pi)}\ar[rrr]^{g^{-1}\eta_{0}^{p}\circ V^{p}(\rho_{E}^{-1})}&&&(\mathbb{A}_{f}^{p})^{2}\ar[d]^{w_{N}}\\
    V^{p}(E_{\overline{s}}^{\prime})\ar[rrr]^{g^{\prime-1}\eta_{0}^{p}\circ V^{p}(\rho_{E^{\prime}}^{-1})}&&&(\mathbb{A}_{f}^{p})^{2}.}
    \end{split}
    \label{disa}
\end{equation}
thus we only focus on the pair $(z,g)\in\mathcal{N}_{0}(N)(S)\times \textup{GL}_{2}(\mathbb{A}_{f}^{p})$. Consider the following morphism
\begin{align*}
    \Theta: \mathcal{N}_{0}(N)\times \textup{GL}_{2}(\mathbb{A}_{f}^{p})&\longrightarrow \hat{\mathcal{H}}/_{(\mathcal{H}_{\mathbb{F}_{p}}^{ss})},\\
    (z,g)&\longmapsto \Theta_{\mathcal{H}}^{-1}(z,(g,g^{\prime})).
\end{align*}
the image of $\Theta$ lies in the closed formal substack $\hat{\mathcal{X}}_{0}(N)/_{(\mathcal{X}_{0}(N)_{\mathbb{F}_{p}}^{ss})}$.
\par
Let $(z_{1},g_{1})$, $(z_{2},g_{2})\in\mathcal{N}_{0}(N)(S)\times\textup{GL}_{2}(\mathbb{A}_{f}^{p})$ be two points, then $\Theta(z_{1},g_{1})=\Theta(z_{2},g_{2})$ if and only if there exists $b,b^{\prime}\in B^{\times}(\mathbb{Q})_{0}$ and $k_{1},k_{1}^{\prime}\in \textup{GL}_{2}(\hat{\mathbb{Z}}^{p})$ such that
$(z_{2},(g_{2},g_{2}^{\prime}))=((b,b^{\prime})_{\ast}z_{1},(bg_{1}k_{1},b^{\prime}g_{1}^{\prime}k_{1}^{\prime}))$. We still use $E\stackrel{\pi}\longrightarrow E^{\prime}$ to denote the corresponding point of $z_{2}$ under $\Theta_{\mathcal{H}}$. Notice that $(z_{2},(g_{2},g_{2}^{\prime}))=(z_{2},(bg_{1},b^{\prime}g_{1}^{\prime}))$ in the quotient stack $[\mathcal{N}\times\textup{GL}_{2}(\mathbb{A}_{f}^{p})^{2}/\textup{GL}_{2}(\hat{\mathbb{Z}}^{p})^{2}]$, therefore $\Theta_{\mathcal{H}}(z_{2},(g_{2},g_{2}^{\prime}))=$ $\linebreak\Theta_{\mathcal{H}}(z_{2},(bg_{1},b^{\prime}g_{1}^{\prime}))$ $\in \hat{\mathcal{X}}_{0}(N)/_{(\mathcal{X}_{0}(N)_{\mathbb{F}_{p}}^{ss})}(S)$, hence both $(g_{2}=bg_{1}k_{1},g_{2}^{\prime}=b^{\prime}g_{1}^{\prime}k_{1}^{\prime})$ and $(bg_{1},b^{\prime}g_{1}^{\prime})$ satisfy the commutative diagram (\ref{disa}), then
\begin{equation*}
    k_{1}^{\prime} = w_{N}k_{1}w_{N}^{-1},
\end{equation*}
since both $k_{1}$ and $k_{1}^{\prime}$ belongs to $\textup{GL}_{2}(\hat{\mathbb{Z}}^{p})\coloneqq\prod\limits_{v\neq\infty,p}\textup{GL}_{2}(\mathbb{Z}_{v})$, there exists $a,b,c,d\in\hat{\mathbb{Z}}^{p}$ such that
\begin{equation*}
    k_{1}=\begin{pmatrix}
    a & b\\
    Nc & d
    \end{pmatrix}\in \Gamma_{0}(N)(\hat{\mathbb{Z}}^{p}).
\end{equation*}
Moreover, the element $b^{\prime}$ is also determined by $b$ by the diagram (\ref{disa}). Therefore $\Theta(z_{1},g_{1})=\Theta(z_{2},g_{2})$ if and only if there exists $b\in B^{\times}(\mathbb{Q})_{0}$ and $k\in \Gamma_{0}(N)(\hat{\mathbb{Z}}^{p})$ such that $(z_{2},g_{2})=(b_{\ast}z_{1},bg_{1}k)$.
\end{proof}
Let $\hat{\mathcal{Z}}^{ss}(T,\boldsymbol{\varphi})$ be the completion of $\mathcal{Z}(T,\boldsymbol{\varphi})$ along its supersingular locus $\mathcal{Z}^{ss}(T,\boldsymbol{\varphi})\coloneqq$$\linebreak\mathcal{Z}(T,\boldsymbol{\varphi})\times_{\mathcal{X}_{0}(N)}\mathcal{X}_{0}(N)_{\mathbb{F}_{p}}^{ss}$. Let $\Delta(N)^{(p)}$ be the unique quadratic space over $\mathbb{Q}$ (up to isometry) such that: 1. It is positive definite at $\infty$; 2. For finite prime $l\neq p$, $\Delta(N)^{(p)}\otimes\mathbb{Q}_{l}$ is isometric to $\delta_{l}(N)\otimes\mathbb{Q}_{l}$; 3. $\Delta(N)^{(p)}\otimes\mathbb{Q}_{p}$ is isometric to $\mathbb{W}$. As a corollary of the formal uniformization of the supersingular locus of $\mathcal{X}_{0}(N)$ (cf. Proposition \ref{uniformization}), we have the following formal uniformization of the special cycles on $\mathcal{X}_{0}(N)$.
\begin{corollary}
Let $T\in\textup{Sym}_{2}(\mathbb{Q})$ be a nonsingular symmetric matrix, and $\textup{Diff}(T,\Delta(N))=\{p\}$. Let $\boldsymbol{\varphi}\in\mathscr{S}(\mathbb{V}_{f}^{2})$ be a $T$-admissible Schwartz function. Let $K_{0}^{\prime}(\hat{\mathcal{X}}_{0}(N)/_{(\mathcal{X}_{0}(N)_{\mathbb{F}_{p}}^{ss})})$ be the Grothendieck group of coherent sheaves of $\mathcal{O}_{\hat{\mathcal{X}}_{0}(N)/_{(\mathcal{X}_{0}(N)_{\mathbb{F}_{p}}^{ss})}}$-modules. Then we have the following identity in $K_{0}^{\prime}(\hat{\mathcal{X}}_{0}(N)/_{(\mathcal{X}_{0}(N)_{\mathbb{F}_{p}}^{ss})})$,
\begin{equation*}
    \hat{\mathcal{Z}}^{ss}(T,\boldsymbol{\varphi}) = \sum\limits_{\substack{\boldsymbol{x}\in B^{\times}(\mathbb{Q})_{0}\backslash (\Delta(N)^{(p)})^{2}\\ T(\boldsymbol{x})= T}}\sum\limits_{g\in B^{\times}_{\boldsymbol{x}}(\mathbb{Q})_{0}\backslash \textup{GL}_{2}(\mathbb{A}_{f}^{p})/\Gamma_{0}(N)(\hat{\mathbb{Z}}^{p})}\boldsymbol\varphi(g^{-1}\boldsymbol{x})\cdot \Theta_{\mathcal{X}_{0}(N)}^{-1}(\mathcal{Z}(\boldsymbol{x}),g),
\end{equation*}
where $B^{\times}_{\boldsymbol{x}}\subset B^{\times}$ is the stabilizer of $\boldsymbol{x}\in(\Delta(N)^{(p)})^{2}$.
\label{puni}
\end{corollary}
\begin{proof}
We only need to prove the corollary when $\boldsymbol{\varphi}$ is the characteristic function of some open compact subset $\boldsymbol{\omega}$ of $\mathbb{V}_{f}^{2}$. Let $S$ be an object in $\textup{Nilp}_{W}$. Suppose $\Theta_{\mathcal{X}_{0}(N)}^{-1}(z,g)\in \hat{\mathcal{Z}}^{ss}(T,\boldsymbol{\varphi})(S)$ for some $z\in \mathcal{N}_{0}(N)(S)$, then $z$ gives rise to a cyclic isogeny $E\stackrel{\pi}\longrightarrow E^{\prime}$, along with two isogenies $x_{1}, x_{2}\in \textup{Hom}(E,E^{\prime})_{(p)}$ such that
\begin{equation*}
    T = (\frac{1}{2}(x_{i},x_{j})),\,\,\textup{and}\,\,(x_{1},\pi) =(x_{2},\pi) = 0.
\end{equation*}
then $x_{1},x_{2}$ and $\pi$ induce endomorphisms of the correponding $p$-divisible groups, and hence endomorphims of $\mathbb{X}$. We still use $x_{1},x_{2}$ to denote the endomorphisms of $\mathbb{X}$, let $T(\boldsymbol{x})\coloneqq(\frac{1}{2}(x_{i},x_{j}))$ be the inner product matrix of $\boldsymbol{x}=(x_{1},x_{2}))$, we have
\begin{equation*}
    T = T(\boldsymbol{x}),\,\,\textup{and}\,\,(x_{1},x_{0}) =(x_{2},x_{0}) = 0.
\end{equation*}
i.e., $x_{1},x_{2}\in \{x_{0}\}^{\bot}=\mathbb{W}\simeq \Delta(N)^{(p)}\otimes_{\mathbb{Q}}\mathbb{Q}_{p}$. We can also identify $x_{1}$ and $x_{2}$ as elements in $\linebreak\Delta(N)^{(p)}\otimes\mathbb{A}_{f}^{p}$ via the level structures $\eta_{0}^{p}\circ V^{p}(\rho_{E}^{-1})$ and $\eta_{0}^{p}\circ V^{p}(\rho_{E^{\prime}}^{-1})$ of $E$ and $E^{\prime}$. The positivity assumption on $T$ makes it embeddable into $\Delta(N)^{(p)}\otimes_{\mathbb{Q}}\mathbb{R}$. By carefully choosing the isometry $\mathbb{W}\simeq \Delta(N)^{(p)}\otimes_{\mathbb{Q}}\mathbb{Q}_{p}$, we can find $\boldsymbol{x}\in (\Delta(N)^{(p)})^{2}$ which induces $x_{1}$ and $x_{2}$ locally at every place of $\mathbb{Q}$.
\par
Then the condition $\Theta_{\mathcal{X}_{0}(N)}^{-1}(z,g)\in \hat{\mathcal{Z}}^{ss}(T,\boldsymbol{\varphi})(S)$ implies that
\begin{equation*}
    z\in\mathcal{Z}(\boldsymbol{x})\,\,\textup{and}\,\, g^{-1}\boldsymbol{x}\in\boldsymbol{\omega} \,\,(\textup{here}\,\,g\in \textup{GL}_{2}(\mathbb{A}_{f})\,\,\textup{with}\,\,g_{p}=1).
\end{equation*}
and this is exactly the meaning of the identity in the theorem.
\end{proof}

\section{Difference formula at the geometric side}
\label{6}
\subsection{$p$-divisible groups over adic noetherian rings}
\begin{definition}
A topological ring $R$ is an adic noetherian ring if it is noetherian as a ring and it has a topological basis consisting of all translations of the neighborhoods of zero of the form $I^{n}$ $(n>0)$ where $I\subset R$ is a fixed ideal of $R$, and $R$ is Hausdorff and complete in that topology. A choice of such an ideal is said to be the defining ideal of the topological ring $R$.
\end{definition}
\begin{lemma}
Let $A$ be an adic noetherian local ring whose defining ideal is the maximal ideal $\mathfrak{m}$, then any ideal $I\subset A$ is complete in the topological ring $A$, i.e.,
\begin{equation*}
    I=\bigcap\limits_{n}(I+\mathfrak{m}^{n}).
\end{equation*}
Moreover, $A/I$ is an adic noetherian ring with defining ideal $\mathfrak{m}/I$.
\label{complete}
\end{lemma}
\begin{proof}
Nakayama lemma implies that $\bigcap\limits_{n}\mathfrak{m}^{n}I=0$, then we can apply \cite[Lemma 031B]{Stacks Project} to conclude that $I$ is $\mathfrak{m}$-adically complete, i.e., $I\simeq\hat{I}\coloneqq\varprojlim\limits_{n}I/\mathfrak{m}^{n}I$.
\par
We have the following exact sequence,
\begin{equation*}
    0\longrightarrow I\longrightarrow A\longrightarrow A/I\longrightarrow 0.
\end{equation*}
since $A$ is noetherian, after taking completion with respect to the maximal ideal $\mathfrak{m}$, we get
\begin{equation*}
    0\longrightarrow \hat{I}\longrightarrow \hat{A}\longrightarrow \widehat{A/I}\longrightarrow 0.
\end{equation*}
However, $A=\hat{A}$ and $I=\hat{I}$, hence $\widehat{A/I}\simeq A/I$. We conclude $A/I$ is an adic noetherian ring.
\par
By definition, $\widehat{A/I}=\varprojlim\limits_{n}A/(I+\mathfrak{m}^{n})$, hence $\widehat{A/I}\simeq A/I$ implies that $I=\bigcap\limits_{n}(I+\mathfrak{m}^{n})$.
\end{proof}
\begin{lemma}
Let $A$ be an adic noetherian ring whose defining ideal is $I$, then the following functor
\begin{align*}
    \left\{\textup{Category of $p$-divisible groups over}\,\,\textup{Spec}\,A\right\}&\rightarrow\left\{\textup{Category of $p$-divisible groups over}\,\, \textup{Spf}(A)\right\},\\
    G=(G_{n}/A)&\mapsto (G_{k}=(G_{k}(n)=G(n)\times_{A}A/I^{k}))_{k\geq1}.
\end{align*}
is an equivalence.
\label{equivalence}
\end{lemma}
\begin{proof}
This is proved by de Jong in \cite[Lemma 2.4.4]{deJ95}.
\end{proof}

\subsection{Difference Divisors on $\mathcal{N}$}
Recall that for every nonzero integral $x\in\mathbb{B}$, we define the special divisor $\mathcal{Z}^{\sharp}(x)$ on $\mathcal{N}$ as the closed formal subscheme of $\mathcal{N}$ over where $x$ lifts to an isogeny (cf. Definition \ref{specialcycle} and Proposition \ref{equation}). It is cut out by an element $f_{x}\in W[[t_{1},t_{2}]]$.
\par
For any nonzero $x\in\mathbb{B}$ such that $\nu_{p}(x^{\vee}\circ x)\geq2$, there is a closed immersion:
\begin{align*}
    i: \mathcal{Z}^{\sharp}(p^{-1}x)&\longrightarrow\mathcal{Z}^{\sharp}(x),
\end{align*}
by composing every deformation of $p^{-1}x$ with the multiplication-by-$p$ morphism. Since $W[[t_{1},t_{2}]]$ is a unique factorization domain, we get $f_{p^{-1}x}\vert f_{x}$. Define $d_{x}=f_{x}/f_{p^{-1}x}\in W[[t_{1},t_{2}]]$ when $\nu_{p}(x^{\vee}\circ x)\geq2$, $d_{x}=f_{x}$ when $\nu_{p}(x^{\vee}\circ x)=0$ or $1$.
\begin{definition}
Let $x\in\mathbb{B}$ be a nonzero and integral element. The difference divisor associated to $x$ to be
\begin{equation*}
    \mathcal{D}(x)=\textup{Spf}\,W[[t_{1},t_{2}]]/d_{x}.
\end{equation*}
\label{diff}
\end{definition}
The notion of difference divisor is first introduced by Terstiege in \cite{Ter11}. Proposition \ref{equation} implies that $p\nmid f_{x}$, hence $p\nmid d_{x}$, therefore the difference divisor $\mathcal{D}(x)$ is flat over $\textup{Spf}\,W$. The following theorem asserts that $\mathcal{D}(x)$ is regular.
\begin{theorem}
Let $x\in\mathbb{B}$ be a nonzero and integral element, $\mathfrak{m}=(p,t_{1},t_{2})$ is the maximal ideal of $W[[t_{1},t_{2}]]$, then $d_{x}\in\mathfrak{m}\backslash\mathfrak{m}^{2}$, i.e., the difference divisor $\mathcal{D}(x)$ is regular. Moreover, for any $i\geq1$, $d_{x}$ and $d_{p^{-i}x}$ are coprime to each other if $p^{-i}x$ is also integral.
\label{regularity}
\end{theorem}
\begin{proof}
Let $n\geq 0$ be the $p$-adic valuation of $x^{\vee}\circ x$. We first prove this result when $n=0$, in this case the result follows from the work of Li and Zhu \cite[Lemma 3.2.2]{LZ18} ($p$ odd) and \cite[Lemma 3.1]{Rap07} ($p=2$), and $W[[t_{1},t_{2}]]/f_{x}\simeq W[[t]]$ is even smooth over $W$.
\par
Now we suppose $n\geq 1$, we can always find an element $x^{\prime}\in\mathbb{B}$ such that $x^{\prime\vee}\circ x^{\prime}$ has $p$-adic valuation 0, and $(x,x^{\prime})=0$. We consider the formal closed subscheme $\mathcal{Z}^{\sharp}(x)\times_{\mathcal{N}}\mathcal{Z}^{\sharp}(x^{\prime})$, it is cut out by the ideal $(f_{x},f_{x^{\prime}})\subset\mathfrak{m}$; it is also a formal closed subscheme of $\mathcal{Z}^{\sharp}(x^{\prime})\simeq\textup{Spf}\,W[[t]]$ cut out by the image $\tilde{f}_{x}$ of $f_{x}$ under the surjective map $A\rightarrow W[[t]]$. By \cite[(5.10)]{GK93} (see also \cite[5.1]{LZ22}), we have the following decomposition of $\mathcal{Z}^{\sharp}(x)\times_{\mathcal{N}}\mathcal{Z}^{\sharp}(x^{\prime})$ into Cartier divisors on $\mathcal{Z}^{\sharp}(x^{\prime})$,
\begin{equation}
    \mathcal{Z}^{\sharp}(x)\times_{\mathcal{N}}\mathcal{Z}^{\sharp}(x^{\prime}) = \sum\limits_{i=0}^{[n/2]}\mathcal{Z}_{i},
    \label{formalcommul}
\end{equation}
each $\mathcal{Z}_{i}\simeq\textup{Spf}\,\mathcal{O}_{\Breve{K},i}$, where $\mathcal{O}_{\Breve{K},i}$ is the ring of integers of some nonarchimedean local field, hence it is a regular local ring, and they are different from each other. Let $d_{i}\in W[[t]]$ be the function defining the divisor $\mathcal{Z}_{i}$ on $\mathcal{Z}^{\sharp}(x^{\prime})$. Then we have the following identity,
\begin{equation}
    \tilde{f}_{x} = (\textup{unit})\times\prod\limits_{i=0}^{[n/2]}d_{i}.
    \label{de1}
\end{equation}
the regularity of $\mathcal{O}_{\Breve{K},i}$ implies that $d_{i}\in (p,t)\backslash (p,t)^{2}$.
\par
Let $\tilde{d}_{p^{-i}x}$ be the image of $d_{p^{-i}x}$ under the surjective map $A\rightarrow A/(f_{x^{\prime}})\simeq W[[t]]$. By definition we have $f_{x} = (\textup{unit})\times\prod\limits_{i=0}^{[n/2]}d_{p^{-i}x}$, therefore,
\begin{equation}
    \tilde{f}_{x} = (\textup{unit})\times\prod\limits_{i=0}^{[n/2]}\tilde{d}_{p^{-i}x}.
    \label{de2}
\end{equation}
\par
We induct on $n$ to conclude that $\tilde{d}_{x}=(\textup{unit})\times d_{[n/2]}\in (p,t)\backslash (p,t)^{2}$. When $n=1$, we simply get $\tilde{d}_{x}=(\textup{unit})\times d_{0}\in (p,t)\backslash (p,t)^{2}$. Let's assume the claim is true for $n<m$ for some $m\geq2$, we will prove the result for $n=m$. For this, we just need to compare (\ref{de1}) and (\ref{de2}) for $p^{-1}x$ and $x$.
\par
Therefore we have proved that $A/(d_{x},f_{x^{\prime}})$ is a regular local ring, hence we conclude that $d_{x}\in\mathfrak{m}\backslash\mathfrak{m}^{2}$, and $\mathcal{D}(x)\simeq\textup{Spf}\,A/(d_{x})$ is regular. Moreover, since every piece on the right hand side of (\ref{formalcommul}) is different from each other, we conclude that $d_{x}$ and $d_{p^{-i}x}$ are coprime to each other.
\end{proof}
Fix a $N$-isogeny $x_{0}\in\mathbb{B}$, recall that we have defined the deformation functor $\mathcal{N}_{0}(N)$ in $\S$\ref{locdeform}. Compare the moduli intepretations of $\mathcal{N}_{0}(N)$ and $\mathcal{Z}^{\sharp}(x_{0})$, we have a natural functor,
\begin{align*}
    i: \mathcal{N}_{0}(N) &\longrightarrow \mathcal{Z}^{\sharp}(x_{0}),\\
    (X\stackrel{x\,\,\textup{cyclic}}\longrightarrow X^{\prime},(\rho,\rho^{\prime}),(\lambda,\lambda^{\prime}))&\longmapsto(X\stackrel{x}\rightarrow X^{\prime},(\rho,\rho^{\prime}),(\lambda,\lambda^{\prime})).
\end{align*}
\begin{theorem}
The natural functor $i$ is a closed immersion, and induces an isomorphism:
\begin{equation*}
    \mathcal{N}_{0}(N)\stackrel{\sim}\longrightarrow \mathcal{D}(x_{0}).
\end{equation*}
\label{keydecomgeo}
\end{theorem}
\begin{proof}
By Proposition \ref{equation}, $\mathcal{Z}^{\sharp}(x_{0})$ is represented by $\textup{Spf}\,W[[t_{1},t_{2}]]/f_{x_{0}}$. We consider the maximal ideal $\mathfrak{m}=(p,t_{1},t_{2})$ of $W[[t_{1},t_{2}]]$ and a projective system of rings $\varprojlim_{n} R_{n}$ where $R_{n}=W[[t_{1},t_{2}]]/(f_{x_{0}}+\mathfrak{m}^{n})$. We use $(X_{n}\stackrel{x_{n}}\rightarrow X^{\prime}_{n},(\rho_{n},\rho^{\prime}_{n}),(\lambda_{n},\lambda^{\prime}_{n}))$ to denote the corresponding object in $\mathcal{Z}^{\sharp}(x_{0})(R_{n})$ by the natural morphism $W[[t_{1},t_{2}]]/f_{x_{0}}\rightarrow R_{n}$, which is essentially the base change from $\mathcal{Z}^{\sharp}(x_{0})$ to $\textup{Spec}\,R_{n}$ of the universal object $(X^{\textup{univ}}\stackrel{x_{0}^{\textup{univ}}}\rightarrow X^{\prime\textup{univ}},(\rho^{\textup{univ}},\rho^{\prime\textup{univ}}),(\lambda^{\textup{univ}},\lambda^{\prime\textup{univ}}))$. The following diagram is commutative,
\begin{equation*}
     \xymatrix{
    X_{n}\ar[d]^{x_{n}}\ar[r]&X_{n+1}\ar[d]^{x_{n+1}}\\
    X_{n}^{\prime}\ar[r]&X_{n+1}^{\prime}.}
\end{equation*}
By \cite[Lemma 2.4.4]{deJ95}, ${x_{n}}$ fits together to be an isogeny of $p$-divisible groups $x_{0}^{\textup{univ}}:X^{\textup{univ}}\rightarrow X^{\prime\textup{univ}}$ over $\textup{Spec}\,W[[t_{1},t_{2}]]/f_{x_{0}}$.
\par
Now we apply Serre-Tate theorem (cf. \cite{CS64}) to the projective system $\varprojlim_{n} R_{n}$, we obtain a direct system of elliptic curves $E_{n},E^{\prime}_{n}$ over $\textup{Spec}\,R_{n}$ and $\tilde{x}_{n}\in\textup{End}_{R_{n}}(E_{n},E^{\prime}_{n})$ such that,\\
(i) There exist isomorphisms $i_{n}:E_{n}[p^{\infty}]\simeq X_{n}$ and $i_{n}^{\prime}:E_{n}^{\prime}[p^{\infty}]\simeq X_{n}^{\prime}$,\\
(ii) $x_{n} = i_{n}^{\prime}\circ \tilde{x}_{n}[p^{\infty}]\circ i_{n}^{-1}$.
\par
Since every elliptic curve is equipped with a canonical ample line bundle given by the unit section, we can apply Grothendieck's algebraization theorem (cf. \cite[Theorem 089A, Lemma 0A42]{Stacks Project}) to obtain a triple $(E^{\textup{univ}}\stackrel{\tilde{x}_{0}^{\textup{univ}}}\rightarrow E^{\prime\textup{univ}},(\rho^{\textup{univ}},\rho^{\prime\textup{univ}}),(\lambda^{\textup{univ}},\lambda^{\prime\textup{univ}}))$ where $E^{\textup{univ}}$ and $E^{\prime\textup{univ}}$ are two elliptic curves over $\textup{Spec}\,W[[t_{1},t_{2}]]/f_{x_{0}}$ with the following isomorphisms
\begin{equation*}
    i^{\textup{univ}}: E^{\textup{univ}}[p^{\infty}]\simeq X^{\textup{univ}},\,\,i^{\prime\textup{univ}}: E^{\prime\textup{univ}}[p^{\infty}]\simeq X^{\prime\textup{univ}}.
\end{equation*}
and $x_{0}^{\textup{univ}} = i^{\prime\textup{univ}}\circ \tilde{x}_{0}^{\textup{univ}}[p^{\infty}]\circ (i^{\textup{univ}})^{-1}$. Then we have
\begin{equation*}
    \textup{ker}(x_{0}^{\textup{univ}})\simeq\textup{ker}(\tilde{x}^{\textup{univ}}[p^{\infty}])=\textup{ker}(\tilde{x}_{0}^{\textup{univ}})[p^{\infty}]\hookrightarrow E^{\textup{univ}}.
\end{equation*}
Where $\textup{ker}(\tilde{x}_{0}^{\textup{univ}})[p^{\infty}]$ is the $p$-torsion subgroup scheme of the finite locally free group scheme $\textup{ker}(\tilde{x}_{0}^{\textup{univ}})$. Therefore the universal kernel $\textup{ker}(x_{0}^{\textup{univ}})$ is embedded into an elliptic curve, we can apply Proposition \ref{cyclo} and conclude that there is an ideal $\mathcal{I}^{cyc}(x_{0})\subset W[[t_{1},t_{2}]]$ containing $f_{x_{0}}$ such that for an object $(X\stackrel{x}\rightarrow X^{\prime},(\rho,\rho^{\prime}),(\lambda,\lambda^{\prime}))\in \mathcal{Z}^{\sharp}(x_{0})(S)$ where $S\in\textup{Nilp}_{W}$, the isogeny $x$ is a cyclic isogeny if and only if the morphism $S\rightarrow\textup{Spf}\,W[[t_{1},t_{2}]]/f_{x_{0}}$ factors through $\textup{Spf}\,W[[t_{1},t_{2}]]/\mathcal{I}^{cyc}(x_{0})$. We conclude from here that $\mathcal{N}_{0}(N)$ is represented by the formal scheme $\textup{Spf}\,W[[t_{1},t_{2}]]/\mathcal{I}^{cyc}(x_{0})$ and the natural functor $i$ is a closed immersion. 
\par
Recall that we use $d_{x_{0}}$ to denote the equation that cuts out the difference divisor $\mathcal{D}(x_{0})$. In the following we use $\mathcal{D}$ to denote the difference divisor $\mathcal{D}(x_{0})$. Let $x_{\mathcal{D}}:X_{\mathcal{D}}\rightarrow X^{\prime}_{\mathcal{D}}$ be the base change of $x_{0}^{\textup{univ}}: X^{\textup{univ}}\rightarrow X^{\prime\textup{univ}}$ to $\mathcal{D}$. We first show that $x_{\mathcal{D}}$ doesn't factor through the multiplication-by-$p$ morphism of $X_{\mathcal{D}}$. Let's assume the converse, i.e., $x_{\mathcal{D}}=p\circ x^{\prime}_{\mathcal{D}}$ where $x^{\prime}_{\mathcal{D}}:X_{\mathcal{D}}\rightarrow X^{\prime}_{\mathcal{D}}$ is an isogeny. Let $\mathcal{D}_{n}=\textup{Spec}\,W[[t_{1},t_{2}]]/(d_{x_{0}}+\mathfrak{m}^{n})$, the base change of $x^{\prime}_{\mathcal{D}}$ from $\mathcal{D}$ to $\mathcal{D}_{n}$ is a deformation of $p^{-1}x_{0}$, hence the natural morphism $\mathcal{D}_{n}\rightarrow\mathcal{Z}^{\sharp}(x_{0})$ factors through $\mathcal{Z}^{\sharp}(p^{-1}x_{0})\simeq\textup{Spf}\,W[[t_{1},t_{2}]]/(f_{p^{-1}x_{0}})$. Since $W[[t_{1},t_{2}]]/(d_{x_{0}})\simeq\varprojlim\limits_{n}W[[t_{1},t_{2}]]/(d_{x_{0}}+\mathfrak{m}^{n})$ by Lemma \ref{complete}, we get a ring homomorphism $W[[t_{1},t_{2}]]/(f_{p^{-1}x_{0}})\rightarrow W[[t_{1},t_{2}]]/(d_{x_{0}})$. However, $d_{x_{0}}$ is coprime to $f_{p^{-1}x_{0}}$ by Theorem \ref{regularity}, this is a contradiction. Hence $x_{\mathcal{D}}$ doesn't factor through the multiplication-by-$p$ morphism of $X_{\mathcal{D}}$.
\par
Lemma \ref{cyclemma} and Corollary \ref{corcyc} imply that $\textup{ker}(x_{\mathcal{D}})$ is a cyclic group scheme since $\mathcal{D}$ is an integral noetherian scheme which is also separated and flat over $W$, hence there exists a natural morphism from $\textup{Spec}\,W[[t_{1},t_{2}]]/d_{x_{0}}$ to $\textup{Spec}\,W[[t_{1},t_{2}]]/\mathcal{I}^{cyc}(x_{0})$. Therefore we conclude that $\mathcal{I}^{cyc}(x_{0})\subset(d_{x_{0}})\subset W[[t_{1},t_{2}]]$. This shows that the closed immersion $\mathcal{D}(x_{0})\rightarrow \mathcal{Z}^{\sharp}(x_{0})$ decomposes in the following way:
\begin{equation*}
    \mathcal{D}(x_{0})\rightarrow \mathcal{N}_{0}(N)\rightarrow \mathcal{Z}^{\sharp}(x_{0}).
\end{equation*}
\par
Therefore we have an inclusion of ideals, $(f_{x_{0}})\subset\mathcal{I}^{cyc}(x_{0})\subset (d_{x_{0}})\in W[[t_{1},t_{2}]]$. \cite[Theorem 6.6.1]{KM85} (see also Case II of 5.3.2.1 of \textit{loc.cit}) asserts that $W[[t_{1},t_{2}]]/\mathcal{I}^{cyc}(x_{0})$ is a 2-dimensional regular local ring. Recall that we have already proved in Theorem \ref{regularity} that $W[[t_{1},t_{2}]]/d_{x_{0}}$ is also a regular local ring, hence we must have $\mathcal{I}^{cyc}(x_{0})=(d_{x_{0}})$, i.e., $\mathcal{D}(x_{0})\simeq\mathcal{N}_{0}(N)$. 
\end{proof}
\subsubsection{Special Fibers}
In this part we use the identification $\mathcal{N}_{0}(N)\stackrel{\sim}\longrightarrow\mathcal{D}(x_{0})$ to explicitly describe the special fiber of the local ring $\mathcal{N}_{0}(N)$. The main results of this part will not be used in the following calculations, readers can skip on first reading.
\par
Let $\mathfrak{a}=(t_{1},t_{2})\subset W[[t_{1},t_{2}]]$. Let $\overline{\mathfrak{a}}$ be the image of $\mathfrak{a}$ in $\mathbb{F}[[t_{1},t_{2}]]$. Let $A_{n}=W[[t_{1},t_{2}]]/\mathfrak{a}^{n}$ and $R_{n}=\mathbb{F}[[t_{1},t_{2}]]/\overline{\mathfrak{a}}^{n}$. Let $A_{0}=W[[t_{1},t_{2}]]$ and $R_{0}=\mathbb{F}[[t_{1},t_{2}]]$. Equip each $A_{n}$ with a morphism $\sigma$ which extends the Frobenius morphism on $W$ and maps $t_{1}$ to $t_{1}^{p}$, $t_{2}$ to $t_{2}^{p}$. Then $A_{n}$ is a frame for $R_{n}$ in the sense of \cite[Definition 1]{Zin01}. For any $n\geq0$, let  $(M,M_{1},\Phi)$ be an $A_{n}$-window in the sense of \cite[Definition 2]{Zin01}. Since $\Phi(M_{1})\subset p\cdot M$ and $p$ is not a zero-divisor in $A_{n}$, we define $\Phi_{1}: M_{1}\rightarrow M$ to be $p^{-1}\Phi$. The morphism $\Phi_{1}$ is $\sigma$-linear and induces an isomorphism $\Phi_{1}^{\sigma}: M_{1}^{\sigma}\rightarrow M$ because both sides are free $A_{n}$-module of the same rank and $\Phi_{1}^{\sigma}$ is surjective by the definition of windows (\cite[Definition 2(ii)]{Zin01}). Let $\alpha$ be the following injective $A_{n}$-morphism,
\begin{equation*}
    \alpha: M_{1}\hookrightarrow M\stackrel{(\Phi_{1}^{\sigma})^{-1}}\longrightarrow M_{1}^{\sigma}.
\end{equation*}
\begin{theorem}
For any $n\geq0$, we have the following category equivalences,
\begin{equation*}
    \left\{\textup{$A_{n}$-window $(M,M_{1},\Phi)$}\right\}\stackrel{\sim}\longleftrightarrow\left\{\textup{formal $p$-divisible groups over $R_{n}$}\right\}
\end{equation*}
Moreover, both these two categories are equivalent to the following category.
\begin{equation*}
    \{\textup{pairs $(M_{1},\alpha: M_{1}\rightarrow M_{1}^{\sigma})$,}\,\,\textup{such that Coker($\alpha$) is a free $R_{n}$-module.}\}
\end{equation*}
where the functor from $A_{n}$-window $(M,M_{1},\Phi)$ to pairs $(M_{1},\alpha: M_{1}\rightarrow M_{1}^{\sigma})$ is given by the constructions above.
\label{window}
\end{theorem}
\begin{proof}
This is proved in \cite[Theorem 4]{Zin01}. 
\end{proof}
Let $((\overline{X},\overline{\rho},\overline{\lambda}),(\overline{X}^{\prime},\overline{\rho}^{\prime},\overline{\lambda}^{\prime}))$ be the universal object in $\mathcal{N}(\mathbb{F}[[t_{1},t_{2}]])$, i.e., the base change of the universal object $((X^{\textup{univ}}, \rho^{\textup{univ}},\lambda^{\textup{univ}}),(X^{\prime\textup{univ}},\rho^{\prime\textup{univ}},\lambda^{\prime\textup{univ}}))$ over $W[[t_{1},t_{2}]]$ to $\mathbb{F}[[t_{1},t_{2}]]$. The corresponding window can be described as follows, let $\mathbb{D}=W\cdot e+W\cdot f$ be the Dieudonne module of $\mathbb{X}$, where $Fe=Ve=f$, $Ff=Vf=p\cdot e$ ($F$, $V$ are the Frobenius and Verschiebung morphisms on $\mathbb{D}$). Then we let $M=\mathbb{D}\otimes_{W}W[[t]]$, and $M_{1}=(W\cdot f+pW\cdot e)\otimes_{W}W[[t]]$. We still use $\sigma$ to denote the Frobenius action on $W[[t]]$ which sends $t$ to $t^{p}$ and extends the Frobenius morphism on $W$. Let $\Phi$ be the $\sigma$-linear map from $M$ to $M$ such that $\Phi(e\otimes1)=t\cdot(e\otimes1)+f\otimes1,\Phi(f\otimes1)=p\cdot (e\otimes1)$. Then $(M,M_{1},\Phi)$ is the $W[[t]]$-window corresponding to the universal deformation of $\mathbb{X}$ over $\mathbb{F}[[t]]$ (cf. \cite[(86)]{Zin02}). Let $(M^{\prime},M_{1}^{\prime},\Phi^{\prime})$ be the correpsonding window for $\mathbb{X}^{\prime}$, then the $W[[t_{1},t_{2}]]$-window corresponding to the universal deformation of $\mathbb{X}\times_{\mathbb{F}}\mathbb{X}^{\prime}$ over $\mathbb{F}[[t_{1},t_{2}]]$ is given by $(M\oplus M^{\prime},M_{1}\oplus M_{1}^{\prime},\Phi\oplus\Phi^{\prime})$, or $(M_{1}\oplus M_{1}^{\prime},\alpha)$ where under the basis $\{p\cdot(e\otimes1),f\otimes1,p(e^{\prime}\otimes1),f^{\prime}\otimes1\}$, the map $\alpha$ is given by the following matrix
\begin{equation*}
   \alpha = \begin{pmatrix}
      & 1 &  & \\
    p & -t_{1} & & \\
     & & & 1\\
     & & p & -t_{2}
    \end{pmatrix}.
\end{equation*}
Any quasi-isogeny $x\in\mathbb{B}$ induces the following endomorphism of the window $M_{1}\oplus M_{1}^{\prime}$ of $\mathbb{X}\times_{\mathbb{F}}\mathbb{X}^{\prime}$ under the basis $\{p\cdot e,f,p\cdot e^{\prime},f^{\prime}\}$,
\begin{equation*}
   \mathbb{D}(x) = \begin{pmatrix}
      &  & \sigma(a) & -\sigma(b)\\
     & & -p\cdot b & a\\
    a & \sigma(b) & & \\
    p\cdot b & \sigma(a) &  &
    \end{pmatrix},
\end{equation*}
where $a,b\in\mathbb{Q}_{p^{2}}$.
\par
Let $M_{1}(n)=M_{1}\otimes_{A_{0}}A_{n}$, $M_{1}^{\prime}(n)=M_{1}^{\prime}\otimes_{A_{0}}A_{n}$, $\alpha(n)=\alpha\otimes_{A_{0}}A_{n}$. By Theorem \ref{window}, a quasi-isogeny $x$ lifts to an isogeny over $R_{n}$ if and only if there exists $x(n)\in \textup{End}((M_{1}(n)\oplus M_{1}^{\prime}(n),\alpha(n)))$ such that $x(1)=\mathbb{D}(x)$ and the following diagram commutes
\begin{equation*}
    \xymatrix{
    M_{1}(n)\oplus M_{1}^{\prime}(n)\ar[d]^{x(n)}\ar[r]^{\alpha(n)}&M_{1}(n)^{\sigma}\oplus M_{1}^{\prime}(n)^{\sigma}\ar[d]^{\sigma(x(n))}\\
    M_{1}(n)\oplus M_{1}^{\prime}(n)\ar[r]^{\alpha(n)}&M_{1}(n)^{\sigma}\oplus M_{1}^{\prime}(n)^{\sigma}.}
\end{equation*}
Under the basis $\{p\cdot(e\otimes1),f\otimes1,p(e^{\prime}\otimes1),f^{\prime}\otimes1\}$, the morphism $x(n)$ has the following form
\begin{equation*}
    x(n)=\begin{pmatrix}
     A(n) & Y(n)\\
     X(n) & B(n)
    \end{pmatrix}.
\end{equation*}
where $X(n),Y(n), A(n), B(n)\in\textup{M}_{2}(A_{n})$ satisfy the following equations,
\begin{align*}
    X(n)&=p^{-1}U^{\prime}(t_{2})\cdot\sigma(X(n))\cdot U(t_{1}),\,\,\,Y(n)=p^{-1}U^{\prime}(t_{1})\cdot\sigma(Y(n))\cdot U(t_{2});\\
    A(n)&=p^{-1}U^{\prime}(t_{1})\cdot\sigma(A(n))\cdot U(t_{1}),\,\,\,B(n)=p^{-1}U^{\prime}(t_{2})\cdot\sigma(B(n))\cdot U(t_{2}).
\end{align*}
where $U(t)=\begin{pmatrix}
     & 1\\
     p & -t
    \end{pmatrix}$ and $U^{\prime}(t)=\begin{pmatrix}
     t & 1\\
     p & 
    \end{pmatrix}$. Since $A(1) = B(1) =0$, we conclude (by comparing degrees of $t_{1}$ and $t_{2}$) that $A(n)=B(n)=0$.
\par
For any $A\in \textup{M}_{2}(A_{n}\otimes_{\mathbb{Z}}\mathbb{Q})$, the matrix $\sigma(A)$ is a well-defined element in $\textup{M}_{2}(A_{pn}\otimes_{\mathbb{Z}}\mathbb{Q})$. Therefore, starting from $X(1)$ and $Y(1)$, we can define successively 
\begin{equation}
    X(p^{l+1})=p^{-1}U^{\prime}(t_{2})\cdot\sigma(X(p^{l}))\cdot U(t_{1}),\,\,\,Y(p^{l+1})=p^{-1}U^{\prime}(t_{1})\cdot\sigma(Y(p^{l}))\cdot U(t_{2}).
    \label{recursion}
\end{equation}
Since the local ring $\mathcal{O}_{\mathcal{Z}(x)}$ only depends (up to noncanonical isomorphisms) on the valuation of $x$, we will take the following specific choice of $x$ and $\mathbb{D}(x)$ in the following computations.\\
$\bullet$ When $\textup{ord}_{p}(x^{\vee}\circ x)=2k$ for some $k\geq0$, we take 
\begin{equation*}
    X(1)=Y(1)=\begin{pmatrix}
     p^{k} & \\
      &p^{k}
    \end{pmatrix}.
\end{equation*}
By computation based on the recursion formula (\ref{recursion}), it turns out that for any $l\geq 1$,
\begin{align*}
    X(p^{l})&=\frac{1}{p^{l-k}}\left(\begin{pmatrix}
     0 & (-1)^{l-1}(t_{1}t_{2})^{\frac{p^{l-1}-1}{p-1}}(t_{2}^{p^{l-1}}-t_{1}^{p^{l-1}})\\
     0 &0
    \end{pmatrix}+p\cdot C\right);\\
    Y(p^{l})&=\frac{1}{p^{l-k}}\left(\begin{pmatrix}
     0 & (-1)^{l-1}(t_{1}t_{2})^{\frac{p^{l-1}-1}{p-1}}(t_{2}^{p^{l-1}}-t_{1}^{p^{l-1}})\\
     0 &0
    \end{pmatrix}+p\cdot D\right).
\end{align*}
for some matrices $C,D\in\textup{M}_{2}(A_{p^{l}})$.\\\\
$\bullet$ When $\textup{ord}_{p}(x^{\vee}\circ x)=2k+1$ for some $k\geq0$, we take 
\begin{equation*}
    X(1)=-Y(1)=\begin{pmatrix}
      & p^{k}\\
     p^{k+1} &
    \end{pmatrix}.
\end{equation*}
By computation based on the recursion formula (\ref{recursion}), it turns out that for any $l\geq 1$, 
\begin{align*}
    X(p^{l})&=\frac{1}{p^{l-k}}\left(\begin{pmatrix}
    0  & (-1)^{l}(t_{1}t_{2})^{\frac{p^{l}-1}{p-1}}\\
    0  &0
    \end{pmatrix}+p\cdot C^{\prime}\right);\\
    Y(p^{l})&=\frac{1}{p^{l-k}}\left(\begin{pmatrix}
     0 & (-1)^{l-1}(t_{1}t_{2})^{\frac{p^{l}-1}{p-1}}\\
     0 &0
    \end{pmatrix}+p\cdot D^{\prime}\right).
\end{align*}
for some matrices $C^{\prime},D^{\prime}\in\textup{M}_{2}(A_{p^{l}})$.\\
\begin{proposition}
Let $x\in\mathbb{B}$ be an integral nonzero element which has valuation $n$ and induces $X(1)$ and $Y(1)$ as described above, let $f_{x}\in W[[t_{1},t_{2}]]$ be the element cutting out $\mathcal{Z}(x)$, then
\begin{equation}
    \overline{f}_{x}\coloneqq f_{x}\,\,\,\textup{mod $p$}=(\textup{unit})\times\begin{cases}
    (t_{1}t_{2})^{\frac{p^{n/2}-1}{p-1}}(t_{2}^{p^{n/2}}-t_{1}^{p^{n/2}}) \,\,\,\textup{mod $(t_{1},t_{2})^{p^{n/2+1}}$}& \textup{when $n$ is even;}\\
    (t_{1}t_{2})^{\frac{p^{(n+1)/2}-1}{p-1}} \,\,\,\,\,\,\,\textup{mod $(t_{1},t_{2})^{p^{(n+1)/2}}$}& \textup{when $n$ is odd.}
    \end{cases}
    \label{modpdec}
\end{equation}
\end{proposition}
\begin{proof}
By the above formula for $X(p^{l})$ and $Y(p^{l})$, we can conclude that $x$ can be lifted to an isogeny over $R_{p^{[n/2]}}$ but not over $R_{p^{[n/2]+1}}$. Then the formula for $X(p^{[n/2]+1})$ and $Y(p^{[n/2]+1})$ imply the equation above.
\end{proof}
\begin{theorem}
Let $x\in\mathbb{B}$ be an integral nonzero element which has valuation $n$ and induces $X(1)$ and $Y(1)$ as described above, let $f_{x}\in W[[t_{1},t_{2}]]$ be the element cutting out $\mathcal{Z}(x)$, then $\overline{f}_{x}$ is divisible by 
\begin{equation*}
    t_{1}-t_{2}^{p^{a}},\,\,\, t_{1}^{p^{a}}-t_{2}\,\,\,\textup{for $0\leq a\leq n$ and $a\equiv n$ mod $2$}.
\end{equation*}
Moreover, $\overline{f}_{x}$ has no other irreducible factors and the multiplicity of $t_{1}-t_{2}^{p^{a}}, t_{1}^{p^{a}}-t_{2}$ in $\overline{f}_{x}$ is $p^{(n-a)/2}$.
\label{spci}
\end{theorem}
\begin{proof}
We first prove that $t_{1}^{p^{k_{1}}}-t_{2}^{p^{k_{2}}}$ divides $\overline{f}_{x}$, where $k_{1},k_{2}\geq0$ and $k_{1}+k_{2}=n$. We will prove this by showing that $X(p^{l}),Y(p^{l})\,\,\textup{mod}\,(t_{1}^{p^{k_{1}}}-t_{2}^{p^{k_{2}}})\in\textup{M}_{2}(A_{p^{l}}/(t_{1}^{p^{k_{1}}}-t_{2}^{p^{k_{2}}}))$ for any $l\geq 0$.\\
$\bullet$ When $n=2k$ is even, the recursion formula (\ref{recursion}) implies that,
\begin{align*}
    X(p^{l}) & = p^{k-l}U^{\prime}(t_{2})U^{\prime}(t_{2}^{p})\cdot\cdot\cdot U^{\prime}(t_{2}^{p^{l-1}})U(t_{1}^{p^{l-1}})\cdot\cdot\cdot U(t_{1}^{p})U(t_{1});\\
    Y(p^{l}) & = p^{k-l}U^{\prime}(t_{1})U^{\prime}(t_{1}^{p})\cdot\cdot\cdot U^{\prime}(t_{1}^{p^{l-1}})U(t_{2}^{p^{l-1}})\cdot\cdot\cdot U(t_{2}^{p})U(t_{2}).
\end{align*}
Let's assume $k_{1}\leq k_{2}$. For any $0\leq t\leq l-k_{1}$, we have the relation $t_{2}^{p^{l-t}}=t_{1}^{p^{k_{2}-k_{1}+l-t}}$, hence $U^{\prime}(t_{2}^{p^{l-t}})U(t_{1}^{p^{k_{2}-k_{1}+l-t}})=p\cdot I_{2}$. Moreover, when $1\leq t\leq k_{2}-k_{1}$, $t_{2}^{p^{l-t}}=t_{1}^{p^{k_{2}-k_{1}+l-t}}=0$, hence $U(t_{2}^{p^{l-t}})=U(0)$.
\begin{align*}
    X(p^{l}) & = U^{\prime}(t_{2})U^{\prime}(t_{2}^{p})\cdot\cdot\cdot U^{\prime}(t_{2}^{p^{k_{2}-1}})U(t_{1}^{p^{k_{1}-1}})\cdot\cdot\cdot U(t_{1}^{p})U(t_{1})\in\textup{M}_{2}(A_{p^{l}}/(t_{1}^{p^{k_{1}}}-t_{2}^{p^{k_{2}}}));\\
    Y(p^{l}) & = U^{\prime}(t_{1})U^{\prime}(t_{1}^{p})\cdot\cdot\cdot U^{\prime}(t_{1}^{p^{k_{1}-1}})U(t_{2}^{p^{k_{2}-1}})\cdot\cdot\cdot U(t_{2}^{p})U(t_{2})\in\textup{M}_{2}(A_{p^{l}}/(t_{1}^{p^{k_{1}}}-t_{2}^{p^{k_{2}}})).
\end{align*}
The proof of the case $k_{1}>k_{2}$ is similar and we get the same formula for $X(p^{l})$ and $Y(p^{l})$ as above, therefore we conclude that $t_{1}^{p^{k_{1}}}-t_{2}^{p^{k_{2}}}$ divides $\overline{f}_{x}$ when $k_{1},k_{2}\geq0$ and $k_{1}+k_{2}=2k$ by Theorem \ref{window}, hence $\overline{f}_{x}$ is divisible by the following polynomial,
\begin{equation*}
    (t_{1}-t_{2})^{p^{k}}\cdot\prod\limits_{a=1}^{k}\left((t_{1}-t_{2}^{p^{2a}})(t_{2}-t_{1}^{p^{2a}})\right)^{p^{k-a}}.
\end{equation*}
We also know that $(t_{1}-t_{2})^{p^{k}}\cdot\prod\limits_{a=1}^{k}\left((t_{1}-t_{2}^{p^{2a}})(t_{2}-t_{1}^{p^{2a}})\right)^{p^{k-a}}\equiv (t_{1}t_{2})^{\frac{p^{k}-1}{p-1}}\cdot(t_{2}^{p^{k}}-t_{1}^{p^{k}})$ mod $(t_{1},t_{2})^{p^{k+1}}$, the lemma follows by comparing this formula with (\ref{modpdec}).\\
$\bullet$ When $n=2k+1$ is odd, the recursion formula (\ref{recursion}) implies that,
\begin{align*}
    X(p^{l}) & = p^{k-l}U^{\prime}(t_{2})U^{\prime}(t_{2}^{p})\cdot\cdot\cdot U^{\prime}(t_{2}^{p^{l-1}})\begin{pmatrix}0 & 1\\p & 0\end{pmatrix} U(t_{1}^{p^{l-1}})\cdot\cdot\cdot U(t_{1}^{p})U(t_{1});\\
    Y(p^{l}) & = p^{k-l}U^{\prime}(t_{1})U^{\prime}(t_{1}^{p})\cdot\cdot\cdot U^{\prime}(t_{1}^{p^{l-1}})\begin{pmatrix}0 & 1\\p & 0\end{pmatrix} U(t_{2}^{p^{l-1}})\cdot\cdot\cdot U(t_{2}^{p})U(t_{2}).
\end{align*}
Let's assume $k_{1}< k_{2}$. For any $0\leq t\leq l-k_{1}$, we have the relation $t_{2}^{p^{l-t}}=t_{1}^{p^{k_{2}-k_{1}+l-t}}$, hence $U^{\prime}(t_{2}^{p^{l-t}})U(t_{1}^{p^{k_{2}-k_{1}+l-t}})=p\cdot I_{2}$. Moreover, when $1\leq t\leq k_{2}-k_{1}$, $t_{2}^{p^{l-t}}=t_{1}^{p^{k_{2}-k_{1}+l-t}}=0$, hence $U(t_{2}^{p^{l-t}})=U(0)$.
\begin{align*}
    X(p^{l}) & = U^{\prime}(t_{2})U^{\prime}(t_{2}^{p})\cdot\cdot\cdot U^{\prime}(t_{2}^{p^{k_{2}-1}})U(t_{1}^{p^{k_{1}-1}})\cdot\cdot\cdot U(t_{1}^{p})U(t_{1})\in\textup{M}_{2}(A_{p^{l}}/(t_{1}^{p^{k_{1}}}-t_{2}^{p^{k_{2}}}));\\
    Y(p^{l}) & = U^{\prime}(t_{1})U^{\prime}(t_{1}^{p})\cdot\cdot\cdot U^{\prime}(t_{1}^{p^{k_{1}-1}})U(t_{2}^{p^{k_{2}-1}})\cdot\cdot\cdot U(t_{2}^{p})U(t_{2})\in\textup{M}_{2}(A_{p^{l}}/(t_{1}^{p^{k_{1}}}-t_{2}^{p^{k_{2}}})).
\end{align*}
therefore we conclude that $t_{1}^{p^{k_{1}}}-t_{2}^{p^{k_{2}}}$ divides $\overline{f}_{x}$ when $k_{1},k_{2}\geq0$ and $k_{1}+k_{2}=2k+1$ by Theorem \ref{window}, hence $\overline{f}_{x}$ is divisible by the following polynomial,
\begin{equation*}
    \prod\limits_{a=0}^{k}\left((t_{1}-t_{2}^{p^{2a+1}})(t_{2}-t_{1}^{p^{2a+1}})\right)^{p^{k-a}}.
\end{equation*}
We also know that $ \prod\limits_{a=0}^{k}\left((t_{1}-t_{2}^{p^{2a+1}})(t_{2}-t_{1}^{p^{2a+1}})\right)^{p^{k-a}}\equiv (t_{1}t_{2})^{\frac{p^{k+1}-1}{p-1}}$ mod $(t_{1},t_{2})^{p^{k+1}}$, the lemma follows by comparing this formula with (\ref{modpdec}).
\end{proof}
\begin{corollary}
Let $x\in\mathbb{B}$ be an integral nonzero element which has valuation $n\geq1$. Let $\mathcal{Z}(x)_{p}$ be special fiber of $\mathcal{Z}(x)$, then
\begin{equation*}
    \mathcal{Z}(x)_{p}\simeq\textup{Spf}\,\mathbb{F}[[t_{1},t_{2}]]/\left(\prod\limits_{a+b=n\atop a,b\geq0}(t_{1}^{p^{a}}-t_{2}^{p^{b}})\right).
\end{equation*}
Let $\mathcal{D}(x)_{p}$ (resp. $\mathcal{N}_{0}(N)_{p}$) be the base change of $\mathcal{D}(x)$ (resp. $\mathcal{N}_{0}(N)$) to $\mathbb{F}[[t_{1},t_{2}]]$, then
\begin{equation*}
    \mathcal{N}_{0}(N)_{p}\simeq\mathcal{D}(x)_{p}\simeq\textup{Spf}\,\mathbb{F}[[t_{1},t_{2}]]/\left((t_{1}-t_{2}^{p^{n}})\cdot(t_{2}-t_{1}^{p^{n}})\cdot\prod\limits_{a+b=n\atop a,b\geq1}(t_{1}^{p^{a-1}}-t_{2}^{p^{b-1}})^{p-1}\right).
\end{equation*}
\label{spcci}
\end{corollary}
\begin{proof}
The statement for $\mathcal{Z}(x)_{p}$ follows from Theorem \ref{spci}. The statement for $\mathcal{D}(x)_{p}$ follows from the definition of difference divisors.
\end{proof}
\begin{remark}
The same formula has been proved in \cite[Theorem 13.4.6, Theorem 13.4.7]{KM85} by a totally different method.
\end{remark}

\subsection{Local arithmetic intersection numbers}
Now we give the definition of the local arithmetic intersection numbers.
\begin{definition}
For any rank $3$ lattice $L\subset\mathbb{B}$, we choose a $\mathbb{Z}_{p}$-basis $\{x_{1},x_{2},x_{3}\}$ of $L$. Let $\mathcal{O}_{\mathcal{Z}^{\sharp}(x_{i})}$ be the structure sheaf of the special cycle $\mathcal{Z}^{\sharp}(x_{i})$. Let $\mathcal{O}_{\mathcal{N}}$ be the structure sheaf of the formal scheme $\mathcal{N}$. Let $-\otimes^{\mathbb{L}}_{\mathcal{O}_{\mathcal{N}}}-$ be the derived tensor product functor in the derived category of coherent sheaves on $\mathcal{N}$. Define the local arithmetic intersection number of $L$ on $\mathcal{N}$ to be
\begin{equation*}
    \textup{Int}^{\sharp}(L) = \chi(\mathcal{N},\mathcal{O}_{\mathcal{Z}^{\sharp}(x_{1})}\otimes^{\mathbb{L}}_{\mathcal{O}_{\mathcal{N}}}\mathcal{O}_{\mathcal{Z}^{\sharp}(x_{2})}\otimes^{\mathbb{L}}_{\mathcal{O}_{\mathcal{N}}}\mathcal{O}_{\mathcal{Z}^{\sharp}(x_{3})}).
\end{equation*}
\end{definition}
\par
This number is finite and independent of the choice of the basis $\{x_{i}\}_{i=1}^{3}$ of $L$ because of the following result.
\begin{lemma}
Let $x,y\in\mathbb{B}$ be two linearly independent elements, then the tor sheaves $\linebreak\underline{\textup{Tor}}_{i}^{\mathcal{O}_{\mathcal{N}}}(\mathcal{O}_{\mathcal{Z}^{\sharp}(x)},\mathcal{O}_{\mathcal{Z}^{\sharp}(y)})$ vanish for all $i\geq 1$. In particular,
\begin{equation*}
    \mathcal{O}_{\mathcal{Z}^{\sharp}(x)}\otimes^{\mathbb{L}}_{\mathcal{O}_{\mathcal{N}}}\mathcal{O}_{\mathcal{Z}^{\sharp}(y)}=\mathcal{O}_{\mathcal{Z}^{\sharp}(x)}\otimes_{\mathcal{O}_{\mathcal{N}}}\mathcal{O}_{\mathcal{Z}^{\sharp}(y)}.
\end{equation*}
Moreover, the same formula holds if $\mathcal{Z}^{\sharp}(x)$ or $\mathcal{Z}^{\sharp}(y)$ or both are replaced by $\mathcal{D}(x)$ resp. $\mathcal{D}(y)$.\par
Let $L\subset\mathbb{B}$ be an integral quadratic lattice of rank 3 over $\mathbb{Z}_{p}$ with basis $\{x_{1},x_{2},x_{3}\}$, then the derived tensor product $\mathcal{O}_{\mathcal{Z}^{\sharp}(x_{1})}\otimes^{\mathbb{L}}_{\mathcal{O}_{\mathcal{N}}}\mathcal{O}_{\mathcal{Z}^{\sharp}(x_{2})}\otimes^{\mathbb{L}}_{\mathcal{O}_{\mathcal{N}}}\mathcal{O}_{\mathcal{Z}^{\sharp}(x_{3})}$ is independent of the choice of the basis.
\label{coprime}
\end{lemma}
\begin{proof}
This is proved in \cite[Lemma 4.1, Proposition 4.2]{Ter11}.
\end{proof}
\par
Now let's come to the local arithmetic intersection numbers on $\mathcal{N}_{0}(N)$. For a fixed $N$-isogeny $x_{0}$ of $\mathbb{X}$, recall that we have defined the space of quasi-isogenies of $\mathbb{X}$ special to $x_{0}$ (cf. Definition \ref{esp1}) to be those $x\in\mathbb{B}$ such that
\begin{equation*}
   x\circ x_{0}^{\vee}+x_{0}\circ x^{\vee} = 0.
\end{equation*}
Recall that we use $\mathbb{W}$ to denote this space.
\begin{definition}
For any rank $2$ lattice $M\subset\mathbb{W}$, we choose a $\mathbb{Z}_{p}$-basis $\{x_{1},x_{2}\}$ of $M$. Let $\mathcal{O}_{\mathcal{Z}(x_{i})}$ be the structure sheaf of the special cycle $\mathcal{Z}(x_{i})$. Let $\mathcal{O}_{\mathcal{N}_{0}(N)}$ be the structure sheaf of the formal scheme $\mathcal{N}_{0}(N)$. Let $-\otimes^{\mathbb{L}}_{\mathcal{O}_{\mathcal{N}_{0}(N)}}-$ be the derived tensor product functor in the derived category of coherent sheaves on $\mathcal{N}_{0}(N)$. Define the local arithmetic intersection number of $M$ on $\mathcal{N}_{0}(N)$ to be
\begin{equation*}
    \textup{Int}_{\mathcal{N}_{0}(N)}(M) = \chi(\mathcal{N}_{0}(N),\mathcal{O}_{\mathcal{Z}(x_{1})}\otimes^{\mathbb{L}}_{\mathcal{O}_{\mathcal{N}_{0}(N)}}\mathcal{O}_{\mathcal{Z}(x_{2})}).
\end{equation*}
\end{definition}
This number is independent of the choice of the basis $\{x_{1},x_{2}\}$ of $M$ because of Lemma \ref{coprime} and Theorem \ref{keydecomgeo}, we will relate it to the derivative of the local density of the quadratic lattice $M$ with level $N$. The following theorem is the starting point of our calculation.
\begin{theorem}
For any prime number $p$, let $L\subset\mathbb{B}$ be a $\mathbb{Z}_{p}$-lattice of rank 3, then
\begin{equation*}
\textup{Int}^{\sharp}(L)  = \partial\textup{Den}^{+}(L).
\end{equation*}
\label{mainlz22}
\end{theorem}
\begin{proof}
In \cite[$\S$4]{GK93}, the Gross-Keating invariants $(a_{1},a_{2},a_{3})$ of the rank 3 quadratic lattice $L$ is defined, then the local arithmetic intersection number $\textup{Int}^{\sharp}(L)$ is computed explicitly in terms of these invariants (see also \cite[Theorem 1.1]{Rap07}). In \cite[$\S$2.11]{Wed07}, the local density $\textup{Den}^{+}(X,L)$ is also expressed explicitly in terms of the Gross-Keating invariants $(a_{1},a_{2},a_{3})$, hence the derived local density $\partial\textup{Den}^{+}(L)$, the theorem is proved by comparing the expressions of both sides in terms of $(a_{1},a_{2},a_{3})$ (see \cite[$\S$2.16]{Wed07}). The readers can also see \cite{LZ22} for a recent new proof when $p$ is odd.
\end{proof}

\subsection{Difference formula of the local arithmetic intersection numbers}
Fix a $N$-isogeny $x_{0}\in \textup{End}(\mathbb{X})$, recall that $\mathbb{W}= \{x_{0}\}^{\bot}\rightarrow\mathbb{B}$.
\begin{theorem}
For any rank 2 lattice $M\subset\mathbb{W}$, the following identity holds,
\begin{equation*}
    \textup{Int}_{\mathcal{N}_{0}(N)}(M) = 
    \textup{Int}^{\sharp}(M\obot\mathbb{Z}_{p}\cdot x_{0})-\textup{Int}^{\sharp}(M\obot\mathbb{Z}_{p}\cdot p^{-1}x_{0}).
\end{equation*}
\label{geodiff}
\end{theorem}
\begin{proof}
Let $\{x_{1},x_{2}\}$ be a basis of $M$. By Lemma \ref{coprime} and the isomorphism $\mathcal{D}(x_{0})\simeq\mathcal{N}_{0}(N)$, we have the following isomorphism as complexes of coherent sheaves on $\mathcal{N}$,
\begin{align*}
    \mathcal{O}_{\mathcal{N}_{0}(N)}\otimes^{\mathbb{L}}_{\mathcal{O}_{\mathcal{N}}}\mathcal{O}_{\mathcal{Z}^{\sharp}(x_{1})}\otimes^{\mathbb{L}}_{\mathcal{O}_{\mathcal{N}}}\mathcal{O}_{\mathcal{Z}^{\sharp}(x_{2})} & \simeq \mathcal{O}_{\mathcal{Z}(x_{1})}\otimes^{\mathbb{L}}_{\mathcal{O}_{\mathcal{N}}}\mathcal{O}_{\mathcal{Z}^{\sharp}(x_{2})}\\
    & \simeq\mathcal{O}_{\mathcal{Z}(x_{1})}\otimes_{\mathcal{O}_{\mathcal{N}_{0}(N)}}^{\mathbb{L}}\mathcal{O}_{\mathcal{D}(x_{0})}\otimes^{\mathbb{L}}_{\mathcal{O}_{\mathcal{N}}}\mathcal{O}_{\mathcal{Z}^{\sharp}(x_{2})}\\
    & \simeq\mathcal{O}_{\mathcal{Z}(x_{1})}\otimes_{\mathcal{O}_{\mathcal{N}_{0}(N)}}^{\mathbb{L}}\mathcal{O}_{\mathcal{Z}(x_{2})}.
\end{align*}
When $\nu_{p}(N)=0$ or $1$, the difference divisor $\mathcal{D}(x_{0})$ is just $\mathcal{Z}^{\sharp}(x_{0})$, hence $\textup{Int}_{\mathcal{N}_{0}(N)}(M) = \textup{Int}^{\sharp}(M\obot\mathbb{Z}_{p}\cdot x_{0})$ and $\textup{Int}^{\sharp}(M\obot\mathbb{Z}_{p}\cdot p^{-1}x_{0})=0$ since $p^{-1}x_{0}$ is not integral, therefore $\textup{Int}_{\mathcal{N}_{0}(N)}(M) = \textup{Int}^{\sharp}(M\obot\mathbb{Z}_{p}\cdot x_{0})-\textup{Int}^{\sharp}(M\obot\mathbb{Z}_{p}\cdot p^{-1}x_{0})$.
\par
When $\nu_{p}(N)\geq2$, we have the following exact sequence,
\begin{equation*}
    0\longrightarrow \mathcal{O}_{\mathcal{Z}^{\sharp}(p^{-1}x_{0})}\stackrel{\times d_{x_{0}}}\longrightarrow \mathcal{O}_{\mathcal{Z}^{\sharp}(x_{0})}\longrightarrow \mathcal{O}_{\mathcal{D}(x_{0})}\simeq\mathcal{O}_{\mathcal{N}_{0}(N)}\longrightarrow 0.
\end{equation*}
Tensoring the above exact sequence with the complex $\mathcal{O}_{\mathcal{Z}^{\sharp}(x_{1})}\otimes^{\mathbb{L}}_{\mathcal{O}_{\mathcal{N}}}\mathcal{O}_{\mathcal{Z}^{\sharp}(x_{2})}$ in the derived category of coherent sheaves on $\mathcal{N}$, we get an exact triangle
\begin{align*}
    \mathcal{O}_{\mathcal{Z}^{\sharp}(p^{-1}x_{0})}\otimes^{\mathbb{L}}_{\mathcal{O}_{\mathcal{N}}}\mathcal{O}_{\mathcal{Z}^{\sharp}(x_{1})}\otimes^{\mathbb{L}}_{\mathcal{O}_{\mathcal{N}}}\mathcal{O}_{\mathcal{Z}^{\sharp}(x_{2})}&\rightarrow \mathcal{O}_{\mathcal{Z}^{\sharp}(x_{0})}\otimes^{\mathbb{L}}_{\mathcal{O}_{\mathcal{N}}}\mathcal{O}_{\mathcal{Z}^{\sharp}(x_{1})}\otimes^{\mathbb{L}}_{\mathcal{O}_{\mathcal{N}}}\mathcal{O}_{\mathcal{Z}^{\sharp}(x_{2})}\\
    &\rightarrow \mathcal{O}_{\mathcal{N}_{0}(N)}\otimes^{\mathbb{L}}_{\mathcal{O}_{\mathcal{N}}}\mathcal{O}_{\mathcal{Z}^{\sharp}(x_{1})}\otimes^{\mathbb{L}}_{\mathcal{O}_{\mathcal{N}}}\mathcal{O}_{\mathcal{Z}^{\sharp}(x_{2})}\rightarrow,
\end{align*}
hence the following identity,
\begin{align*}
    \chi(\mathcal{O}_{\mathcal{Z}^{\sharp}(x_{0})}\otimes^{\mathbb{L}}_{\mathcal{O}_{\mathcal{N}}}\mathcal{O}_{\mathcal{Z}^{\sharp}(x_{1})}\otimes^{\mathbb{L}}_{\mathcal{O}_{\mathcal{N}}}\mathcal{O}_{\mathcal{Z}^{\sharp}(x_{2})}) &=  \chi(\mathcal{O}_{\mathcal{Z}^{\sharp}(p^{-1}x_{0})}\otimes^{\mathbb{L}}_{\mathcal{O}_{\mathcal{N}}}\mathcal{O}_{\mathcal{Z}^{\sharp}(x_{1})}\otimes^{\mathbb{L}}_{\mathcal{O}_{\mathcal{N}}}\mathcal{O}_{\mathcal{Z}^{\sharp}(x_{2})})\\
    &+\chi(\mathcal{O}_{\mathcal{N}_{0}(N)}\otimes^{\mathbb{L}}_{\mathcal{O}_{\mathcal{N}}}\mathcal{O}_{\mathcal{Z}^{\sharp}(x_{1})}\otimes^{\mathbb{L}}_{\mathcal{O}_{\mathcal{N}}}\mathcal{O}_{\mathcal{Z}^{\sharp}(x_{2})}).
\end{align*}
We already know that $\mathcal{O}_{\mathcal{N}_{0}(N)}\otimes^{\mathbb{L}}_{\mathcal{O}_{\mathcal{N}}}\mathcal{O}_{\mathcal{Z}^{\sharp}(x_{1})}\otimes^{\mathbb{L}}_{\mathcal{O}_{\mathcal{N}}}\mathcal{O}_{\mathcal{Z}^{\sharp}(x_{2})}\simeq \mathcal{O}_{\mathcal{Z}(x_{1})}\otimes_{\mathcal{O}_{\mathcal{N}_{0}(N)}}^{\mathbb{L}}\mathcal{O}_{\mathcal{Z}(x_{2})}$ since $\mathcal{N}_{0}(N)\simeq\mathcal{D}(x_{0})$, hence
\begin{align*}
    \textup{Int}_{\mathcal{N}_{0}(N)}(M)&=\chi(\mathcal{O}_{\mathcal{Z}(x_{1})}\otimes_{\mathcal{O}_{\mathcal{D}(x_{0})}}^{\mathbb{L}}\mathcal{O}_{\mathcal{Z}(x_{2})})=\chi(\mathcal{O}_{\mathcal{N}_{0}(N)}\otimes^{\mathbb{L}}_{\mathcal{O}_{\mathcal{N}}}\mathcal{O}_{\mathcal{Z}^{\sharp}(x_{1})}\otimes^{\mathbb{L}}_{\mathcal{O}_{\mathcal{N}}}\mathcal{O}_{\mathcal{Z}^{\sharp}(x_{2})})\\
    &=\chi(\mathcal{O}_{\mathcal{Z}^{\sharp}(x_{0})}\otimes^{\mathbb{L}}_{\mathcal{O}_{\mathcal{N}}}\mathcal{O}_{\mathcal{Z}^{\sharp}(x_{1})}\otimes^{\mathbb{L}}_{\mathcal{O}_{\mathcal{N}}}\mathcal{O}_{\mathcal{Z}^{\sharp}(x_{2})})-\chi(\mathcal{O}_{\mathcal{Z}^{\sharp}(p^{-1}x_{0})}\otimes^{\mathbb{L}}_{\mathcal{O}_{\mathcal{N}}}\mathcal{O}_{\mathcal{Z}^{\sharp}(x_{1})}\otimes^{\mathbb{L}}_{\mathcal{O}_{\mathcal{N}}}\mathcal{O}_{\mathcal{Z}^{\sharp}(x_{2})})\\
    &=\textup{Int}^{\sharp}(M\obot\mathbb{Z}_{p}\cdot x_{0})-\textup{Int}^{\sharp}(M\obot\mathbb{Z}_{p}\cdot p^{-1}x_{0}).
\end{align*}
\end{proof}

\section{Difference formula at the analytic side}
\label{7}
Let $p$ be a prime number. Let $F$ be a nonarchimedean local field of residue characteristic $p$, with ring of integers $\mathcal{O}_{F}$, residue field $\kappa = \mathbb{F}_{q}$ of size $q$, and uniformizer $\pi$.
\subsection{Primitive decomposition}
Let $N\in F$, recall that we use $(\langle N\rangle,q_{\langle N\rangle})$ to denote the rank 1 quadratic lattice over $\mathcal{O}_{F}$ with a $\mathcal{O}_{F}$-generator $l_{N}$ such that $q_{\langle N\rangle}(l_{N})=N$, then $\langle N\rangle$ is an integral quadratic lattice if and only if $N\in\mathcal{O}_{F}$. Let $n=\nu_{\pi}(N)$, all the rank 1 integral quadratic lattice $L^{\prime}$ containing $\langle N\rangle$ has the following form,
\begin{equation*}
    L^{\prime} = \pi^{-i}\langle N\rangle\simeq \langle \pi^{-2i}N\rangle,\,\,\textup{for $i=0,1,\cdot\cdot\cdot,[\frac{n}{2}]$}.
\end{equation*}
Let $H$ be a self-dual quadratic $\mathcal{O}_{F}$-lattice of finite rank. Since $q_{H}(x)\in\mathcal{O}_{F}$ for every $x\in H$, Lemma \ref{stratification} gives the following decomposition, 
\begin{align*}
    \textup{Rep}_{H,\langle N\rangle}(\mathcal{O}_{F}) & = \bigsqcup_{i=0}^{[n/2]}\textup{PRep}_{H,\langle \pi^{-2i}N\rangle}(\mathcal{O}_{F}).
\end{align*}
Now for every $0\leq i\leq [\frac{n}{2}]$, we pick an arbitrary $\phi\in \textup{PRep}_{H,\langle \pi^{-2i}N\rangle}(\mathcal{O}_{F})$ and consider the following sub-lattice of $H$,
\begin{equation*}
    H(\phi) \coloneqq \{x\in H:(x,\phi(l_{N}))=0\}.
\end{equation*}
\begin{lemma}
The isometric class of the quadratic lattice $H(\phi)$ is independent of the choice of $\phi\in \textup{PRep}_{H,\langle \pi^{-2i}N\rangle}(\mathcal{O}_{F})$.
\label{codim1}
\end{lemma}
\begin{proof} Let $\phi^{\prime}\in \textup{PRep}_{H,\langle \pi^{-2i}N\rangle}(\mathcal{O}_{F})$ be another element. The homomorphims $\phi$ and $\phi^{\prime}$ are totally determined by $x\coloneqq\phi(l_{\pi^{-2i}N})$ and $x^{\prime}\coloneqq\phi^{\prime}(l_{\pi^{-2i}N})$. The fact that $\phi$ and $\phi^{\prime}$ are primitive implies that $x\notin \pi\cdot H$ and $x^{\prime}\notin \pi\cdot H$. Therefore,
\begin{equation*}
    (x,H) =\mathcal{O}_{F},\,\,\,\,\,(x^{\prime},H) =\mathcal{O}_{F}.
\end{equation*}
where we use $(\cdot,\cdot)$ to denote the associated bilinear form on $H$. Since $H$ is self-dual, then by the work of Morin-Strom \cite[Theorem 5.3]{Mor79}, there exists $\varphi\in \textup{O}(H)(\mathcal{O}_{F})$ such that $\varphi(x)=x^{\prime}$. The homomorphism $\varphi$ also induces an isometric between $H(\phi)$ and $H(\phi^{\prime})$ because $H(\phi)=x^{\bot}\cap H$ and $H(\phi^{\prime})=x^{\prime\bot}\cap H$.
\end{proof}
Let $N\in\mathcal{O}_{F}$ be an element of valuation $n$, for every $0\leq i\leq [\frac{n}{2}]$ and $\phi\in \textup{PRep}_{H,\langle \pi^{-2i}N\rangle}(\mathcal{O}_{F})$, we use $H(N,i)$ to denote the quadratic lattice $H(\phi)$. 
\begin{example}
Let $N\in\mathcal{O}_{F}$ has valuation $n$. When $k>4$, we have an orthogonal decomposition,
\begin{equation*}
    H_{k}^{\varepsilon} \simeq H_{4}^{+}\,\obot\,H_{k-4}^{\varepsilon}.
\end{equation*}
Recall that the symbol $H_{k}^{\varepsilon}$ is understood in the following way: when $p$ is odd, $k$ can be any positive integer, and $\varepsilon\in\{\pm1\}$ is arbitrary; when $p=2$, $k$ is even and $\varepsilon=+1$. The lattice $\textup{M}_{2}(\mathcal{O}_{F})$ is equipped with the quadratic form induced by the determinant, it is self-dual and $\chi_{F}(\textup{M}_{2}(\mathcal{O}_{F}))=1$, hence we can view $\textup{M}_{2}(\mathcal{O}_{F})$ as a model lattice for $H_{4}^{+}$.
\par
Let's consider the following element $\phi\in \textup{PRep}_{H_{k}^{\varepsilon},\langle \pi^{-2i}N\rangle}(\mathcal{O}_{F})$
\begin{align*}
    \phi_{i}: \langle \pi^{-2i}N\rangle&\longrightarrow \textup{M}_{2}(\mathcal{O}_{F})\simeq H_{4}^{+}\hookrightarrow H_{k}^{\varepsilon},\\
    l_{\pi^{-2i}N}&\longmapsto \begin{pmatrix} \pi^{-2i}N & 0\\0 & 1  \end{pmatrix}.
\end{align*}
the corresponding element in $\textup{Rep}_{H_{k}^{\varepsilon},\langle N\rangle}(\mathcal{O}_{F})$ sends $l_{N}$ to $\begin{pmatrix} \pi^{-i}N & 0\\0 & \pi^{i} \end{pmatrix}$.
\par
Lemma \ref{codim1} implies that the following quadratic lattices are isometric,
\begin{equation*}
    H_{k}^{\varepsilon}(N,i)=H_{k}^{\varepsilon}(\phi_{i})\simeq \phi_{i}(l_{\pi^{-2i}N})^{\bot}\,\obot\,H_{k-4}^{\varepsilon},
\end{equation*}
where $\phi_{i}(l_{\pi^{-2i}N})^{\bot}$ is the space of elements in $\textup{M}_{2}(\mathcal{O}_{F})$ that are orthogonal to $\phi_{i}(l_{\pi^{-2i}N})$, it can be described explicitly as follows,
\begin{equation*}
    \phi_{i}(l_{\pi^{-2i}N})^{\bot} = \left\{x=\begin{pmatrix}
    -\pi^{-2i}Na & b\\
    c & a
    \end{pmatrix}:\,\, a,b,c\in\mathcal{O}_{F}\right\}.
\end{equation*}
It is exactly the lattice $\delta_{F}(\pi^{-2i}N)$ we defined in Example \ref{lambda}, therefore $H_{k}^{\varepsilon}(N,i)\simeq \delta_{F}(\pi^{-2i}N)\,\obot\, H_{k-4}^{\varepsilon}$.
\label{explictchoice}
\end{example}

\subsection{Difference formula of local densities}
\begin{theorem}
Let $H$ be a self-dual quadratic $\mathcal{O}_{F}$-lattice of finite rank $k$. Let $M$ be an integral quadratic $\mathcal{O}_{F}$-lattice of finite rank $r$. Let $N\in\mathcal{O}_{F}$ be an element of valuation $n$, then 
\begin{equation*}
    \textup{Den}(H,M\obot\langle N\rangle) = \sum\limits_{i=0}^{[n/2]}q^{(2-k+r)i}\cdot\textup{Pden}(H, \langle \pi^{-2i}N\rangle)\cdot\textup{Den}(H(N,i), M).
\end{equation*}
\label{anadecom}
\end{theorem}
The proof of this theorem is based on the following lemmas.
\begin{lemma}
Let $H$ be a self-dual quadratic $\mathcal{O}_{F}$-lattice. Let $N\in\mathcal{O}_{F}$ be an element of valuation $n$, then there is a bijective map $D$ between the following sets when the positive integer $d$ is large enough, 
\begin{equation*}
    D: \textup{Rep}_{H,\langle N\rangle}(\mathcal{O}_{F}/\pi^{d}) \stackrel{\sim}\longrightarrow \bigsqcup_{i=0}^{[n/2]}\bigsqcup_{\overline{x}\in\mathcal{O}_{F}/\pi^{i}}\,\textup{PRep}_{H,\langle\pi^{-2i}N+\pi^{d-2i}x\rangle}(\mathcal{O}_{F}/\pi^{d-i}).
\end{equation*}
\label{decomcodim1}
\end{lemma}
\begin{proof}
Let $l_{N}$ be a generator of the rank 1 $\mathcal{O}_{F}$-module $\langle N\rangle$ such that $q_{\langle N\rangle}(l_{N})=N$. Any $f\in \textup{Rep}_{H, \langle N\rangle}(\mathcal{O}_{F}/\pi^{d})$ is determined by $f(\overline{l_{N}})\in H/\pi^{d}H$. There is a natural filtration on $H/\pi^{d}H$ given as follows,
\begin{equation*}
    0\subset \pi^{d-1}H/\pi^{d}H\subset \pi^{d-2}H/\pi^{d}H\subset\cdot\cdot\cdot\subset
    \pi^{2}H/\pi^{d}H\subset \pi H/\pi^{d}H\subset H/\pi^{d}H.
\end{equation*}
\par
Let $i$ be the minimal integer such that $f(\overline{l_{N}})\in \pi^{i}H/\pi^{d}H$, then $0\leq i\leq [\frac{n}{2}]$ since $\nu_{\pi}(N)=n$. Then there exists $l\in H$ such that $f(\overline{l_{N}})=\overline{\pi^{i}l
}\in \pi^{i}H/\pi^{d}H$, the image of $l$ in $H/\pi^{d-i}H$ is uniquely determined by $f$. Let $q$ be the quadratic form on $H$, then
\begin{equation*}
    N\,\,\textup{mod}\,\,\pi^{d} = \overline{q_{\langle N\rangle}}(\overline{l_{N}}) = \overline{q}(f(\overline{l_{N}})) = \pi^{2i}\overline{q}(\overline{l})=\overline{\pi^{2i}q(l)},
\end{equation*}
hence $\overline{\pi^{2i}q(l)} \equiv \pi^{-2i}N\,\,\textup{mod $\pi^{d-2i}$}$. Therefore there exists $x\in\mathcal{O}_{F}$ such that $q(l)=\pi^{-2i}N+\pi^{d-2i}x$. Next we show that $\overline{x}\in\mathcal{O}_{F}/\pi^{i}$ is independent of the choice of $l\in H$ when $d$ is large enough. Suppose $l^{\prime}$ is another element of $H$ such that $f(\overline{l_{N}})=\overline{\pi^{i}l^{\prime}
}$, then there exists $\delta\in H$ such that $l^{\prime}-l=\pi^{d-i}\delta$. Therefore when $d$ is large enough,
\begin{equation*}
    q(l^{\prime}) = q(l+\pi^{d-i}\delta) = q(l) + \pi^{d-i}(l,\delta)+\pi^{2d-2i}q(\delta) \equiv q(l)\,\,\textup{mod $\pi^{d-i}$}.
\end{equation*}
Suppose $q(l^{\prime})=\pi^{-2i}N+\pi^{d-2i}x^{\prime}$ for some $x^{\prime}\in\mathcal{O}_{F}$, the above congruence formula between $q(l)$ and $q(l^{\prime})$ implies that $x^{\prime}\equiv x\,\,\textup{mod $\pi^{i}$}$. The above construction gives the following element $D(f)$ in $\linebreak\textup{PRep}_{H,\langle\pi^{-2i}N+\pi^{d-2i}x\rangle}(\mathcal{O}_{F}/\pi^{d-i})$: the homomrphism $D(f)$ sends the generator $\overline{l_{\pi^{-2i}N+\pi^{d-2i}x}}$ of $\linebreak\langle\pi^{-2i}N+\pi^{d-2i}x\rangle/\pi^{d-i}\langle\pi^{-2i}N+\pi^{d-2i}x\rangle$
to $\overline{l}\in H/\pi^{d-i} H$.
\par
Now for any element $\varphi\in \textup{PRep}_{H,\langle\pi^{-2i}N+\pi^{d-2i}x\rangle}(\mathcal{O}_{F}/\pi^{d-i})$, we consider the following morphism,
\begin{align*}
    \tilde{\varphi}:\langle N\rangle/\pi^{d}\langle N\rangle & \longrightarrow H/\pi^{d} H,\\
    \overline{l_{N}}&\longmapsto \pi^{i}\varphi(\overline{l_{\pi^{-2i}N+\pi^{d-2i}x}}).
\end{align*}
then $\tilde{\varphi}\in \textup{Rep}_{H,\langle N\rangle}(\mathcal{O}_{F}/\pi^{d})$ because $\overline{q}(\tilde{\varphi}(\overline{l_{N}}))=\overline{\pi^{2i}(\pi^{-2i}N+\pi^{d-2i}x)}=N\,\,\textup{mod $\pi^{d}$}$. This construction gives the inverse map of $D$.
\end{proof}
Let $M$ be an integral quadratic $\mathcal{O}_{F}$-lattice of finite rank. Let $N\in\mathcal{O}_{F}$ be an element of valuation $n$. Let $M^{\sharp}=M\obot\langle N\rangle$ be a quadratic $\mathcal{O}_{F}$-lattice of 1 rank higher than $M$, there is a natural restriction map as follows for any positive integer $d$ and any self-dual quadratic $\mathcal{O}_{F}$-lattice $H$,
\begin{equation*}
    r: \textup{Rep}_{H,M^{\sharp}}(\mathcal{O}_{F}/\pi^{d})\rightarrow\textup{Rep}_{H,\langle N\rangle}(\mathcal{O}_{F}/\pi^{d}),
\end{equation*}
given by composing any element in the set $\textup{Rep}_{H,M^{\sharp}}(\mathcal{O}_{F}/\pi^{d})$ with the natural inclusion $\langle N\rangle/\pi^{d}\langle N\rangle\hookrightarrow M^{\sharp}/\pi^{d}M^{\sharp}$. The next lemma describes the fiber of the map $D\circ r$.
\begin{lemma}
Let $H$ be a self-dual quadratic $\mathcal{O}_{F}$-lattice. Let $M$ be an integral quadratic $\mathcal{O}_{F}$-lattice of finite rank $r$. Let $N\in\mathcal{O}_{F}$ be an element of valuation $n$. Let $M^{\sharp}=M\obot\langle N\rangle$ be a quadratic $\mathcal{O}_{F}$-lattice of rank $r+1$. Let $0\leq i\leq[\frac{n}{2}]$ be an integer. Let $\varphi\in \textup{PRep}_{H,\langle\pi^{-2i}N+\pi^{d-2i}x\rangle}(\mathcal{O}_{F}/\pi^{d-i})$, then for $d$ large enough,
\begin{equation*}
    \#(D\circ r)^{-1}(\varphi) = q^{ir}\cdot \#\textup{Rep}_{H(N,i),M}(\mathcal{O}_{F}/\pi^{d}).
\end{equation*}
\label{fiberr}
\end{lemma}
\begin{proof}
Let $f$ be an element in $\textup{Rep}_{H,M^{\sharp}}(\mathcal{O}_{F}/\pi^{d})$ such that $D\circ r(f)=\varphi$, by the proof of Lemma \ref{decomcodim1}, there exists $l_{N}^{\prime}\in H\backslash\pi H$ such that $f(\overline{l_{N}})=\overline{\pi^{i}l^{\prime}_{N}}$, and $q(l_{N}^{\prime})=\pi^{-2i}N$ when $d$ is large enough. 
\par
Let $\{e_{i}\}_{i=1}^{r}$ be an $\mathcal{O}_{F}$-basis of $M$, then $f$ is determined by $\{x_{i}\coloneqq f(\overline{e_{i}})\in H/\pi^{d}H\}_{i=1}^{r}$. Therefore $(D\circ r)^{-1}(\varphi)$ can be described by the following set,
\begin{align}
    (D\circ r)^{-1}(\varphi)=\big\{(x_{1},\cdot\cdot\cdot,x_{r})\in (H/\pi^{d}H)^{r}:(x_{i},\overline{\pi^{i}l^{\prime}_{N}})=0&,\,\,(x_{i},x_{j})=(\overline{e_{i}},\overline{e_{j}})\,\,\text{for $i\neq j$,}\label{des}\\ &\text{and\,\,\,}\overline{q}(x_{i})=\overline{q_{M}}(\overline{e_{i}})\,\,\text{for every $i$.}\big\}\notag.    
\end{align}
\par
Let $L$ be the rank 1 sub-lattice of $H$ generated by $l_{N}^{\prime}$. We have the following exact sequence,
\begin{equation*}
    0\longrightarrow L\oplus H(N,i)\stackrel{\theta}\longrightarrow H\longrightarrow Q\coloneqq H/L\oplus H(N,i)\longrightarrow 0.
\end{equation*}
where $\theta$ is the natural inclusion map. After tensoring the above exact sequence with $\mathcal{O}_{F}/\pi^{d}$, we get the following exact sequence by the flatness of $H$ over $\mathcal{O}_{F}$,
\begin{equation}
    0\longrightarrow \textup{Tor}_{\mathcal{O}_{F}}^{1}(Q,\mathcal{O}_{F}/\pi^{d})\longrightarrow L/\pi^{d}L\,\oplus\, H(N,i)/\pi^{d}H(N,i)\stackrel{\overline{\theta}}\longrightarrow H/\pi^{d}H\longrightarrow Q/\pi^{d}Q\longrightarrow 0.
    \label{exactsequence}
\end{equation}
Claim: Let $K = \left\{x\in H/\pi^{d}H:(x,\overline{\pi^{i}l^{\prime}_{N}})=0 \right\}$. When $d$ is large enough, for every $\overline{x}\in K$, there exists $x^{\prime}\in L$ and $x^{\prime\prime}\in H(N,i)$ such that the image of $\overline{x^{\prime}}+\overline{x^{\prime\prime}}\in L/\pi^{d}L\,\oplus\, H(N,i)/\pi^{d}H(N,i)$ under $\overline{\theta}$ in $H/\pi^{d}H$ is $\overline{x}$.
\par\textit{Proof of the claim}: We have the following decomposition in the quadratic space $V=H\otimes_{\mathcal{O}_{F}}F$,
\begin{equation*}
    x = x^{\prime} + x^{\prime\prime},
\end{equation*}
where $x^{\prime}\in L_{F}\coloneqq L\otimes_{\mathcal{O}_{F}}F$ and $x^{\prime\prime}\in (L_{F})^{\bot}\subset V$.
\par
The fact that $\overline{x}\in K$ implies that $(x^{\prime},l_{N}) = (x,l_{N})\in (\pi^{d})$, therefore $x^{\prime}\in (\pi^{d-n})\cdot l_{N}\in L_{F}$. It turns out that $x^{\prime}\in L\subset H$ when $d$ is large enough, hence $x^{\prime}=x-x^{\prime}\in H\cap\{l_{N}\}^{\bot}=H(N,i)$.
\\\par
We get the following description of the inverse image of the set $(D\circ r)^{-1}(\varphi)$ under $\Theta\coloneqq\overline{\theta}\times\cdot\cdot\cdot\times\overline{\theta}$ by (\ref{des})
\begin{equation}
    \Theta^{-1}((D\circ r)^{-1}(\varphi)) =(\pi^{d+i-n}L/\pi^{d}L)^{r}\times\textup{Rep}_{H(N,i),M}(\mathcal{O}_{F}/\pi^{d}).
    \label{keydesc}
\end{equation}
and the claim implies that the map $\Theta^{-1}((D\circ r)^{-1}(\varphi))\stackrel{\Theta}\longrightarrow (D\circ r)^{-1}(\varphi)$ is surjective.
\par
Now we compute $\#\textup{ker}(\Theta)$, which equlas $(\#\textup{ker}(\theta))^{r}$ by definition. By the exact sequence (\ref{exactsequence}), $\#\textup{ker}(\overline{\theta})=\#\textup{Tor}_{\mathcal{O}_{F}}^{1}(Q,\mathcal{O}_{F}/\pi^{d})=\# Q/\pi^{d}Q$. Therefore when $d$ is large enough, $Q/\pi^{d}Q=Q$. Since $l_{N}^{\prime}\notin \pi H$, there exists $y\in H$ such that $(l_{N}^{\prime},y)=1$. The existence of $y$ implies the following exact sequence,
\begin{align*}
    0\longrightarrow H(N,i)\stackrel{\theta}\longrightarrow H&\longrightarrow L^{\vee}\longrightarrow 0,\\
    x&\longmapsto l(x):v\in L\mapsto (x,v).
\end{align*}
Therefore $H\simeq L^{\vee}\oplus H(N,i)$ as $\mathcal{O}_{F}$-modules, and $Q\simeq L^{\vee}/L\simeq \pi^{2i-n}L/L$. Then by (\ref{keydesc})
\begin{equation*}
    \#(D\circ r)^{-1}(\varphi) = \frac{q^{r(n-i)}}{q^{r(n-2i)}}\cdot \#\textup{Rep}_{H(N,i),M}(\mathcal{O}_{F}/\pi^{d})= q^{ir}\cdot \#\textup{Rep}_{H(N,i),M}(\mathcal{O}_{F}/\pi^{d}).
\end{equation*}
\end{proof}
\textit{Proof of Theorem \ref{anadecom}}: By Lemma \ref{decomcodim1} and Lemma \ref{fiberr}, we only need to know the size of the set $\textup{PRep}_{H,\langle\pi^{-2i}N+\pi^{d-2i}x\rangle}(\mathcal{O}_{F}/\pi^{d-i})$ when $x\in\mathcal{O}_{F}$. We first show that when $d$ is large enough, the following identity holds for any $x\in\mathcal{O}_{F}$,
\begin{equation*}
    \# \textup{PRep}_{H, \langle\pi^{-2i}N+\pi^{d-2i}x\rangle}(\mathcal{O}_{F}/\pi^{d-i}) = \# \textup{PRep}_{H, \langle \pi^{-2i}N\rangle}(\mathcal{O}_{F}/\pi^{d-i}).
\end{equation*}
Because when $d$ is large enough, for any $x\in\mathcal{O}_{F}$, we could find $c_{x}\in\mathcal{O}_{F}^{\times}$ such that $c_{x}^{-2}= 1+\pi^{d}N^{-1}x$ mod $\pi^{d-i}$, then for any element $l\in \textup{PRep}_{H, \langle\pi^{-2i}N+\pi^{d-2i}x\rangle}(\mathcal{O}_{F}/\pi^{d-i})$, $c_{x}\cdot l\in \textup{PRep}_{H, \langle \pi^{-2i}N\rangle}(\mathcal{O}_{F}/\pi^{d-i})$. Let $M^{\sharp}=M\obot\langle N\rangle$, we have
\begin{align*}
    \textup{Den}(H,M\obot\langle N\rangle) &= \lim\limits_{d\rightarrow\infty}\frac{\#\textup{Rep}_{H,M^{\sharp}}(\mathcal{O}_{F}/\pi^{d})}{q^{d(k(r+1)-(r+1)(r+2)/2)}}\\
    &=\lim\limits_{d\rightarrow\infty}\sum\limits_{i=0}^{[n/2]}q^{i}\cdot\frac{\#\textup{PRep}_{H, \langle \pi^{-2i}N\rangle}(\mathcal{O}_{F}/\pi^{d-i})}{q^{(d-i)(k-1)}}\cdot\frac{q^{ir}}{q^{i(k-1)}}\cdot\frac{\#\textup{Rep}_{H(N,i),M}(\mathcal{O}_{F}/\pi^{d})}{q^{d((k-1)r-r(r+1)/2)}}\\
    &=\sum\limits_{i=0}^{[n/2]}q^{(2-k+r)i}\cdot\textup{Pden}(H, \langle \pi^{-2i}N\rangle)\cdot\textup{Den}(H(N,i), M).
\end{align*}$\hfill\qed$\\
\begin{remark}
When $p$ odd, it has been calculated explicitly that for any $N\in\mathcal{O}_{F}$ (cf. \cite[(3.3.2.1)]{LZ22}), 
\begin{equation}
    \textup{Pden}(H_{k}^{\varepsilon}, \langle N\rangle) = \begin{cases}
    1-q^{1-k}, & \textup{when $k$ is odd and $\pi\,\vert\, N$;}\\
    1+\varepsilon\chi_{F}(N)q^{(1-k)/2}, &\textup{when $k$ is odd and $\pi\nmid N$;}\\
    (1-\varepsilon q^{-k/2})(1+\varepsilon q^{1-k/2}), & \textup{when $k$ is even and $\pi\,\vert\, N$;}\\
    1-\varepsilon q^{-k/2}, & \textup{when $k$ is even and $\pi\nmid N$.}
    \end{cases}
    \label{rank1}
\end{equation}
When $p=2$, the same formula makes sense and holds true only in the case that $k$ is even and $\varepsilon=+1$.
\end{remark}
\begin{definition}
Let $M$ be a quadratic lattice of rank $r$ over $\mathcal{O}_{F}$. Let $N\in\mathcal{O}_{F}$. When $p$ is odd, we define the local density of $M$ with level $N$ to be a polynomial $\textup{Den}^{\varepsilon}_{\Delta(N)}(X,M)$ satisfying
\begin{equation*}
    \textup{Den}^{\varepsilon}_{\Delta(N)}(X,M)\,\big\vert_{X=q^{-m}} = \left\{\begin{aligned}
    &\frac{\textup{Den}(\delta_{F}(N)\obot H_{2m+r-2}^{\varepsilon},M)}{\textup{Nor}^{\varepsilon}(q^{-m},r-1)} , & \textup{When $\pi\,\vert\, N$};\\
    &\frac{\textup{Den}(\delta_{F}(N)\obot H_{2m+r-2}^{\varepsilon},M)}{\textup{Nor}^{\varepsilon\chi_{F}(N)}(q^{-m},r)}, & \textup{When $\pi\nmid N$}.
    \end{aligned}\right.
\end{equation*}
for $m\geq0$. When $p=2$, $r$ is even and $\varepsilon=+1$, the local density of $M$ with level $N$ is the polynomial $\textup{Den}^{+}_{\Delta(N)}(X,M)$ with the same evaluation formula at $X=q^{-m}$ ($m\geq0$) as above.\par
We define the derived local density of $M$ with level $N$ to be
\begin{equation*}
    \partial\textup{Den}^{\varepsilon}_{\Delta(N)}(M) = -\frac{\textup{d}}{\textup{d}X}\bigg|_{X=1}\textup{Den}^{\varepsilon}_{\Delta(N)}(X,M).
\end{equation*}
\label{twisted}
\end{definition}
\begin{theorem}
Let $M$ be a quadratic $\mathcal{O}_{F}$-lattice of finite rank $r$. Let $N\in\mathcal{O}_{F}$, then when $p$ is odd,
\begin{equation*}
    \textup{Den}^{\varepsilon}_{\Delta(N)}(X,M)=\textup{Den}^{\varepsilon}(X, M\obot\langle N\rangle)-X^{2}\cdot\textup{Den}^{\varepsilon}(X, M\obot\langle \pi^{-2}N\rangle).  
\end{equation*}
When $p=2$ and $r$ is even, the same formula for $\varepsilon=+1$ holds.\par
Moreover, if $M\obot\langle N\rangle$ can't be isometrically embedded into the quadratic space $H_{r+2}^{\varepsilon}\otimes_{\mathcal{O}_{F}}F$, then
\begin{equation*}
    \partial\textup{Den}^{\varepsilon}_{\Delta(N)}(M)= \partial\textup{Den}^{\varepsilon}(M\obot\langle N\rangle)-\partial\textup{Den}^{\varepsilon}(M\obot\langle \pi^{-2}N\rangle).
\end{equation*}
When $p=2$ and $r$ is even, the same formula for $\varepsilon=+1$ holds.
\label{anadiff}
\end{theorem}
\begin{proof}
Recall the definition of the polynomial $\textup{Nor}^{\varepsilon}(X,n)$ in Definition \ref{normal1}, we can verify immediately by formula (\ref{rank1}) that, for any $x\in\mathcal{O}_{F}$, 
\begin{align*}
    \textup{Nor}^{\varepsilon}(q^{-m},r+1)&=\textup{Pden}(H_{2m+r+2}^{\varepsilon}, \langle x\rangle)\cdot\textup{Nor}^{\varepsilon\chi_{F}(x)}(q^{-m},r),\,\,\, \text{when $\pi\nmid x$;}\\
    \textup{Nor}^{\varepsilon}(q^{-m},r+1)&=\textup{Pden}(H_{2m+r+2}^{\varepsilon}, \langle x\rangle)\cdot\textup{Nor}^{\varepsilon}(q^{-m},r-1),\,\,\, \text{when $\pi\,\vert\, x$.}
\end{align*}
Note that when $p=2$, the above formulas only make sense for $r$ is even and $\varepsilon=+1$. Let $n=\nu_{\pi}(N)$, Theorem \ref{anadecom} and Example \ref{explictchoice} imply the following decomposition,
\begin{align*}
    \textup{Den}(H_{2m+r+2}^{\varepsilon}, M\obot\langle N\rangle) &= \sum\limits_{i=0}^{[n/2]}q^{-2mi}\cdot\textup{Pden}(H_{2m+r+2}^{\varepsilon},\langle \pi^{-2i}N\rangle)\cdot \textup{Den}(H_{2m+r+2}^{\varepsilon}(N,i), M)\\
    &=\sum\limits_{i=0}^{[n/2]}q^{-2mi}\cdot\textup{Pden}(H_{2m+r+2}^{\varepsilon},\langle \pi^{-2i}N\rangle)\cdot \textup{Den}(\delta_{F}(\pi^{-2i}N)\obot H_{2m+r-2}^{\varepsilon}, M).
\end{align*}
By Definition \ref{twisted}, when $p$ is odd, the following formula holds 
\begin{equation}
    \textup{Den}^{\varepsilon}(X,M\obot\langle N\rangle) = \sum\limits_{i=0}^{[n/2]}X^{2i}\cdot\textup{Den}^{\varepsilon}_{\Delta(\pi^{-2i}N)}(X,M).
    \label{decompp}
\end{equation}
when $p=2$ and $r$ is even, formula (\ref{decompp}) for $\varepsilon=+1$ holds.
\par
When $n=0$ or $1$, $\textup{Den}^{\varepsilon}(X,M\obot\langle N\rangle)=\textup{Den}^{\varepsilon}_{\Delta(N)}(X,M)$ and $\textup{Den}^{\varepsilon}(X,M\obot\langle \pi^{-2}N\rangle)=0$ since $\pi^{-2}N$ is not in $\mathcal{O}_{F}$, therefore $\textup{Den}_{\Delta(N)}^{\varepsilon}(X,M)=\textup{Den}^{\varepsilon}(X, M\obot\langle N\rangle)-X^{2}\cdot\textup{Den}^{\varepsilon}(X, M\obot\langle \pi^{-2}N\rangle)$. When $n\geq2$, $\textup{Den}_{\Delta(N)}^{\varepsilon}(X,M)=\textup{Den}^{\varepsilon}(X, M\obot\langle N\rangle)-X^{2}\cdot\textup{Den}^{\varepsilon}(X, M\obot\langle \pi^{-2}N\rangle)$ follows from the formula (\ref{decompp}).
\par
The fact that the lattice $M\obot\langle N\rangle$ can't be isometrically embedded into the quadratic space $H_{r+2}^{\varepsilon}\otimes_{\mathcal{O}_{F}}F$ implies that $\textup{Den}^{\varepsilon}(1, M\obot\langle N\rangle)=\textup{Den}^{\varepsilon}(1, M\obot\langle \pi^{-2}N\rangle)=0$, the second formula in the theorem follows from the first one and the definitions of the symbols $\partial\textup{Den}^{\varepsilon}_{\Delta(N)}$ and $\partial\textup{Den}^{\varepsilon}$.
\end{proof}
Now we apply Theorem \ref{anadiff} to the case that we are interested, i.e., $F=\mathbb{Q}_{p}$, $\epsilon=+1$ and $r=2$, Let $N\in\mathbb{Z}_{p}$, we get the following difference formula of local density functions,
\begin{equation}
    \textup{Den}^{+}_{\Delta(N)}(X,M)=\textup{Den}^{+}(X, M\obot\langle N\rangle)-X^{2}\cdot\textup{Den}^{+}(X, M\obot\langle p^{-2}N\rangle).
    \label{decder}
\end{equation}
Note that the lattice $M\obot\langle N\rangle$ is a sub-lattice of $\mathbb{B}\simeq\textup{End}^{0}(\mathbb{X})$ which is the unique division quaternion algebra over $\mathbb{Q}_{p}$, hence the lattice $M\obot\langle N\rangle$ can't be isometrically embedded into the quadratic space $H_{4}^{+}\otimes\mathbb{Q}_{p}\simeq\textup{M}_{2}(\mathbb{Q}_{p})$, therefore Theorem \ref{anadiff} implies the following difference formula of the derivatives of local densities,
\begin{equation}
    \partial\textup{Den}_{\Delta(N)}^{+}(M)=\partial\textup{Den}^{+}(M\obot\langle N\rangle)-\partial\textup{Den}^{+}(M\obot\langle p^{-2}N\rangle).
    \label{diffana}
\end{equation}
\par
\subsection{Induction formula}
In this section, we assume $p$ is odd. Let $N_{k}=\pi^{2k}N_{0}$ where $N_{0}\in\mathcal{O}_{F}$ has valuation 0 or 1. Next we show that when $k$ is sufficiently large, $\textup{Den}^{\varepsilon}_{\Delta(N_{k})}(X,M)$ is independent of $k$, i.e., it will become a constant polynomial.
\begin{proposition}
Let $M$ be a quadratic lattice of rank $r\geq2$ over $\mathcal{O}_{F}$ which is $\textup{GL}_{2}(\mathcal{O}_{F})$-equivalent to $\textup{diag}\{\varepsilon_{1}\pi^{a_{1}},\varepsilon_{2}\pi^{a_{2}},\cdot\cdot\cdot,\varepsilon_{r}\pi^{r}\}$, where $\varepsilon_{i}\in\mathcal{O}_{F}^{\times}$, $a_{i}\in\mathbb{N}$ and $a_{1}\leq a_{2}\leq\cdot\cdot\cdot\leq a_{r}$. Let $N_{k}=\pi^{2k}N_{0}$ where $N_{0}\in\mathcal{O}_{F}$ has valuation 0 or 1. Then when $k>\frac{a_{r}-\nu_{\pi}(N_{0})}{2}$,
\begin{equation*}
    \textup{Den}^{\varepsilon}_{\Delta(N_{k})}(X,M) = (1-\varepsilon\chi_{F}(M)X)^{-1}(1-X^{2})\cdot\textup{Den}^{\flat\varepsilon}(X,M).
\end{equation*}
Moreover, if $M\obot\langle N_{0}\rangle$ can't be isometrically embedded into the quadratic space $H_{r+2}^{\varepsilon}\otimes_{\mathcal{O}_{F}}F$, then
\begin{equation*}
   \partial\textup{Den}^{\varepsilon}_{\Delta(N_{k})}(M)=\begin{cases}
    -\varepsilon\chi_{F}(M)\cdot\textup{Den}^{\flat\varepsilon}(1,M), & \textup{if $\chi_{F}(M)\neq 0$;}\\
    2\cdot \textup{Den}^{\flat\varepsilon}(1,M), &\textup{if $\chi_{F}(M)=0$.}\\
    \end{cases}
\end{equation*}
\end{proposition}
\begin{proof}
Proposition \ref{anadiff} implies that for any $k\geq 1$,
\begin{equation*}
    \textup{Den}^{\varepsilon}(X,M\obot\langle N_{k}\rangle)-X^{2}\cdot\textup{Den}^{\varepsilon}(X,M\obot\langle N_{k-1}\rangle)=\textup{Den}^{\varepsilon}_{\Delta(N_{k})}(X,M).
\end{equation*}
Lemma \ref{induction} implies the following induction formula for $k>\frac{a_{r}-\nu_{\pi}(N_{0})}{2}$,
\begin{equation*}
    \textup{Den}^{\varepsilon}(X,M\obot\langle N_{k}\rangle)-X^{2}\cdot\textup{Den}^{\varepsilon}(X,M\obot\langle N_{k-1}\rangle) = (1-\varepsilon\chi_{F}(M)X)^{-1}(1-X^{2})\cdot\textup{Den}^{\flat\varepsilon}(X,M).
\end{equation*}
therefore we obtain the first formula, the second formula follows from the first one by applying $-\frac{\textup{d}}{\textup{d}X}\big\vert_{X=1}$.
\end{proof}
\subsection{Examples}In this section, we assume $p$ is odd. We will compare our formulas with known formulas given in \cite{Wed07} and \cite{SSY22}. 
\begin{example}
When $\nu_{p}(N)=1$, the above formula implies that $\textup{Den}^{+}(X,M\obot\langle N\rangle) =$ $\linebreak  \textup{Den}^{+}_{\Delta(N)}(X,M)$, suppose $M\hookrightarrow\mathbb{W}$ is $\textup{GL}_{2}(\mathbb{Z}_{p})$-equivalent to the matrix $\textup{diag}\{\varepsilon_{1}p^{a},\varepsilon_{2}p^{b}\}$ where $\varepsilon_{1},\varepsilon_{2}\in\mathbb{Z}_{p}^{\times}$ and $1\leq a\leq b$. \\
$\bullet$ When $a$ is even, by \cite[$\S$2.11]{Wed07},
\begin{align*}
    \textup{Den}^{+}(X,M\obot\langle N \rangle)=\sum\limits_{i=0}^{a/2}p^{i}X^{2i}+\sum\limits_{i=1}^{a/2}p^{i}X^{2i-1}-\sum\limits_{i=0}^{a/2}p^{i}X^{a+b+1-2i}-\sum\limits_{i=1}^{a/2}p^{i}X^{a+b+2-2i}.
\end{align*}
we compare this formula with the formula appeared in \cite[Proposition 3.6]{SSY22}, it computes $\textup{Den}^{+}_{\Delta(N)}(X,M)$ in the following way,
\begin{equation*}
    \textup{Den}^{+}_{\Delta(N)}(X,M)=\frac{\alpha_{p}(X,M,\delta_{p}(N))}{1-p^{-1}X},
\end{equation*}
where $\alpha_{p}(X,M,\delta_{p}(N))=X^{-2}(p\cdot\alpha(X)+(X-p)(R_{1}(X)+1+p^{-1}X))$, and
\begin{equation*}
    \alpha(X)  = (1-p^{-2}X^{2})\left(\sum\limits_{i=0}^{a/2}p^{i}(X^{2i}-X^{a+b+2-2i})\right),\,\,\,
    R_{1}(X) = -1 -p^{a/2}X^{b+2}+(1+p^{-1}X^{2})\sum\limits_{i=0}^{a/2}p^{i}X^{2i},
\end{equation*}
then 
\begin{align*}
    &\frac{\alpha_{p}(X,M,\delta_{p}(N))}{1-p^{-1}X} \\& = pX^{-2}\left((1+p^{-1}X)\left(\sum\limits_{i=0}^{a/2}p^{i}(X^{2i}-X^{a+b+2-2i})\right)-p^{-1}X +p^{a/2}X^{b+2}-(1-p^{-1}X^{2})\sum\limits_{i=0}^{a/2}p^{i}X^{2i}\right)\\
    & = (pX^{-2}+X^{-1})\left(\sum\limits_{i=0}^{a/2}p^{i}(X^{2i}-X^{a+b+2-2i})\right)-X^{-1}+p^{a/2+1}X^{b} +(1-pX^{-2})\sum\limits_{i=0}^{a/2}p^{i}X^{2i}\\
    & =\sum\limits_{i=0}^{a/2}p^{i+1}X^{2i-2}+\sum\limits_{i=0}^{a/2}p^{i}X^{2i-1}-\sum\limits_{i=0}^{a/2}p^{i+1}X^{a+b-2i}-\sum\limits_{i=0}^{a/2}p^{i}X^{a+b+1-2i}+\sum\limits_{i=0}^{a/2}p^{i}X^{2i}\\&-\sum\limits_{i=0}^{a/2}p^{i+1}X^{2i-2}+p^{a/2+1}X^{b}-X^{-1}=\textup{Den}^{+}(X,M\obot\langle N\rangle).
\end{align*}
$\bullet$ When $a$ is odd, again by \cite[$\S$2.11]{Wed07} (note that $b$ must be even since $M$ is anisotropic),
\begin{align*}
    \textup{Den}^{+}(X,M\obot\langle N\rangle)=&\sum\limits_{i=0}^{(a-1)/2}p^{i}X^{2i}+\sum\limits_{i=1}^{(a-1)/2}p^{i}X^{2i-1}-\sum\limits_{i=0}^{(a-1)/2}p^{i}X^{a+b+1-2i}\\&-\sum\limits_{i=1}^{(a-1)/2}p^{i}X^{a+b+2-2i}-p^{(a+1)/2}(X^{b+1}-X^{a}).
\end{align*}
we also compare this with the formula appeared in \cite[Proposition 3.6]{SSY22}, it computes $\textup{Den}^{+}_{\Delta(N)}(X,M)$ in the following way,
\begin{equation*}
    \textup{Den}_{\Delta(N)}^{+}(X,M)=\frac{\alpha_{p}(X,M,\delta_{p}(N))}{1-p^{-1}X}.
\end{equation*}
where $\alpha_{p}(X,M,\delta_{p}(N))=X^{-2}(p\cdot\alpha(X)+(X-p)(R_{1}(X)+1+p^{-1}X))$, and
\begin{align*}
    \alpha(X) & = (1-p^{-2}X^{2})\left(\sum\limits_{i=0}^{(a-1)/2}p^{i}(X^{2i}-X^{a+b+2-2i})+p^{(a+1)/2}\sum\limits_{i=a+1}^{b+1}(-X)^{k}\right),\\
    R_{1}(X) & = (1-p^{-2})\sum\limits_{i=1}^{(a+1)/2}p^{i}X^{2i}+p^{(a+1)/2}(1+p^{-1}X)\sum\limits_{i=a+2}^{b+1}(-X)^{k}.
\end{align*}
then 
\begin{align*}
    &\frac{\alpha_{p}(X,M,\delta_{p}(N))}{1-p^{-1}X}  = (pX^{-2}+X^{-1})\left(\sum\limits_{i=0}^{(a-1)/2}p^{i}(X^{2i}-X^{a+b+2-2i})+p^{(a+1)/2}\sum\limits_{i=a+1}^{b+1}(-X)^{k}\right)\\
    &-(pX^{-2}-p^{-1}X^{-2})\sum\limits_{i=1}^{(a+1)/2}p^{i}X^{2i}-p^{(a+3)/2}(1+p^{-1}X)\sum\limits_{i=a+2}^{b+1}(-X)^{k-2}-pX^{-2}-X^{-1}\\
    & = \sum\limits_{i=1}^{(a+1)/2}p^{i}X^{2i-4}+\sum\limits_{i=0}^{(a-1)/2}p^{i}X^{2i-1}-\sum\limits_{i=1}^{(a+1)/2}p^{i}X^{a+b+2-2i}-\sum\limits_{i=0}^{(a-1)/2}p^{i}X^{a+b+1-2i}\\
    &-\sum\limits_{i=2}^{(a+3)/2}p^{i}X^{2i-4}+\sum\limits_{i=0}^{(a-1)/2}p^{i}X^{2i}-pX^{-2}-X^{-1}+p^{(a+3)/2}X^{a-1}+p^{(a+1)/2}X^{a}\\& =\textup{Den}^{+}(X,M\obot\langle N \rangle).
\end{align*}
therefore we see that our computation result of $\textup{Den}^{+}_{\Delta(N)}(X,M)$ coincides with the computation in \cite{SSY22}, but our method is different: we embed $M$ and $\delta_{p}(N)$ into a higher rank self-dual quadratic lattice and do all the computations in the self-dual lattice. 
\end{example}

\begin{example}
We give examples for higher $n$ when $M\hookrightarrow\mathbb{W}$ is $\textup{GL}_{2}(\mathbb{Z}_{p})$-equivalent to the matrix $\textup{diag}\{\varepsilon_{1}p,\varepsilon_{2}p^{2}\}$ where $\varepsilon_{1},\varepsilon_{2}\in\mathbb{Z}_{p}^{\times}$. Let $N_{k}=p^{2k}N_{0}$ where $N_{0}$ is a positive integer with $\nu_{p}(N_{0})=0$ or $1$ and $k\geq1$ is an integer. By the formula in \cite[$\S$2.11]{Wed07}, $\textup{Den}^{+}(X,M\obot\langle N_{0}\rangle )=1-X^{3}$ when $\nu_{p}(N_{0})=0$, $\textup{Den}^{+}(X,M\obot\langle N_{0}\rangle )=1+pX-pX^{3}-X^{4}$ when $\nu_{p}(N_{0})=1$, and
\begin{equation*}
    \textup{Den}^{+}(X,M\obot\langle N_{k}\rangle )=
    1+pX+pX^{2}-pX^{2k+1+\nu_{p}(N_{0})}-pX^{2k+2+\nu_{p}(N_{0})}-X^{2k+3+\nu_{p}(N_{0})},\,\,\,\textup{when $k\geq1$.}
\end{equation*}
therefore we get
\begin{equation*}
    \textup{Den}^{+}_{\Delta(N_{k})}(X,M)=1+pX+(p-1)X^{2}-pX^{3}-pX^{4}.
\end{equation*}
when $k\geq1$, and $\textup{Den}^{+}_{\Delta(N_{0})}(X,M)=\textup{Den}^{+}(X,M\obot\langle N_{0}\rangle )=1-X^{3}$ or $1+pX-pX^{3}-X^{4}$ depending on $\nu_{p}(N_{0})=0$ or $1$.
\par
Therefore $\partial\textup{Den}^{+}_{\Delta(N_{0})}(M)=3$ or $2p+4$ depending on $\nu_{p}(N_{0})=0$ or $1$, and
\begin{equation*}
    \partial\textup{Den}^{+}_{\Delta(N_{k})}(M)=-\frac{\textup{d}}{\textup{d}X}\bigg|_{X=1}\textup{Den}^{+}_{\Delta(N_{k})}(X,M)=4p+2.
\end{equation*}
when $k\geq1$.
\par
We will double check our results by comparing it with the formulas of local density given in \cite[Theorem 7.1]{Yn98}. The theorem implies that for sufficiently large positive integer $m$,
\begin{equation*}
    \textup{Den}(\delta_{p}(N_{k})\obot H_{2m}^{+},M)=1+R_{1,k}(X)+R_{2,k}(X)\,\big\vert_{X=p^{-m}},
\end{equation*}
where $R_{1,k}(X)=\sum\limits_{i=1}^{8}I_{1,i,k}(X)$ and $R_{2,k}(X)=(1-p^{-1})\sum\limits_{i=1}^{8}I_{2,i,k}(X)+p^{-1}I_{2,6,k}(X)$. $I_{1,i,k}(X)$ and $I_{2,i,k}(X)$ are explicitly constructed polynomials in the beginning of section 7 of \cite{Yn98}. In our case, when $k\geq1$,
\begin{equation*}
   I_{1,1,k}(X)=(p-p^{-1})X,\,\, I_{1,2,k}(X)=-X^{2},\,\,  I_{1,3,k}(X)=0,\,\, I_{1,4,k}=-pX^{3}.
\end{equation*}
\begin{equation*}
    I_{2,1,k}(X)=I_{2,3,k}(X)=I_{2,5,k}(X)=I_{2,7,k}(X)=0.
\end{equation*}
\begin{equation*}
    I_{2,2,k}(X)=-X^{2},\,\,I_{2,4,k}(X)=-pX^{4},\,\,I_{2,6,k}(X)=X^{5},\,\,I_{2,8,k}(X)=pX^{2}.
\end{equation*}
therefore when $m$ is sufficiently large,
\begin{equation*}
    \textup{Den}(\delta_{p}(N_{k})\obot H_{2m}^{+},M)=1+(p-p^{-1})X+(p-2)X^{2}-(p+1-p^{-1})X^{3}+(1-p)X^{4}+X^{5}\,\big\vert_{X=p^{-m}}.
\end{equation*}
By Definition \ref{twisted}, when $k\geq1$,
\begin{equation*}
     \textup{Den}^{+}_{\Delta(N_{k})}(X,M)\,\big\vert_{X=p^{-m}}=\frac{\textup{Den}(\delta_{p}(N_{k})\obot H_{2m}^{+},M)}{1-p^{-m-1}}=1+pX+(p-1)X^{2}-pX^{3}-pX^{4}\,\big\vert_{X=p^{-m}}.
\end{equation*}
hence $\textup{Den}^{+}_{\Delta(N_{k})}(X,M)=1+pX+(p-1)X^{2}-pX^{3}-pX^{4}$ when $k\geq1$, this agrees with our previous calculations.
\end{example}

\begin{example}
We give another examples for higher $n$ when $M\hookrightarrow\mathbb{W}$ is $\textup{GL}_{2}(\mathbb{Z}_{p})$-equivalent to the matrix $\textup{diag}\{\varepsilon_{1}p^{2},\varepsilon_{2}p^{3}\}$ where $\varepsilon_{1},\varepsilon_{2}\in\mathbb{Z}_{p}^{\times}$. Let $N_{k}=p^{2k}N_{0}$ where $N_{0}$ is a positive integer with $\nu_{p}(N_{0})=0$ or $1$ and $k\geq1$ is an integer. By the formula in \cite[$\S$2.11]{Wed07},
\begin{align*}
    \textup{Den}^{+}(X,M\obot\langle N_{k}\rangle )=&
    1+pX+(p+p^{2})X^{2}+p^{2}X^{3}+p^{2}X^{4}-p^{2}X^{2k+1+\nu_{p}(N_{0})}-p^{2}X^{2k+2+\nu_{p}(N_{0})}\\&-(p+p^{2})X^{2k+3+\nu_{p}(N_{0})}-pX^{2k+4+\nu_{p}(N_{0})}-X^{2k+5+\nu_{p}(N_{0})},\,\,\,\textup{when $k\geq3$.}
\end{align*}
the formula (\ref{decder}) implies when $k\geq3$,
\begin{equation*}
    \textup{Den}^{+}_{\Delta(N_{k})}(X,M)=1+pX+(p^{2}+p-1)X^{2}+(p^{2}-p)X^{3}-pX^{4}-p^{2}X^{4}-p^{2}X^{5}-p^{2}X^{6}.
\end{equation*}
when $k\geq3$. Therefore $\partial\textup{Den}^{+}_{\Delta(N_{0})}(M)=2+4p+6p^{2}$ when $k\geq3$.
\par
We will double check our results by comparing it with the formulas of local density given in \cite[Theorem 7.1]{Yn98}. The theorem implies that for sufficiently large positive integer $m$,
\begin{equation*}
    \textup{Den}(\delta_{p}(N_{k})\obot H_{2m}^{+},M)=1+R_{1,k}(X)+R_{2,k}(X)\,\big\vert_{X=p^{-m}}.
\end{equation*}
where $R_{1,k}(X)=\sum\limits_{i=1}^{8}I_{1,i,k}(X)$ and $R_{2,k}(X)=(1-p^{-1})\sum\limits_{i=1}^{8}I_{2,i,k}(X)+p^{-1}I_{2,6,k}(X)$. $I_{1,i,k}(X)$ and $I_{2,i,k}(X)$ are explicitly constructed polynomials in the beginning of section 7 of \cite{Yn98}. In our case, when $k\geq3$,
\begin{equation*}
   I_{1,1,k}(X)=(p-p^{-1})X+(p^{2}-1)X^{2},\,\, I_{1,2,k}(X)=-X^{3},\,\,  I_{1,3,k}(X)=0,\,\, I_{1,4,k}=-p^{2}X^{4}.
\end{equation*}
\begin{equation*}
    I_{2,1,k}(X)=(p^{2}-1)X^{3},\,\,I_{2,3,k}(X)=I_{2,5,k}(X)=I_{2,7,k}(X)=0.
\end{equation*}
\begin{equation*}
    I_{2,2,k}(X)=-pX^{4}-pX^{5},\,\,I_{2,4,k}(X)=-p^{2}X^{5}-p^{2}X^{6},\,\,I_{2,6,k}(X)=pX^{7},\,\,I_{2,8,k}(X)=pX^{2}+p^{2}X^{4}.
\end{equation*}
therefore when $m$ is sufficiently large,
\begin{align*}
    \textup{Den}(\delta_{p}(N_{k})\obot H_{2m}^{+},M)&=1+(p-p^{-1})X+(p^{2}+p-2)X^{2}+(p^{2}-2p+p^{-1}-1)X^{3}\\&-(2p-1)X^{4}+(1-p^{2})X^{5}+(p-p^{2})X^{6}+pX^{7}\,\big\vert_{X=p^{-m}}.
\end{align*}
By Definition \ref{twisted}, when $k\geq3$,
\begin{align*}
     &\textup{Den}^{+}_{\Delta(N_{k})}(X,M)\,\big\vert_{X=p^{-m}}=\frac{\textup{Den}(\delta_{p}(N_{k})\obot H_{2m}^{+},M)}{1-p^{-m-1}}\\&=1+pX+(p^{2}+p-1)X^{2}+(p^{2}-p)X^{3}-pX^{4}-p^{2}X^{4}-p^{2}X^{5}-p^{2}X^{6}\,\big\vert_{X=p^{-m}}.
\end{align*}
hence $\textup{Den}^{+}_{\Delta(N_{k})}(X,M)=1+pX+(p^{2}+p-1)X^{2}+(p^{2}-p)X^{3}-pX^{4}-p^{2}X^{4}-p^{2}X^{5}-p^{2}X^{6}$ when $k\geq3$, this agrees with our previous calculations.
\end{example}

\section{Proof of Arithmetic Siegel-Weil formula on $\mathcal{X}_{0}(N)$}
\label{8}
\subsection{Local arithmetic Siegel-Weil formula with level $N$}
Let $p$ be a prime number. The difference formulas at the analytic side and the geometric side are combined together to prove the following theorem.
\begin{theorem}
Let $M\subset\mathbb{W}$ be a $\mathbb{Z}_{p}$-lattice of rank 2. Then
\begin{equation}
    \textup{Int}_{\mathcal{N}_{0}(N)}(M) = \partial\textup{Den}^{+}_{\Delta(N)}(M).
    \label{tws}
\end{equation}
\label{INT=DER}
\end{theorem}
\begin{proof}
Theorem \ref{geodiff} gives the following difference formula of local arithmetic intersection numbers,
\begin{equation*}
    \textup{Int}_{\mathcal{N}_{0}(N)}(M) = 
    \textup{Int}^{\sharp}(M\obot\mathbb{Z}_{p}\cdot x_{0})-\textup{Int}^{\sharp}(M\obot\mathbb{Z}_{p}\cdot p^{-1}x_{0}).
\end{equation*}
We also have the difference formula of the derived local densities (cf.(\ref{diffana})),
\begin{equation*}
    \partial\textup{Den}^{+}_{\Delta(N)}(M)= \partial\textup{Den}^{+}(M\obot\mathbb{Z}_{p}\cdot x_{0})-\partial\textup{Den}^{+}(M\obot\mathbb{Z}_{p}\cdot p^{-1}x_{0}).
\end{equation*}
Theorem \ref{mainlz22} implies that $\textup{Int}^{\sharp}(L) = \partial\textup{Den}^{+}(L)$ for any rank 3 lattice $L\subset\mathbb{B}$. Therefore $\linebreak\textup{Int}_{\mathcal{N}_{0}(N)}(M)=\partial\textup{Den}^{+}_{\Delta(N)}(M)$ holds by combining the above two difference formulas.
\end{proof}

\subsection{Intersection Numbers and Whittaker functions} Let $p$ be a prime number.
\begin{proposition}
Let $M\subset\mathbb{W}$ be a $\mathbb{Z}_{p}$-lattice of rank 2. Then
\begin{equation}
    W_{T}^{\prime}(1,0,1_{\delta_{p}(N)^{2}}) = c_{p}\cdot\textup{Int}_{\mathcal{N}_{0}(N)}(M)\cdot\textup{log}(p) 
    \label{wht}
\end{equation}
where the constant $c_{p}$ is given as follows
\begin{equation*}
    c_{p} =\begin{cases}
    (1-p^{-1})\cdot (N,-1)_{p}\cdot\vert N\vert_{p}\cdot\vert 2\vert_{p}^{3/2} & \textup{when $p\,\vert\, N$;}\\
    (1-p^{-2})\cdot (N,-1)_{p}\cdot\vert N\vert_{p}\cdot\vert 2\vert_{p}^{3/2} & \textup{when $p\nmid N$.}
    \end{cases}
\end{equation*}
\label{loccc}
\end{proposition}
\begin{proof}
Recall that $\delta_{p}(N)^{\vee}/\delta_{p}(N)\simeq\mathbb{Z}_{p}/2N\mathbb{Z}_{p}$ (cf. Example \ref{lambda}). By Proposition \ref{non-homo} and the explicit formula given by Rao in the appendix of \cite{Rao93},
\begin{align}
    W_{T}(1,k,1_{\delta_{p}(N)^{2}}) & = \vert 2N\vert_{p}\cdot\gamma(\delta_{p}(N)\otimes\mathbb{Q}_{p})^{2}\cdot\vert 2\vert_{p}^{1/2}\cdot \textup{Den}(\delta_{p}(N)\oplus H_{2k}^{+},M)\notag\\
    & = \vert N\vert_{p}\cdot(N,-1)_{p}\cdot\vert 2\vert_{p}^{3/2}\cdot \textup{Den}(\delta_{p}(N)\oplus H_{2k}^{+},M).\label{wder}
\end{align}
Taking derivatives of both sides of (\ref{wder}),
\begin{equation*}
    W_{T}^{\prime}(1,0,1_{\delta_{p}(N)^{2}}) = c_{p}\cdot\partial\textup{Den}^{+}_{\Delta(N)}(M)\cdot\textup{log}(p).
\end{equation*}
The formula (\ref{wht}) follows from Theorem \ref{INT=DER}.
\end{proof}

\subsection{Proof of the main theorem}
\begin{proposition}
Let $T\in \textup{Sym}_{2}(\mathbb{Q})$ be a positive definite symmetric matrix. Let $\boldsymbol{\varphi}\in\mathscr{S}(\mathbb{V}_{f}^{2})$ be a $T$-admissible Schwartz function. Suppose $\boldsymbol\varphi=\varphi_{1}\times\varphi_{2}$ where $\varphi_{i}\in\mathscr{S}(\mathbb{V}_{f})$, then for any $\mathsf{y}\in\textup{Sym}_{2}(\mathbb{R})_{>0}$, we have
\begin{equation*}
    \widehat{\textup{deg}}(\widehat{\mathcal{Z}}(T,\mathsf{y},\boldsymbol\varphi)) = \begin{cases}
    \chi(\mathcal{Z}(T,\boldsymbol{\varphi}), \mathcal{O}_{\mathcal{Z}(t_{1},\varphi_{1})}\otimes^{\mathbb{L}}\mathcal{O}_{\mathcal{Z}(t_{2},\varphi_{2})})\cdot\textup{log}(p), & \textup{when $\text{Diff}(T,\Delta(N))=\{p\}$;}\\
    0, & \textup{when $\#\text{Diff}(T,\Delta(N))\neq1$.}
    \end{cases}
\end{equation*}
\label{intdeg}
\end{proposition}
\begin{proof}
By definition (cf. (\ref{codim2})), the arithmetic special cycle $\widehat{\mathcal{Z}}(T,\mathsf{y},\boldsymbol\varphi)=([\mathcal{Z}(T,\boldsymbol{\varphi})],0)$, therefore $\widehat{\textup{deg}}(\widehat{\mathcal{Z}}(T,\mathsf{y},\boldsymbol\varphi))$ is independent of $\mathsf{y}$. We can assume $\textup{Diff}(T,\Delta(N)) = \{p\}$ for some prime number $p$ since otherwise both sides are 0 since $\mathcal{Z}(T,\boldsymbol{\varphi})$ would be an empty stack.
\par
Let $x\in\mathcal{Z}(T,\boldsymbol{\varphi})(\overline{\mathbb{F}}_{p})$ be a geometric point, it is contained in $\mathcal{Y}_{0}(N)$ by Corollary \ref{nocusp2}, hence the special divisors $\mathcal{Z}(t_{1},\varphi_{1})$ and $\mathcal{Z}(t_{2},\varphi_{2})$ intersect properly at $x$ because $T$ is nonsingular. Then $\chi(\mathcal{Z}(T,\boldsymbol{\varphi}), \mathcal{O}_{\mathcal{Z}(t_{1},\varphi_{1})}\otimes^{\mathbb{L}}\mathcal{O}_{\mathcal{Z}(t_{2},\varphi_{2})})\cdot\textup{log}(p)$ is the sum of the length of local rings $\mathcal{O}_{\mathcal{X}_{0}(N),x}$ cut out by these two divisors times $\textup{log}\,(p)$, which is exactly $\widehat{\textup{deg}}(\widehat{\mathcal{Z}}(T,\mathsf{y},\boldsymbol\varphi))$ by the definition of the degree homomorphism.
\end{proof}
\noindent \textit{Proof of Theorem \ref{main}}: 
We first consider the case that $T$ is positive definite. By Proposition \ref{onlyone}, we only need to consider the case $\textup{Diff}(T,\Delta(N))=\{p\}$ for some prime number $p$ because otherwise both sides are 0. The same proposition and Corollary \ref{nocusp2} imply that the special cycle $\mathcal{Z}(T,\boldsymbol{\varphi})$ lies in the supersingular locus of $\mathcal{X}_{0}(N)_{\mathbb{F}_{p}}$. Then by the definition of special cycles and the formal uniformization of the special cycle $\mathcal{Z}(T,\boldsymbol{\varphi})$ (cf. Corollary \ref{puni}),
\begin{align*}
   \chi(\mathcal{Z}(T,\boldsymbol{\varphi}), \mathcal{O}_{\mathcal{Z}(t_{1},\varphi_{1})}\otimes^{\mathbb{L}}\mathcal{O}_{\mathcal{Z}(t_{2},\varphi_{2})})\cdot\textup{log}(p) =&\\ \sum\limits_{\substack{\boldsymbol{x}\in B^{\times}(\mathbb{Q})_{0}\backslash (\Delta(N)^{(p)})^{2}\\ T(\boldsymbol{x})= T}}\sum\limits_{g\in B^{\times}_{\boldsymbol{x}}(\mathbb{Q})_{0}\backslash \textup{GL}_{2}(\mathbb{A}_{f}^{p})/\Gamma_{0}(N)(\hat{\mathbb{Z}}^{p})}&\boldsymbol\varphi(g^{-1}\boldsymbol{x})\cdot \textup{Int}_{\mathcal{N}_{0}(N)}(\boldsymbol{x})\cdot\textup{log}(p).
\end{align*}
we already know that (cf. (\ref{wht}))
\begin{equation*}
    W_{T}^{\prime}(1,0,1_{\delta_{p}(N)^{2}}) = c_{p}\cdot \textup{Int}_{\mathcal{N}_{0}(N)}(\boldsymbol{x})\cdot\textup{log}(p).
\end{equation*}
with constants $c_{p}$ given by (\ref{loccc}).
\par
There exists a Haar measure on $\textup{GL}_{2}(\mathbb{A}_{f}^{p})$ such that 
\begin{equation*}
    \sum\limits_{\substack{\boldsymbol{x}\in B^{\times}(\mathbb{Q})_{0}\backslash (\Delta(N)^{(p)})^{2}\\ T(\boldsymbol{x})= T}}\sum\limits_{g\in B^{\times}_{\boldsymbol{x}}(\mathbb{Q})_{0}\backslash \textup{GL}_{2}(\mathbb{A}_{f}^{p})/\Gamma_{0}(N)(\hat{\mathbb{Z}}^{p})}\boldsymbol\varphi(g^{-1}\boldsymbol{x}) = \frac{1}{\textup{vol}(\Gamma_{0}(N)(\hat{\mathbb{Z}}^{p}))}\cdot\int_{\textup{SO}(\Delta(N)^{(p)})(\mathbb{A}_{f}^{p})}\boldsymbol\varphi^{p}(g^{-1}\boldsymbol{x})\textup{d}g.
\end{equation*}
By definition, the last integral is a product of ``local" integrals
\begin{align*}
    \int_{\textup{SO}(\Delta(N)^{(p)})(\mathbb{A}_{f}^{p})}\boldsymbol\varphi^{p}(g^{-1}\boldsymbol{x})\textup{d}g = \prod\limits_{v\neq p,\infty}\int_{\textup{SO}(\delta_{v}(N))(\mathbb{Q}_{v})}\boldsymbol\varphi_{v}(g_{v}^{-1}\boldsymbol{x})\textup{d}g_{v}.
\end{align*}
By the classical local Siegel-Weil formula which is made explicit in the work of Kudla, Rapoport and Yang \cite[Proposition 5.3.3]{KRY06}, for every place $v$ of $\mathbb{Q}$, there exists a number $d_{v}\in\mathbb{R}^{\times}$ such that
\begin{equation*}
    \int_{\textup{SO}(\delta_{v}(N))(\mathbb{Q}_{v})}\boldsymbol\varphi_{v}(g_{v}^{-1}\boldsymbol{x})\textup{d}g_{v} = d_{v}\cdot W_{T,v}(1, 0, \boldsymbol\varphi_{v}),
\end{equation*}
with $\prod\limits_{v\leq\infty}d_{v}=1$. Moreover, \cite[Lemma 5.3.9]{KRY06} implies the following,
\begin{equation*}
    \textup{vol}(\Gamma_{0}(N)_{v},dg_{v}) = d_{v}\cdot\gamma(\delta_{v}(N))^{2}\cdot\vert2\vert_{v}^{3/2}\cdot\begin{cases}
    (1-v^{-2}), &\textup{when $v\nmid N$};\\
    \vert N\vert_{v}^{-1}(1+v^{-1}) &\textup{when $v\vert N$}.
    \end{cases}
\end{equation*}
then it can be checked immediately that
\begin{equation*}
     \textup{vol}(\Gamma_{0}(N)(\hat{\mathbb{Z}}^{p})) \cdot d_{p}d_{\infty}\cdot c_{p} = 2^{-1/2}\psi(N)^{-1}\cdot\frac{3}{\pi^{2}}.
\end{equation*}
Suppose $\mathsf{z}=\mathsf{x}+i\mathsf{y}$, it's a classical result that 
\begin{equation*}
    W_{T,\infty}(g_{\mathsf{z}},0,\Phi_{\infty}^{3/2})=-2^{7/2}\pi^{2}\cdot\textup{det}(\mathsf{y})^{3/4}q^{T}.
\end{equation*}
Combining these together with the definitions made in previous sections (cf. (\ref{dewhit}) and (\ref{classi})) and Proposition \ref{intdeg}, we get the formula stated in the theorem.
\par
When $T$ is not positive definite, the equality follows from \cite[$\S$4.2]{SSY22} and our computations of the volume of $\textup{vol}(\Gamma_{0}(N)(\hat{\mathbb{Z}}))=\prod\limits_{v<\infty}\textup{vol}(\Gamma_{0}(N)_{v},dg_{v})$ above.
$\hfill\qed$
\label{proof}

\end{document}